\documentclass[a4paper, leqno]{article}

\title{\textbf{Heights and metrics with logarithmic singularities}}
\author{\textbf{Gerard Freixas i Montplet}}
\date{}

\usepackage{amsmath, mathrsfs, amssymb, amsthm, amscd}
\usepackage[]{fontenc}
\usepackage{xy}
\usepackage{enumerate}
\xyoption{all}

\numberwithin{equation}{section}

\theoremstyle{plain}
\newtheorem{theorem}{Theorem}[section]
\newtheorem{corollary}[theorem]{Corollary}
\newtheorem{proposition}[theorem]{Proposition}
\newtheorem{lemma}[theorem]{Lemma}

\theoremstyle{definition}
\newtheorem{definition}[theorem]{Definition}
\newtheorem{notation}[theorem]{Notation}
\newtheorem{example}[theorem]{Example}
\newtheorem{remark}[theorem]{Remark}

\DeclareMathOperator{\Spec}{Spec}

\DeclareMathOperator{\CC}{\mathcal{C}}
\DeclareMathOperator{\Z}{Z}

\DeclareMathOperator{\HH}{H}
\DeclareMathOperator{\supp}{supp}
\DeclareMathOperator{\gdiv}{div}

\DeclareMathOperator{\pd}{\partial}
\DeclareMathOperator{\cpd}{\bar{\partial}}
\DeclareMathOperator{\pre}{pre}
\DeclareMathOperator{\red}{red}
\DeclareMathOperator{\loc}{loc}

\DeclareMathOperator{\c1}{c_{1}}
\DeclareMathOperator{\adeg}{\widehat{deg}}
\DeclareMathOperator{\gdeg}{deg}
\DeclareMathOperator{\Norm}{N}
\DeclareMathOperator{\Real}{Re}
\DeclareMathOperator{\Ric}{Ric}

\nopagebreak
\newcommand{\dd}{d}
\newcommand{\dc}{d^{c}}
\newcommand{\ddc}{dd^{c}}
\newcommand{\OO}{\mathcal{O}}

\newcommand{\X}{\mathscr{X}}
\newcommand{\Y}{\mathscr{Y}}
\newcommand{\BS}{\mathscr{S}}
\newcommand{\BT}{\mathscr{T}}

\newcommand{\C}{\mathbb{C}}
\newcommand{\R}{\mathbb{R}}
\newcommand{\Int}{\mathbb{Z}}

\newcommand{\Q}{\mathbb{Q}}
\newcommand{\KK}{\overline{K}}
\newcommand{\cz}{\overline{z}}

\newcommand{\AL}{\mathscr{L}}
\newcommand{\HAL}{\overline{\mathscr{L}}}

\newcommand{\PP}{\mathbb{P}}
\newcommand{\Af}{\mathbb{A}}

\newcommand{\cdz}{d\bar{z}}
\newcommand{\cdzeta}{d\bar{\zeta}}

\newcommand{\M}{\mathcal{M}}
\newcommand{\CM}{\overline{\mathcal{M}}}
\newcommand{\UC}{\mathcal{C}}
\newcommand{\UCC}{\overline{\mathcal{C}}}

\begin{document}
\setcounter{tocdepth}{1}
\setcounter{section}{0}
\maketitle

\begin{abstract}
We prove lower bound and finiteness properties for arakelovian heights with respect to pre-log-log hermitian ample line bundles. These heights were introduced by Burgos, Kramer and K\"uhn in \cite{BKK}, in their extension of the arithmetic intersection theory of Gillet and Soul\'e \cite{GS}, aimed to deal with hermitian vector bundles equipped with metrics admitting suitable logarithmic singularities. Our results generalize the corresponding properties for the heights of cycles in Bost-Gillet-Soul\'e \cite{BGS}, as well as the properties established by Faltings \cite{Faltings} for heights of points attached to hermitian line bundles whose metrics have logarithmic singularities. We also discuss various geometric constructions where such pre-log-log hermitian ample line bundles naturally arise.
\end{abstract}
\tableofcontents
\subsubsection*{Conventions}
We fix some conventions and notations to be followed throughout this paper.

The open disc of $\C$ centered at 0 and of radius $\varepsilon>0$ will be denoted by $\Delta_{\varepsilon}$.

If $f,g:E\rightarrow\R$ are two functions on a set $E\neq\emptyset$, we write $f\prec g$ to mean that there exists a constant $C$ such that $f(x)\leq C g(x)$ for all $x\in E$. If the involved constant depends on some data $D$ that we want to specify, we write $f\prec_{D} g$.

If $X$ is a complex analytic manifold we decompose the exterior differential operator as $\dd=\pd+\cpd$ and we define $\dc=(4\pi i)^{-1}(\pd-\cpd)$, so that $\ddc=i\pd\cpd/2\pi$.

Let $k$ be a field. By an algebraic variety $X$ over $k$ we mean a separated and reduced scheme of finite type over $k$. In particular, $X$ is noetherian and it has a finite number of irreducible components. When $k=\C$, for every separated scheme of finite type $X$ over $\C$ there is an associated complex analytic space $X^{\text{an}}$, whose underlying topological space equals the set of complex points $X(\C)$. If $F$ is a closed subscheme of $X$, then $F^{\text{an}}$ is an analytic subspace of $X^{\text{an}}$. The scheme $X$ is proper over $\C$ if, and only if, $X^{\text{an}}$ is compact. Also, $X$ is a nonsingular variety over $\C$ if, and only if, $X^{\text{an}}$ is a complex analytic manifold. To simplify notations we will write $X$ instead of $X^{\text{an}}$ or $X(\C)$ (it will be clear from the context the category we are working on).

Let $K$ be a number field. Its ring of integers is denoted by $\OO_{K}$. Let $\BS=\Spec\OO_K$. An arithmetic variety over $\BS$ will be a flat and projective scheme $\pi:\X\rightarrow\BS$, with regular generic fiber $\X_K=\X\times_{\BS}\Spec K$ of pure dimension $n$. The set of complex points $\X(\C)$ has a natural structure of complex analytic manifold, and it can be partitioned as
\begin{displaymath}
	\X(\C)=\coprod_{\sigma:K\hookrightarrow\C}\X_{\sigma}(\C).
\end{displaymath}
The complex conjugation induces an antiholomorphic involution $F_{\infty}:\X(\C)\rightarrow\X(\C)$.

\section{Introduction}
In this paper we establish a common generalization of the following two statements.
\begin{theorem}[Faltings \cite{Faltings}, Lemma 3]\label{Lemma_Faltings}
Let $\X$ be a projective arithmetic variety and $\Y\subseteq \X$ a Zariski closed subset. Let $\HAL=(\AL,\|\cdot\|)$ be an ample line bundle on $\X$ endowed with a smooth hermitian metric on $\AL\mid_{\X(\C)\setminus \Y(\C)}$. Suppose that $\|\cdot\|$ has logarithmic singularities along $\Y(\C)$. Fix a number field $K$. For any real constant $C$, there are only finitely many points $P\in \X(K)\setminus \Y(K)$ with $h_{\HAL}(P)\leq C$.
 \end{theorem}
 The notions of function and metric with logarithmic singularities are recalled in \textsection 2 and \textsection 3 below.
 \begin{theorem}[Bost-Gillet-Soul\'e \cite{BGS}, Proposition 3.2.4 and Theorem 3.2.5]\label{thm_BGS}
 Let $\X$ be a projective arithmetic variety and $\HAL$ a smooth hermitian ample line bundle on $\X$. 
\begin{itemize}
 	\item[1.] For any real constant $A$, there are only finitely many effective cycles $z\in\Z^{p}(\X)$ such that $\deg_{\AL_{K}}z\leq A$ and $h_{\HAL}(z)\leq A$.
	\item[2.] There exists a positive constant $\kappa$ such that $h_{\HAL}(z)\geq -\kappa\deg_{\AL_{K}}z$ for every effective cycle $z\in\Z^{p}(\X)$.
\end{itemize}
\end{theorem}
Let us place under the hypothesis of Theorem \ref{Lemma_Faltings}. Let $P\in\X(K)\setminus\Y(K)$, extended to $\overline{P}:\Spec\OO_{K}\rightarrow\X$ by properness of $\X$. Since $\AL$ is ample, there exists some positive power $\AL^{\otimes N}$ admitting a global section $s$ non-vanishing at $P$. Then
\begin{displaymath}
	Nh_{\HAL}(P)=\log\sharp\left(\frac{\overline{P}^{*}(\AL^{\otimes N})}{\overline{P}^{*}(s\OO_{K})}\right)-\sum_{\sigma:K\hookrightarrow\C}\log\|s\|_{\sigma}(P_{\sigma}),
\end{displaymath}
where $P_{\sigma}$ is the point in $\X_{\sigma}(\C)$ induced by $P$. Therefore, in Faltings' result, only the metric as a function is required. The definition of the height of a cycle of positive relative dimension, with respect to a smooth hermitian line bundle, involves the derivatives of the metric up to second order (see \cite{BGS} \textsection 3 and also \textsection \ref{section_heights} below). Therefore, to extend both theorems, we need to describe the kind of logarithmic singularities allowed to the derivatives of the metric, up to second order. The arithmetic intersection theory of Gillet and Soul\'e needs to be reinterpreted so we can deal with such metrics. This has been done in \cite{BKK}, where Burgos, Kramer and K\"uhn develop a theory of abstract arithmetic Chow groups, and apply it to the case of logarithmic singularities.

Before the statement of our main theorem we introduce some notations. Let $K$ be a number field, $\OO_K$ its ring of integers and $\X$ a projective arithmetic variety over $\Spec\OO_K$. Suppose that $D\subseteq\X_K$ is a divisor such that $D(\C)$ has normal crossings. Write $U=\X(\C)\setminus D(\C)$. By $\Z^{p}_{U}(\X)$ we denote the group of codimension $p$ cycles $z$ on $\X$, such that $z(\C)$ intersects $D(\C)$ properly. A cycle on $\X$ is called horizontal if it is the Zariski closure of a cycle on $\X_K$. We write $\Z^{p}_{U}(\X_K)$ for the subgroup of $\Z^{p}_{U}(\X)$ consisting of horizontal cycles. Now let $\AL$ be a line bundle on $\X$, endowed with a pre-log-log hermitian metric $\|\cdot\|$, with singularities along $D$ (see \textsection \ref{section_vector_bundles} below for the definition). According to Burgos, Kramer and K\"uhn, this is the notion of metric with logarithmic singularities well suited to define heights on $\Z^{p}_{U}(\X)$ (see \cite{BKK} \textsection 7). If $\AL_{K}$ is ample there is a well defined normalized height $\widetilde{h}_{\HAL}$ on $\Z^{p}_{U}(\X_K)$ with respect to $\HAL=(\AL,\|\cdot\|)$. We refer to \textsection \ref{section_heights} for a summary of the theory of heights in Arakelov theory.
\begin{theorem}[Main Theorem]\label{main_theorem}
Let $\X$ be a projective arithmetic variety over $\Spec\OO_K$, $D\subseteq\X_K$ a reduced divisor such that $D(\C)\subseteq\X(\C)$ has simple normal crossings and $U=\X(\C)\setminus D(\C)$. Let $\AL$ be a line bundle on $\X$ with $\AL_K$ ample. Let $\|\cdot\|$ be a pre-log-log hermitian metric on $\AL$, with singularities along $D$, and $\|\cdot\|_{0}$ a smooth metric on $\AL$. Then there exist constants $\alpha$, $\beta$, $\gamma>0$ and $R\in\Int_{\geq 0}$ such that for every effective cycle $z\in\Z^{p}_{U}(\X_K)$ we have $\widetilde{h}_{\HAL_{0}}(z)+\gamma\geq 1$ and
\begin{equation}\label{main_ineq}
	\left|\widetilde{h}_{\HAL}(z)-\widetilde{h}_{\HAL_0}(z)\right|\leq\alpha+\beta\log^{R}\left(\widetilde{h}_{\HAL_{0}}(z)+\gamma\right).
\end{equation}
If moreover $\|\cdot\|$ is good along $D$, then we can take $R=1$.
\end{theorem}
In conjunction with Theorem \ref{thm_BGS}, Theorem \ref{main_theorem} immediately yields the desired finiteness property as well as the existence of a universal lower bound.
\begin{corollary}\label{Corollary_Main}
Let $\X$ be a projective arithmetic variety over $\Spec\OO_K$, $D\subseteq\X_{K}$ a reduced divisor such that $D(\C)$ has simple normal crossings and $U=\X(\C)\setminus D(\C)$. Let $\HAL$ be a pre-log-log hermitian ample line bundle on $\X$, with singularities along $D$.
\begin{itemize}
 	\item[1.] For any real constant $A$, there are only finitely many effective cycles $z\in\Z^{p}_{U}(\X)$ such that $\deg_{\AL_{K}}z\leq A$ and $h_{\HAL}(z)\leq A$.
	\item[2.] There exists a positive constant $\kappa$ such that $h_{\HAL}(z)\geq -\kappa\deg_{\AL_{K}}z$ for every effective cycle $z\in\Z^{p}_{U}(\X)$.
\end{itemize}
\end{corollary}
The techniques employed for the proof of the main theorem were initially inspired by the work of Carlson and Griffiths on the defect relation for equidimensional holomorphic maps, in higher dimensional Nevanlinna theory \cite{CarGri}. Moreover (\ref{main_ineq}) may be interpreted as a vast generalization of the naive Liouville's inequality for the distance between an algebraic number and a rational number. This point of view is conceptually interesting, since it opens the natural question of finding an analogue to Roth's theorem for the distance\footnote{Suitable candidates for the logarithm of the distance are provided by the \textit{height integrals} introduced in \textsection \ref{section_bounding_height_integrals} below.} between a divisor with simple normal crossings and effective cycles of arbitrary dimension (all defined over a number field)\footnote{The height morphism $h_{\HAL}$ is defined for any pre-log-log hermitian line bundle $\HAL$, with singularities along a divisor in $\X(\C)$. However, for the main theorem to hold, the divisor needs to be defined over a number field. This essentially goes back to the construction of transcendental numbers due to Liouville.}.

This paper is organized as follows.

In Section \ref{sec_diff_form_log} we review the theory of differential forms with logarithmic singularities necessary in the rest of the paper. In Section \ref{section_vector_bundles} we study in detail several notions of logarithmically singular hermitian vector bundles. The results we recall provide a wealth of constructions to which Theorem \ref{Lemma_Faltings} and Theorem \ref{main_theorem} apply. Both sections are complemented with examples for a better understanding of the theory. In Section \ref{section_decomp_thm} we establish global bounds for real log-log growth (1,1)-forms. An outstanding consequence is a decomposition theorem (Theorem \ref{decomp_thm_2}) for pre-log-log functions, which plays a crucial role in the proof of the main theorem. In Section \ref{section_bounding_height_integrals} we prove estimates for integrals of pre-log-log forms appearing in the archimedian part of the definition of height. This leads to the proof of the main theorem in Section \ref{section_heights}, where the reader will find an overview of the theory of heights in Arakelov geometry. In Section \ref{section_examples} we present some examples of good hermitian line bundles interesting for arithmetic purposes (for instance in relation with Theorem \ref{main_theorem}). We treat the case of fully decomposed automorphic vector bundles on locally symmetric varieties, the relative dualizing sheaf of the universal curve over the moduli space of stable curves, equipped with the family hyperbolic metric, the Weil-Petersson metric on the moduli space of curves and the K\"ahler-Einstein metric on quasi-projective varieties. Finally, in the Appendix we prove a Bertini's type theorem needed for the preliminary reductions in the proof of the main theorem.

\textit{Acknowledgements.} I am deeply indebted to J.-B. Bost and J. I. Burgos Gil for proposing me to work on this subject as the starting point of my PhD. thesis, and for their guidance and constant encouragement. During the preparation of this paper I benefited from stimulating discussions with several people: I warmly thank S. Fischler, U. K\"uhn, V. Maillot, C. Mourougane and D. Roessler for the time they devoted to me. Their useful advice is reflected throughout the paper.

This work was presented, in a preliminary form, at the \textit{Internationales Graduiertenkolleg Arithmetic and Geometry} of the Humboldt-Universit\"at zu Berlin and the ETH of Z\"urich, in May 2005. I am grateful to J. Kramer for inviting me to talk to their summer school.

\section{Differential forms with logarithmic singularities}\label{sec_diff_form_log}
Let $X$ be a complex analytic manifold and $F$ a closed analytic subspace. In this section we introduce several notions of differential forms on $X$ with logarithmic singularities along $F$, relevant to our work. We distinguish the case of functions from the case of differential forms, since the former can be presented in a more general geometric frame. Indeed, while we can define functions with logarithmic singularities along an arbitrary closed analytic subspace $F$, the appropriate analogues for differential forms of any order require $F$ to be a divisor with normal crossings.
\subsection{Functions with logarithmic singularities}
Let $X$ be a complex analytic manifold and $F$ a closed analytic subspace of $X$. We denote by $\mathcal{I}_{F}\subseteq\OO_{X}$ the ideal sheaf defining $F$ and $\supp(F)$ for the support of $F$. For every $x\in X$, there exist an open neighborhood $V$ and holomorphic functions $s_{1},\ldots,s_{m}\in\OO_{X}(V)$ such that
\begin{itemize}
\item[\textit{i}.] the germ of ideal sheaf $\mathcal{I}_{F,x}$ is generated by $s_{1},\ldots,s_{m}$;
\item[\textit{ii}.] the trace of the support of $F$ on $V$ is $\supp(F)\cap V=\lbrace z\in V\mid s_{1}(z)=\ldots=s_{m}(z)=0\rbrace.$
\end{itemize}
Since $\OO_{X}$ is a coherent sheaf, so is $\mathcal{I}_{F}$. Then, given $s_{1},\ldots,s_{m}$ as above, after possibly shrinking $V$, $s_{1},\ldots,s_{m}$ generate all the germs $\mathcal{I}_{F,z}$ for $z\in V$. In this case we say that $s_{1},\ldots,s_{m}$ generate $\mathcal{I}_{F\mid V}$ and we write $\mathcal{I}_{F\mid V}=(s_{1},\ldots,s_{m})$ as a shortcut. The reader is referred to \cite{Demailly} for further details on analytic spaces.
\begin{definition}\label{def_func_log_sing}
Let $X$ be a complex analytic manifold and $F\subseteq X$ a closed analytic subspace. A smooth function $f:X\setminus \supp(F)\rightarrow\C$ has \textit{logarithmic singularities along} $F$ if for every open subset $V$ of $X$ such that $\mathcal{I}_{F\mid V}=(s_1,\ldots,s_m)$ and every relatively compact open subset $V^{\prime}\subset\subset V$, there exists an integer $N\geq 0$ such that
\begin{equation}\label{eq1}
	|f_{\mid_{V^{\prime}\setminus \supp(F)}}|\prec\left|\log(\max_{i=1,\ldots,m}|s_i|)^{-1}\right|^{N}.
\end{equation}
\end{definition}
For a function $f:X\setminus \supp(F)\rightarrow\C$ to be with logarithmic singularities along $F$ it is enough that (\ref{eq1}) be satisfied for a given open covering of $X$ and given local generators of $\mathcal{I}_{F}$. This is the content of the next lemma.
\begin{lemma}
Let $X$ be a complex analytic manifold and $F$ a closed analytic subspace. Fix an open covering $\lbrace V_{\alpha}\rbrace_{\alpha}$ of $X$ such that $\mathcal{I}_{F\mid V_{\alpha}}$ is generated by $s_{1}^{\alpha},\ldots, s_{m_{\alpha}}^{\alpha}$. Suppose given $V_{\alpha}^{\prime}\subset\subset V_{\alpha}$ still forming an open covering of $X$. Then a smooth function $f:X\setminus \supp(F)\rightarrow\C$ has logarithmic singularities along $F$ if, and only if, for every $\alpha$ there exists an integer $N\geq 0$ such that
\begin{displaymath}
	|f_{\mid_{V^{\prime}_{\alpha}\setminus \supp(F)}}|\prec\left|\log(\max_{i=1,\ldots, m_{\alpha}}|s_i^{\alpha}|)^{-1}
	\right|^{N}.
\end{displaymath}
\end{lemma}
\begin{proof}
Left as an easy exercise.
\end{proof}
\begin{lemma}
Let $X$ be a complex analytic manifold and $F$, $G$ closed analytic subspaces with $\supp(F)=\supp(G)$.
A smooth function $f:X\setminus\supp(F)\rightarrow\C$ has logarithmic singularities along $F$ if, and only if, it has logarithmic singularities along $G$.
\end{lemma}
\begin{proof}
Write $\mathcal{I}$, $\mathcal{J}$ for the ideal sheaves of $F$, $G$, respectively. Let $V$ be an open neighborhood of $x\in X$ such that $\mathcal{I}_{\mid V}=(s_{1},\ldots, s_{l})$ and $\mathcal{J}_{\mid V}=(t_{1},\ldots,t_{m})$. By Hilbert's Nullstellensatz (see (4.22) in \cite{Demailly}), after possibly shrinking $V$, there exists an integer $N\geq 0$ such that
\begin{displaymath}
	s_{i}^{N}=\sum_{j=1}^{m}\lambda_{i}^{j}t_{j}, t_{j}^{N}=\sum_{i=1}^{l}\mu_{j}^{i}s_{i}.
\end{displaymath}
for some holomorphic functions $\lambda_{i}^{j}$, $\mu_{j}^{i}$. Hence, if $V^{\prime}\subseteq V$ is a relatively compact open subset, there exists a constant $C>0$ such that
\begin{displaymath}
	|s_{i}|^{N}\leq C\max_{j=1,\ldots,m}|t_j|, |t_{j}|^{N}\leq C\max_{i=1,\ldots,l}|s_i|.
\end{displaymath}
The lemma follows.
\end{proof}
\noindent The meaning of the lemma is that the notion of function with logarithmic singularities along $F$ depends only on the support  of $F$.
\begin{proposition}\label{prop_prop_log_sing}
i. Let $X$ be a complex analytic manifold and $F,G$ closed analytic subspaces with $\supp(F)\subseteq\supp(G)$. A smooth function $f:X\setminus\supp(F)\rightarrow\C$ with logarithmic singularities along $F$ has logarithmic singularities along $G$.\\
ii. Let $\varphi:X\rightarrow Y$ be a morphism of complex analytic manifolds and $F\subseteq Y$ a closed analytic subspace. If $f:Y\setminus\supp(F)\rightarrow\C$ has logarithmic singularities along $F$, then $\varphi^{*}f=f\circ\varphi$ has logarithmic singularities along $\varphi^{-1}(F)$. If $\varphi$ is surjective and proper, the converse holds.
\end{proposition}
\begin{proof}
The first item \textit{i} is straightforward. We shall prove the second part \textit{ii}. Let $V$ be an open subset of $Y$ such that $\mathcal{I}_{F\mid V}=(s_1,\ldots,s_m)$. The ideal sheaf of $\varphi^{-1}(F\cap V)$ is $(\varphi^{*} s_1,\ldots,\varphi^{*} s_m)$. Let $\lbrace V_{\alpha}\rbrace_{\alpha}$ be an open covering of $V$ by relatively compact subsets. For every open $U\subset\subset\varphi^{-1}(V)$ define $U_{\alpha}=U\cap\varphi^{-1}(V_{\alpha})$. Then $U_{\alpha}$ is relatively compact in $\varphi^{-1}(V)$ and $\varphi(U_{\alpha})\subseteq V_{\alpha}$. The estimate
\begin{displaymath}
	|f_{\mid V_{\alpha}\setminus\supp(F)}|\prec\left|\log(\max_{i=1,\ldots,m}|s_i|)^{-1}\right|^{N}
\end{displaymath}
implies the corresponding inequality for $\varphi^{*}(f)$ on $U_{\alpha}$.\\
We now prove the converse under the surjectivity and properness assumption. Let $V$ be as above and $V'\subset\subset V$ an open subset. Since $\varphi$ is proper, $\varphi^{-1}(V')$ is relatively compact in $\varphi^{-1}(V)$. The hypothesis of logarithmic singularities of $\varphi^{*}f$ asserts the existence of an integer $N\geq 0$ such that

\begin{displaymath}
	|\varphi^{*}f_{\mid\varphi^{-1}(V'\setminus\supp(F))}|\prec
	\left|\log(\max_{i=1\ldots, m}|\varphi^{*}s_i|)^{-1}\right|^{N}.
\end{displaymath}
We come up with the conclusion by the surjectivity of $\varphi$.
\end{proof}

\subsection{Differential forms with logarithmic singularities}
\begin{definition}[Divisor with normal crossings]
Let $X$ be a complex analytic manifold of dimension $n$. A reduced analytic subspace $D$ of $X$ is a \textit{divisor with normal crossings} if $X$ can be covered by open subsets $V$ with coordinates $z_1,\ldots,z_n$ such that $D\cap V$ is defined by $z_1\cdot\ldots\cdot z_m=0$, for some $0\leq m\leq n$. We say that $D$ has \textit{simple normal crossings} if it can be written as a finite union of smooth analytic hypersurfaces of $X$.
\end{definition}
\begin{definition}[Adapted analytic chart \cite{BKK}]
Let $X$ be a complex analytic manifold of dimension $n$ and $D$ a divisor with normal crossings in $X$. An analytic chart $(V;\lbrace z_{i}\rbrace_{i=1}^{n})$ is said to be \textit{adapted to} $D$ if $|z_i|<1/e$, $i=1,\ldots,n$ and $D\cap V$ is defined by $z_1\cdot\ldots\cdot z_m=0$, for some $0\leq m\leq n$. The integer $m$ will be understood when no confusion can arise.
\end{definition}
\begin{notation}\label{not_dzeta}
Let $X$ be a complex analytic manifold and $D\subseteq X$ a divisor with normal crossings. Let $(V;\lbrace z_{i}\rbrace_{i})$ be an analytic chart such that $D\cap V$ is defined by $z_{1}\cdot\ldots\cdot z_{m}=0$. We define
\begin{displaymath}
	\dd\zeta_{k}=\begin{cases}
		\frac{\dd z_{k}}{z_{k}\log|z_{k}|^{-1}}, &\text{ if } 1\leq k\leq m\\
		\dd z_{k}, &\text{ if } k>m
	\end{cases}
\end{displaymath}
and similarly for $\cdzeta_{k}$. Given $I,J$ ordered subsets of $\lbrace 1,\ldots,n\rbrace$, we abbreviate
\begin{displaymath}
	\dd\zeta_{I}\wedge\cdzeta_{J}=\bigwedge_{i\in I}\dd\zeta_{i}\wedge\bigwedge_{j\in J}\cdzeta_{j}.
\end{displaymath}
\end{notation}
In the following definitions we write $X$ for a complex analytic manifold of dimension $n$, $D$ a divisor with normal crossings in $X$, $U=X\setminus D$ and $\iota:U\hookrightarrow X$ for the natural open immersion. The sheaf of $\CC^{\infty}$ complex differential forms on $U$ is denoted by $\mathcal{E}^{*}_{U}$.
\begin{definition}[Poincar\'e growth forms \cite{Mumford}]\label{def_Poincare_form}
The \textit{sheaf of Poincar\'e growth forms} on $X$, \textit{with singularities along} $D$, is the subalgebra $\mathcal{P}_{D}$ of $\iota_{*}\mathcal{E}^{*}_{U}$ generated, on every open analytic chart $(V;\lbrace z_{i}\rbrace_{i})$ adapted to $D$, by $\CC^{\infty}(V\setminus D,\C)\cap L^{\infty}_{\loc}(V,\C)$ and the differential forms $\dd\zeta_{k}$, $\cdzeta_{k}$, $k=1,\ldots,n$. Namely, for every analytic chart $(V;\lbrace z_{i}\rbrace_{i})$ adapted to $D$, $\iota_{*}\mathcal{E}^{*}_{U}(V)$ is the $\C$-vector space generated by differential forms
\begin{displaymath}
	\sum_{I,J}\alpha_{I,J}\dd\zeta_{I}\wedge\cdzeta_{J}
\end{displaymath}
where $\alpha_{I,J}\in\CC^{\infty}(V\setminus D)$ are locally bounded on $V$.
\end{definition}
\begin{definition}[Good forms \cite{Mumford}]\label{def_good_forms}
A \textit{good differential form} on an open subset $V$ of $X$, \textit{with singularities along} $D$, is a section $\omega\in\Gamma(V,\mathcal{P}_{D})$ such that $d\omega\in\Gamma(V,\mathcal{P}_{D})$.
\end{definition}
\begin{definition}[log-log growth forms \cite{BKK}]
The \textit{sheaf of log-log growth forms} on $X$, \textit{with singularities along} $D$, is the subalgebra $\mathcal{L}_{D}$ of $\iota_{*}\mathcal{E}^{*}_U$ generated, on every analytic chart $(V;\lbrace z_{i}\rbrace_{i})$ adapted to $D$, by the functions $f\in\CC^{\infty}(V\setminus D,\C)$ such that on every open $V^{\prime}\subset\subset V$
\begin{equation}\label{eq6}
	|f(z_1,\ldots,z_n)|\prec\prod_{k=1}^{m}(\log\log|z_k|^{-1})^{N}
\end{equation}
for some integer $N\geq 0$, together with the differential forms $\dd\zeta_{k}$, $\cdzeta_{k}$, $k=1,\ldots,n$.
\end{definition}
\begin{remark}
Observe that the following inequalities hold:
\begin{displaymath}
	\prod_{k=1}^{m}(\log\log|z_{k}|^{-1})^{N}\leq\sum_{k=1}^{m}(\log\log|z_{k}|^{-1})^{Nm}
	\leq 2^{m-1}\prod_{k=1}^{m}(\log\log|z_{k}|^{-1})^{Nm}.
\end{displaymath}
Therefore, a smooth function $f$ on $V\setminus D$ has log-log growth along $D$ if, and only if, for every open $V^{\prime}\subset\subset V$ there is an estimate
\begin{equation}\label{eq7}
	|f_{\mid V^{\prime}\setminus D}|\prec\sum_{k=1}^{m}(\log\log|z_{k}|^{-1})^{N}.
\end{equation}
In concrete computations involving functions with log-log growth, we may use either formulations (\ref{eq6}) or (\ref{eq7}).
\end{remark}
\begin{proposition}\label{prop_log-log_growth}
Let $f:X\setminus D\rightarrow\C$ be a smooth function. Suppose that $\dd f$ has log-log growth with singularities along $D$. Then $f$ has log-log growth with singularities along $D$. More precisely, for every analytic chart $(V;\lbrace z_{i}\rbrace_{i})$ adapted to $D$ and every open $V^{\prime}\subset\subset V$, if $\dd f=\sum_{j}g_{j}\dd\zeta_{j}+\sum_{j}h_{j}\cdzeta_{j}$ with
\begin{displaymath}
	|g_{j\mid_{V^{\prime}\setminus D}}|, 
	|h_{j\mid_{V^{\prime}\setminus D}}|\prec\prod_{k=1}^{m}(\log\log|z_{k}|^{-1})^{N},
\end{displaymath}
then we have
\begin{displaymath}
	|f_{\mid V^{\prime}\setminus D}|\prec\sum_{i=1}^{m}(\log\log|z_{i}|^{-1})^{N+1}\prod_{i<j\leq m}(\log\log|z_{j}|^{-1})^{N}.
\end{displaymath}
\end{proposition}
\begin{proof}
After localizing to an analytic chart adapted to $D$ we reduce to $V=\Delta_{1/e}^r\times\Delta_{1/e}^s$ with $V\setminus D=\Delta_{1/e}^{\ast r}\times\Delta_{1/e}^{s}$. Let $0<\varepsilon<1$ and $U_{\varepsilon}=\Delta_{\varepsilon/e}^{r}\times\Delta_{\varepsilon/e}^{s}$ be contained in $V$.  Let $(w_1,\ldots,w_{n})\in \overline{U_{\varepsilon}}\setminus D$. Fix $1\leq i\leq r$. Define the curve
\begin{displaymath}
	\begin{split}
		\gamma:[0,1]&\longrightarrow V\\
		t&\longmapsto (t w_1+(1-t)\frac{\varepsilon w_1}{e|w_1|},w_{2},\ldots,w_n).
	\end{split}
\end{displaymath}
Then we have
\begin{displaymath}
	(f\circ\gamma)(1)-(f\circ\gamma)(0)=\int_{0}^{1}\gamma^{*}(\dd f).
\end{displaymath}
Now we write
\begin{displaymath}
	\gamma^{*}(\dd f_{\mid \overline{U_{\varepsilon}}\setminus D})=\gamma^{*}\left(g\frac{\dd z_1}{z_1 \log|z_1|^{-1}}\right)+\gamma^{*}\left(h\frac{\cdz_1}{\cz_1\log|z_1|^{-1}}\right)
\end{displaymath}
where $g,h\in\CC^{\infty}(\overline{U_{\varepsilon}}\setminus D,\C)$ satisfy
\begin{displaymath}
	|g|,|h|\prec\prod_{k=1}^{r}(\log\log|z_k|^{-1})^{N} \text{ on } \overline{U_{\varepsilon}}\setminus D.
\end{displaymath}
A straightforward computation yields
\begin{displaymath}
	\begin{split}
		\bigg|\int_{0}^{1}&\gamma^{*}\left(\dd f_{\mid \overline{U_{\varepsilon}}\setminus D}\right)\bigg|\\
		&\prec_{f,\overline{U_{\varepsilon}}}
		\prod_{1<j\leq r}(\log\log|w_{j}|^{-1})^{N}\cdot
		\int_{0}^{1}(\log\log\gamma^{*}|z_{1}|^{-1})^{N}\frac{\frac{\dd}{\dd t}\log\gamma^{*}|z_{1}|^{-1}}{\log\gamma^{*}|z_{1}|^{-1}}\dd t\\
		&\leq(\log\log|w_1|^{-1})^{N+1}\prod_{1<j\leq r}(\log\log|w_{j}|^{-1})^{N}.
	\end{split}
\end{displaymath}
It follows that
\begin{displaymath}
	\begin{split}
		|f(w_1,\ldots,w_n)|\prec_{f,\overline{U_{\varepsilon}}}&\left|f\left(\frac{\varepsilon w_1}{e|w_1|},w_2,\ldots,w_n\right)\right|\\
		+&(\log\log|w_{1}|^{-1})^{N+1}\prod_{1<j\leq r}(\log\log |w_j|^{-1})^{N}.
	\end{split}
\end{displaymath}
By induction we find
\begin{displaymath}
	\begin{split}
		|f(w_1,\ldots,w_n)|\prec_{f,\overline{U_{\varepsilon}}}&\sup_{\pd\Delta_{\varepsilon/e}^{r}\times\overline{\Delta_{\varepsilon/e}^{s}}}|f(z_1,\ldots z_n)|\\
	+&\sum_{i=1}^{r}(\log\log|w_{i}|^{-1})^{N+1}\prod_{i<j\leq r}(\log\log|w_j|^{-1})^{N}
	\end{split}
\end{displaymath}
Observe that the $\sup$ is finite because $f$ is smooth on $X\setminus D$. The proof is complete.
\end{proof}
\begin{definition}[pre-log-log forms \cite{BKK}]
The \textit{sheaf of pre-log-log forms} on $X$, \textit{with singularities along} $D$, is the subalgebra $\mathcal{E}_{X}^{*}\langle\langle D\rangle\rangle_{\pre}$ of $\iota_{*}\mathcal{E}^{*}_{U}$ generated by log-log growth forms $\omega$ in $\mathcal{L}_{D}$ such that $\pd\omega$, $\cpd\omega$ and $\pd\cpd\omega$ are also log-log growth forms in $\mathcal{L}_{D}$.
\end{definition}
A particular case of pre-log-log forms that deserves to be distinguished is that of P-singular functions.
\begin{definition}[P-singular function]
Let $f:X\setminus D\rightarrow\C$ be a smooth function. We say that $f$ is a P-\textit{singular function}, \textit{with singularities along} $D$, if, and only if, $\dd f$ and $\ddc f$ have Poincar\'e growth along $D$.
\end{definition}
\begin{corollary}\label{cor_P_singular}
Let $f:X\setminus D\rightarrow\C$ be a P-singular function. Then, for every adapted analytic chart $(V;\lbrace z_i\rbrace_{i})$ and every open $V^{\prime}\subset\subset V$, we have
\begin{displaymath}
	|f_{V^{\prime}\setminus D}|\prec\sum_{k=1}^{m}(\log\log|z_{k}|^{-1}).
\end{displaymath}
Consequently, $f$ is a pre-log-log function.
\end{corollary}
\begin{proof}
This is a straightforward application of Proposition \ref{prop_log-log_growth}.
\end{proof}
For later computations it will be worth having at our disposal the following basic properties of log-log growth forms.
\begin{proposition}\label{prop_prop_log}
i. Any log-log growth form is locally integrable. Moreover, log-log growth functions and log-log growth 1-forms are locally $L^2$.\\
ii. (Stokes' theorem for pre-log-log forms.) If $\omega\in\Gamma(X,\mathcal{E}^{*}_{X}\langle\langle D\rangle\rangle_{\pre})$ and $[\omega]$ denotes its associated current, then
		\begin{displaymath}
			d[\omega]=[d\omega]
		\end{displaymath}
		and similarly for $\pd$, $\cpd$ and $\pd\cpd$.\\
iii. If $f:X\rightarrow Y$ is a morphism of complex analytic manifolds and $D_X\subseteq X$, $D_Y\subseteq Y$ normal crossing divisors with $f^{-1}(D_Y)\subseteq D_X$, then $f^{*}\mathcal{P}_{D_Y}\subseteq\mathcal{P}_{D_X}$ and $f^{*}\mathcal{L}_{D_Y}\subseteq\mathcal{L}_{D_X}$. Therefore $f^{*}\mathcal{E}^{*}_{Y}\langle\langle D_Y\rangle\rangle_{\pre}\subseteq\mathcal{E}^{*}_{X}\langle\langle D_X\rangle\rangle_{\pre}$. In particular, this is true for $f$ being the natural inclusion of a complex analytic submanifold $X\subseteq Y$ intersecting $D_Y$ transversally and $D_X=D_Y\cap X$.
\end{proposition}
\begin{proof}
The proposition quotes \cite{BKK}, Proposition 7.5 and Proposition 7.6. However, for later use, we may comment on the proof of \textit{i}. After changing to polar coordinates, it is enough to observe that for every $0<\delta<1$ we have an estimate
\begin{displaymath}
	\int_{0}^{\varepsilon/e}(\log\log t^{-1})^{N}\frac{\dd t}{t(\log t)^2}\prec\int_{0}^{\varepsilon/e}\frac{\dd t}{t(\log t^{-1})^{1+\delta}}<+\infty.
\end{displaymath}
\end{proof}
We finally give an example showing that the notion of pre-log-log function depends on the compactification $X$ of $U$.
\begin{example}
Let $X=\PP^{2}_{\C}$ with projective coordinates $(w_0:w_1:w_2)$. As divisor with normal crossings set $D=(w_0=0)\cup (w_1=0)$. Define the smooth function on $U=X\setminus D$
\begin{displaymath}
	g(w_0:w_1:w_2)=\frac{|w_0|^{2}}{|w_0|^{2}+|w_1|^{2}}.
\end{displaymath}
Denote by $\widetilde{X}$ the blowing-up of $X$ at $(0:0:1)$. $\widetilde{X}$ admits the following description:
\begin{displaymath}
	\widetilde{X}=\left\lbrace ((w_0:w_1:w_2),(z_0:z_1))\in X\times\PP^{1}_{\C}\mid w_0 z_1=w_1 z_0\right\rbrace.
\end{displaymath}
The map realizing the blowing-up is the projection onto the first factor $\pi:\widetilde{X}\rightarrow X$. In particular, since $(0:0:1)\in D$, we have an isomorphism $\pi^{-1}(U)\overset{\sim}{\rightarrow} U$, and $\pi^{-1}(D)$ is a divisor with normal crossings. Observe that the pullback of $g$ by $\pi$ is
\begin{displaymath}
	f((w_0:w_1:w_2),(z_0:z_1))=\frac{|z_0|^{2}}{|z_0|^{2}+|z_1|^{2}}.
\end{displaymath}
The function $f$ extends to a smooth function on the whole $\widetilde{X}$, in particular pre-log-log along $\pi^{-1}(D)$. However, we claim that $g$ is not a pre-log-log function along $D$. To see this we compute $\pd g$. We may localize at the affine neighborhood $w_2\neq 0$ of $(0:0:1)$ and write $u=w_0/w_2$, $v=w_1/w_2$. In coordinates $u$, $v$ we find
\begin{displaymath}
	\pd g=\frac{|u|^2|v|^2}{(|u|^2+|v|^2)^2}\left(\frac{\dd u}{u}-\frac{\dd v}{v}\right).
\end{displaymath}
But the function $|u|^2|v|^2\log|u|^{-1}/(|u|^2+|v|^2)^{2}$ does not have log-log growth along $D$, as we see after restriction to $|u|=|v|$. This proves the claim.\\
\end{example}
\subsection{Variants: log-log forms}
Following \cite{BKK2} we introduce a variant of the sheaf of pre-log-log forms, by imposing bounds on all the derivatives of the component functions of the differential forms. There are also corresponding variants for Poincar\'e growth forms and good forms, for which we refer to \textit{loc. cit.}

We fix a complex analytic manifold $X$ and $D\subseteq X$ a divisor with normal crossings. Write $U=X\setminus D$ and $\iota:U\hookrightarrow X$ for the natural open immersion.
\begin{definition}[log-log functions of infinite order \cite{BKK2}]
A smooth function $f:X\setminus D\rightarrow\C$ is said to be a \textit{log-log function of infinite order}, \textit{with singularities along} $D$, if for every analytic chart $(V;\lbrace z_{i}\rbrace_{i})$ adapted to $D$, every open $V^{\prime}\subset\subset V$ and multi-indices $\alpha=(\alpha_{1},\ldots,\alpha_n)$, $\beta=(\beta_1,\ldots,\beta_n)$, there is a bound on $V^{\prime}$
\begin{displaymath}
	\left|\frac{\pd^{|\alpha|}}{\pd z^{\alpha}}\frac{\pd^{|\beta|}}{\cz^{\beta}}f(z_1,\ldots,z_n)\right|
	\prec\frac{\prod_{k=1}^{m}(\log\log|z_{k}|^{-1})^{N}}{|z^{\alpha^{\leq m}}\cz^{\beta^{\leq m}}|},
\end{displaymath}
where $z^{\alpha^{\leq m}}=z_{1}^{\alpha_1}\ldots z_{m}^{\alpha_m}$ (similarly for $\cz^{\beta^{\leq m}}$) and $N$ depends on $V^{\prime}$, $\alpha$, $\beta$.
\end{definition}
\begin{definition}[log-log growth forms of infinite order \cite{BKK2}]
The \textit{sheaf of log-log growth forms of infinite order} on $X$, \textit{with singularities along} $D$, is the subalgebra of $\iota_{*}\mathcal{E}^{*}_{U}$ generated, on every analytic chart $V$ adapted to $D$, by log-log growth functions of infinite order and the differential forms $\dd\zeta_{k}$, $\cdzeta_{k}$, $k=1,\ldots,n$ (see Notation \ref{not_dzeta}).
\end{definition}
\begin{remark}
Let $\omega$ be a log-log growth form of infinite order along $D$, defined on some analytic open subset of $X$. Then the complex conjugate $\bar{\omega}$ is also a log-log growth form of infinite order along $D$.
\end{remark}
\begin{definition}[log-log forms \cite{BKK2}]
The \textit{sheaf of log-log forms} on $X$, \textit{with singularities along} $D$, is the subalgebra $\mathcal{E}_{X}^{*}\langle\langle D\rangle\rangle$ of $\iota_{*}\mathcal{E}_{U}^{*}$ generated by log-log growth forms $\omega$ of infinite order, such that $\pd\omega$, $\cpd\omega$ and $\pd\cpd\omega$ are also log-log growth forms of infinite order.
\end{definition}
\begin{remark}
There is an obvious inclusion $\mathcal{E}_{X}^{*}\langle\langle D\rangle\rangle\subseteq\mathcal{E}_{X}^{*}\langle\langle D\rangle\rangle_{\pre}$.
\end{remark}
Log-log growth forms of infinite order enjoy of analogous properties to the log-log growth forms introduced before. We refer to \cite{BKK2} for details. An advantage of the sheaf of log-log forms over the sheaf of pre-log-log forms is that a Poincar\'e's type lemma holds for the former. The next essential property follows.
\begin{theorem}[\cite{BKK2}]\label{thm_resolution}
The natural inclusion
\begin{displaymath}
	\Omega_{X}^{\ast}\longrightarrow\mathcal{E}^{\ast}_{X}\langle\langle D\rangle\rangle
\end{displaymath}
is a filtered quasi-isomorphism with respect to the Hodge filtration.
\end{theorem}
\begin{proposition}\label{prop_L2_log-log}
Let $f:X\setminus D\rightarrow\C$ be a smooth function. Then $f$ is a log-log form, with singularities along $D$, if, and only if, $\dd f$ is locally $L^{2}$ on $X$ and $\ddc f$ is a log-log growth form of infinite order along $D$.
\end{proposition}
\begin{proof}
The direct implication is an easy exercise. Let us see the converse. By hypothesis, $\cpd\pd f$ is a log-log form, with singularities along $D$. Let $x\in X$. By Theorem \ref{thm_resolution}, there exists an open neighborhood $V$ of $x$ and a log-log form $\omega$ on $V$ such that
\begin{displaymath}
	\cpd\pd f=\cpd\omega.
\end{displaymath}
Therefore, we can write
\begin{displaymath}
	\pd f=\omega+\theta.
\end{displaymath}
for some holomorphic form $\theta$ on $V\setminus D$. Observe that $\theta$ is locally $L^{2}$, because $\pd f$ is locally $L^2$ by hypothesis and $\omega$ is a log-log 1-form (see Proposition \ref{prop_prop_log}). By Lemma \ref{lem_holom_L2} below, $\theta$ must be holomorphic on $V$. This proves that $\pd f$ has log-log growth of infinite order along $D$. The same reasoning applied to $\cpd\pd\bar{f}$ (the complex conjugate of $\pd\cpd f$) proves that $\cpd f$ has log-log growth of infinite order along $D$. Therefore $\dd f$ has log-log growth of infinite order.\\
Again by Theorem \ref{thm_resolution}, after possibly shrinking $V$, there exists a log-log function of infinite order $g$ and a holomorphic function $h$ on $V\setminus D$ such that $f=g+h$. We claim that $h$ is locally $L^{2}$. Since this is true for $g$, we are reduced to prove it for $f$. But we have already shown that $\dd f$ has log-log growth of infinite order along $D$, so that Proposition \ref{prop_log-log_growth} implies that $f$ has log-log growth along $D$. In particular, $f$ is locally $L^2$ (Proposition \ref{prop_prop_log}). By Lemma \ref{lem_holom_L2}, $g$ is holomorphic on $V$ and hence $f$ has log-log growth of infinite order along $D$. This finishes the proof.
\end{proof}
\begin{lemma}\label{lem_holom_L2}
Let $X$ be a complex analytic manifold and $D\subseteq X$ a divisor with normal crossings. Let $\theta$ be a holomorphic function on $X\setminus D$. If $\theta$ is locally $L^{2}$ on $X$, then $\theta$ extends to a holomorphic function on $X$.
\end{lemma}
\begin{proof}
The lemma is well-known, but we include the proof for lack of reference. It is enough to treat the case when $X=\Delta_{\varepsilon}^{n}\subseteq\C^{n}$ and $D$ is defined by $z_1\cdot\ldots\cdot z_r=0$, so that $X\setminus D=\Delta_{\varepsilon}^{\ast r}\times\Delta_{\varepsilon}^{s}$. We write $\delta=(\delta_{1},\ldots,\delta_{r})\in\R_{>0}^{r}$. Since $\theta$ is locally $L^2$,
\begin{equation}\label{eq3}
	\|\theta\|^{2}_{\varepsilon}:=\lim_{\delta\to 0}I_{\delta}<+\infty
\end{equation}
where
\begin{displaymath}
	I_{\delta}=\int_{(\prod_{k=1}^{r}\Delta_{\varepsilon/2}\setminus\Delta_{\delta_k})\times\Delta_{\varepsilon/2}^{s}}|\theta|^{2}
	\left|\prod_{k=1}^{n}\dd z_{k}\wedge\cdz_{k}\right|.
\end{displaymath}
The Laurent series development
\begin{equation}\label{eq2}
	\theta(z_1,\ldots,z_n)=\sum_{\nu\in\Int^{n}}a_{\nu}z^{\nu}
\end{equation}
is absolutely and uniformly convergent on any $(\prod_{k=1}^{r}\Delta_{\varepsilon/2}\setminus\Delta_{\delta_{k}})\times\Delta_{\varepsilon/2}^{s}$. Therefore, the integral $I_{\delta}$ can be computed term by term:
\begin{displaymath}
	\begin{split}
	I_{\delta}=\sum_{\nu,\mu\in\Int^{n}}a_{\nu}\overline{a_{\mu}}&\left(\prod_{k=1}^{r}\int_{\Delta_{\varepsilon/2}\setminus\Delta_{\delta_{k}}}
	z_{k}^{\nu_k}\cz_{k}^{\mu_{k}}|\dd z_k\wedge\cdz_k|\right)\\
	\cdot&\left(\prod_{k>r}\int_{\Delta_{\varepsilon/2}}z_{k}^{\nu_k}\cz_{k}^{\mu_k}
	|\dd z_k\wedge\cdz_k|\right).
	\end{split}
\end{displaymath}
Recall that given integers $a,b$ we have
\begin{displaymath}
	\int_{0}^{2\pi}e^{ai\theta}\overline{e^{bi\theta}}\dd\theta=2\pi\delta_{a,b},
\end{displaymath}
so that
\begin{equation}\label{eq5}
	\begin{split}
	I_{\delta}=\sum_{\nu\in\Int^{n}}|a_{\nu}|^{2}&\left(\prod_{k=1}^{r}\int_{\Delta_{\varepsilon/2}\setminus\Delta_{\delta_{k}}}
	|z_{k}|^{2\nu_k}|\dd z_k\wedge\cdz_k|\right)\\
	\cdot&\left(\prod_{k>r}\int_{\Delta_{\varepsilon/2}}|z_{k}|^{2\nu_k}
	|\dd z_k\wedge\cdz_k|\right).
	\end{split}
\end{equation}
We reason by contradiction and assume that $\theta$ does not extend to a holomorphic function on $\Delta_{\varepsilon/2}^{n}$. We can suppose that in (\ref{eq2}) there appears a term $a_{\nu}z^{\nu}\neq 0$, with $\nu=(\nu_1,\ldots,\nu_n)\in\Int_{<0}^{l}\times\Int_{\geq 0}^{n-l}$, $1\leq l\leq r$. From (\ref{eq5}) and by direct computation we find
\begin{equation}\label{eq4}
	I_{\delta}\geq (4\pi)^{n}|a_{\nu}|^{2}\prod_{k=1}^{l}J_{\delta_k}
	\prod_{k=l+1}^{r}\left(\frac{(\varepsilon/2)^{2\nu_{k}+2}}{2\nu_{k}+2}-\frac{\delta_{k}^{2\nu_k+2}}{2\nu_k+2}\right)
	\prod_{k>r}\frac{(\varepsilon/2)^{2\nu_k+2}}{2\nu_k +2}.
\end{equation}
where
\begin{displaymath}
	J_{\delta_k}=\begin{cases}
		\log\left(\frac{\varepsilon/2}{\delta_k}\right)&\text{ if }\nu_{k}=-1,\\
		\frac{(\varepsilon/2)^{2\nu_{k}+2}}{2\nu_k+2}-\frac{\delta_{k}^{2\nu_k+2}}{2\nu_k+2}&\text{ if }\nu_{k}<-1.
	\end{cases}
\end{displaymath}
Since $a_{\nu}\neq 0$ and $J_{\delta_k}\to +\infty$ as $\delta\to 0$, we see from (\ref{eq4}) that $I_{\delta}\to+\infty$ as $\delta\to 0$. This contradicts (\ref{eq3}). The proof is complete.
\end{proof}
\section{Logarithmically singular hermitian vector bundles}\label{section_vector_bundles}
Let $X$ be a complex analytic manifold and $D\subseteq X$ a divisor with normal crossings. Write $U=X\setminus D$ and $\iota:U\hookrightarrow X$ for the natural open immersion. In this section we study vector bundles endowed with hermitian metrics with singularities of logarithmic type along $D$. The reader is referred to \textsection \ref{sec_diff_form_log} for the several definitions and properties of differential forms with singularities of logarithmic type along $D$.
\begin{definition}[\cite{BKK2} and \cite{Mumford}]\label{def_good_vb}
Let $E$ be a vector bundle of rank $r$ on $X$. A smooth hermitian metric $h$ on $E_{\mid U}$ is said to have \textit{logarithmic singularities along} $D$ if, for every trivializing open subset $V$ and holomorphic frame $e_{1},\ldots,e_{r}$ of $E_{\mid V}$, putting $h_{ij}=h(e_i,e_j)$ and $H=(h_{ij})$ on $V\setminus D$, the following condition is fulfilled:
\begin{itemize}
	\item[] $(L(E,h))$ the functions $|h_{ij}|$, $\det H^{-1}$ have logarithmic singularities along $D\cap V$ (see Definition \ref{def_func_log_sing}).
\end{itemize}
We say that $h$ is \textit{(pre-)log-log (resp. good) along} $D$ if moreover, for every such data $V$ and $e_{1},\ldots, e_{r}$, the following property holds:
\begin{itemize}
	\item[] $(G(E,h))$ the entries of the matrix $(\pd H)H^{-1}$ are (pre-)log-log (resp. good) forms on $V$, with singularities along $D$ (see Definition \ref{def_good_forms}).
\end{itemize}
We will usually write $\overline{E}=(E,h)$ when no confusion on the metric can arise. Sometimes we use some variants of the definition, and we say for instance \textquotedblleft $\overline{E}$ has logarithmic singularities along $D$\textquotedblright\, or \textquotedblleft $\overline{E}$ is (pre-)log-log (resp. good) along $D$\textquotedblright.
\end{definition}
In the case of line bundles, the notions of (pre-)log-log and good hermitian metrics can be characterized by slightly simpler properties.
\begin{proposition}
Let $L$ be a hermitian line bundle on $X$ and $h$ a smooth hermitian metric on $L_{\mid U}$. Write $\|\cdot\|$ for the norm associated to $h$.\\
i. The metric $h$ is (pre-)log-log (resp. good) with singularities along $D$ if, and only if, for every trivializing open subset $V$ and holomorphic frame $e$ of $L_{\mid V}$, the function $\log h(e,e)$ is (pre-)log-log (resp. P-singular) on $V$, with singularities along $D$.\\
ii. The metric $h$ is log-log with singularities along $D$ if, and only if, for every trivializing open subset $V$ and holomorphic frame $e$ of $L_{\mid V}$, the form $\pd\log h(e,e)$ is locally $L^2$ on $V$ and $\cpd\pd\log h(e,e)$ has log-log growth of infinite order, with singularities along $D$. 
\end{proposition}
\begin{proof}
This follows from the definitions and Proposition \ref{prop_log-log_growth}, Proposition \ref{cor_P_singular} and Proposition \ref{prop_L2_log-log} applied to the smooth real function $\log h(e,e)$ on $V\setminus D$.
\end{proof}
An essential extension property of hermitian vector bundles with logarithmic singularities is the following observation due to Mumford.
\begin{proposition}\label{prop_extension}
Let $(E^{\circ},h)$ be a smooth hermitian vector bundle on $U$. Then there exists at most one extension of $(E^{\circ},h)$ to a hermitian vector bundle $(E,h)$ on $X$, with logarithmic singularities along $D$. More precisely, if $(E,h)$ is such an extension, then for every open subset $V$ in $X$
\begin{displaymath}
	\Gamma(V,E)=\big\lbrace s\in\Gamma(V,\iota_{\ast}E^{\circ})\mid h(s,s)\,\text{ has log. sing. along }\,D\cap V\big\rbrace. 
\end{displaymath}
\end{proposition}
\begin{proof}
This is Proposition 1.3 in \cite{Mumford}.
\end{proof}
Hermitian vector bundles with logarithmic singularities along $D$ admit the following characterization.
\begin{proposition}\label{prop_charac_log}
Let $E$ be a vector bundle on $X$ and $h_E$ a smooth hermitian metric on $E_{\mid U}$. Denote by $h_{E^\vee}$ the dual metric. Then $h_E$ has logarithmic singularities along $D$ if, and only if, the following condition is satisfied with $F=\overline{E}$ and $F=\overline{E}^{\vee}$:
\begin{itemize}
	\item[] $(\widetilde{L}(F,h_F))$ for every open subset $V$ and any holomorphic section $s$ of $F_{\mid V}$, the function $h_{F}(s,s)$ has logarithmic singularities along $D\cap V$.
\end{itemize}
\end{proposition}
\begin{proof}
For the direct implication, first take a holomorphic section $s$ of $E$ over an open subset $V$. We may assume that $s$ does not vanish on $V$. After possibly shrinking $V$, we can complete $s$ to a holomorphic frame $e_1=s,\ldots, e_r$ of $E_{\mid V}$. By the definition of metric with logarithmic singularities, the function $h_{E}(s,s)=h_{E}(e_1,e_1)$ has logarithmic singularities along $D$.\\
Secondly, let $V$ be a trivializing open subset of $E$ and $e_1,\ldots,e_r$ a holomorphic frame of $E_{\mid V}$. Write $H=(h_{ij})$ for the matrix of $h_{E}$ in base $\lbrace e_i\rbrace_{i}$ and $H^{-1}=(g_{ij})$ for the inverse matrix. From the very construction of $H^{-1}$ and the logarithmic singularities of the functions $h_{ij}$ and $\det H^{-1}$, it is immediate to check that the functions $g_{ij}$ have logarithmic singularities along $D$. If $B$ is the matrix of $h_{E^{\vee}}$ in any holomorphic frame of $E^{\vee}_{\mid V}$, then there exists $A\in\text{GL}_{r}(\Gamma(V,\OO_{X}))$ such that
\begin{displaymath}
	B=A^{t}\cdot H^{-1} \cdot\overline{A}.
\end{displaymath}
Since the entries of $A$ are holomorphic, the entries of $B$ inherit from $H^{-1}$ the logarithmic singularities along $D\cap V$.\\
Let now $s$ be a holomorphic section of $E^{\vee}$ over an open subset $V$. Replacing $V$ by a smaller open subset, we can complete $s$ to a holomorphic frame of $E^{\vee}_{\mid V}$, $v_1=s,\ldots,v_r$. As we have just proven the functions $h_{E^{\vee}}(v_i,v_j)$ have logarithmic singularities along $D$, in particular so does $h_{E^{\vee}}(s,s)=h(v_1,v_1)$.

Now for the converse. Let $V$ be a trivializing open subset, adapted to $D$. Let $e_1,\ldots, e_r$ be a frame for $E_{\mid V}$. Write $H=(h_{ij})$ for the matrix of the hermitian metric $h_E$ in base $\lbrace e_i\rbrace_i$. By hypothesis, for every open subset $V^{\prime}\subset\subset V$, there exists an integer $N\geq 0$ such that on $V^{\prime}$
\begin{displaymath}
	h_{E}(e_i,e_i)\prec(\log |z_1\cdot\ldots\cdot z_m|^{-1})^{N}.
\end{displaymath}
Applying Schwarz's inequality we get
\begin{displaymath}
	|h_{ij}|^{2}\prec(\log|z_1\cdot\ldots\cdot z_m|^{-1})^{2N}.
\end{displaymath}
The same argument provides similar bounds for the entries of the matrix of $h_{E^{\vee}}$ in the dual basis, namely $H^{-1}$. Since the determinant of $H^{-1}$ is a polynomial in the entries of this matrix, we derive a bound
\begin{displaymath}
	\det H^{-1}\prec(\log|z_1\cdot\ldots\cdot z_m|^{-1})^{M}
\end{displaymath}
for some integer $M$. This concludes the proof.
\end{proof}
As an immediate consequence of the proposition we establish the following corollary.
\begin{corollary}\label{cor_ex_seq}
Let $\overline{E}=(E,h_E)$ be a hermitian vector bundle with logarithmic singularities along $D$. For every exact sequence of vector bundles
\begin{displaymath}
	0\longrightarrow F\longrightarrow E\longrightarrow Q\longrightarrow 0,
\end{displaymath}
the induced hermitian vector bundles $\overline{F}=(F,h_F)$ (restricted metric) and $\overline{Q}=(Q,h_Q)$ (quotient metric) have logarithmic singularities along $D$.
\end{corollary}
\begin{proof}
It is enough to prove that for every exact sequence as in the statement, conditions $(\widetilde{L}(\overline{F}))$ and $(\widetilde{L}(\overline{Q}))$ hold. Indeed, since $\overline{E}^{\vee}$ has logarithmic singularities along $D$, conditions $(\widetilde{L}(\overline{F}^{\vee}))$ and $(\widetilde{L}(\overline{Q}^{\vee}))$ automatically follow by duality. Then we conclude applying Proposition \ref{prop_charac_log}. The validity of $\widetilde{L}(\overline{F})$ is clear. For $\widetilde{L}(\overline{Q})$, we just observe that if $s$ is a holomorphic section of $Q_{\mid V}$ and $\tilde{s}$ is a holomorphic section of $E_{\mid V}$ lifting $s$, then
\begin{displaymath}
	h_{Q}(s,s)\leq h_{E}(\tilde{s},\tilde{s}).
\end{displaymath}
Thus we see that $\widetilde{L}(\overline{E})$ implies $\widetilde{L}(\overline{Q})$.
\end{proof}
We next state the main formal properties of logarithmically singular (resp. (pre-)log-log, resp. good) hermitian vector bundles.
\begin{proposition}\label{prop_basic_prop}
Let $E$, $F$ be two vector bundles on $X$ and $h_E$ and $h_F$ smooth hermitian metrics on  $E_{\mid U}$ and $F_{\mid U}$, respectively. If $h_E$ and $h_F$ have logarithmic (resp. (pre-)log-log, resp. good) singularities along $D$, then $\overline{E}^{\vee}$, $\overline{E}\otimes\overline{F}$, $S^{k}\overline{E}$ and $\wedge^{k}\overline{E}$ have logarithmic (resp. (pre-)log-log, resp. good) singularities along $D$.
\end{proposition}
\begin{proof}
Left as an elementary exercise.
\end{proof}
\begin{proposition}\label{prop_vb_log}
Let $X$, $Y$ be complex analytic manifolds and $D_X\subseteq X$, $D_Y\subseteq Y$ normal crossing divisors. Let $f:X\rightarrow Y$ be a morphism of complex analytic manifolds. Let $\overline{E}=(E,h)$ be a hermitian vector bundle on $Y$ whose metric is defined and smooth on $Y\setminus D_Y$.\\
i. If $f^{-1}(D_Y)\subseteq D_X$ and $h$ has logarithmic (resp. (pre-)log-log, resp. good) singularities along $D_Y$, then the metric $f^{*}(h)$ on $f^{*}(E)$ has logarithmic (resp. (pre-)log-log, resp. good) singularities along $D_X$.\\
ii. Suppose that $f$ is surjective, proper and $f^{-1}(D_Y)=D_X$. Then $h$ has logarithmic singularities along $D_Y$ if, and only if, $f^{*}(\overline{E})$ has logarithmic singularities along $D_X$.
\end{proposition}
\begin{proof}
The first item \textit{i} follows from Proposition \ref{prop_prop_log_sing} \textit{i} and Proposition \ref{prop_prop_log} \textit{iii}. The second item \textit{ii} is automatically deduced from Proposition \ref{prop_prop_log_sing} \textit{ii}.
\end{proof}
\begin{corollary}\label{coro_vb_log}
Let $(E,h)$ be a hermitian vector bundle on $X$, with singularities along $D$. Let $\OO_{E}(1)$ be the dual of the trivial vector bundle of $\PP(E)$, the projective space of lines in $E^{\vee}$. Denote by $\pi:\PP(E)\rightarrow X$ the natural projection. Then the metric on $\OO_{E}(1)$ induced by $\pi^{*}(h)$ has logarithmic singularities along $\pi^{-1}(D)$.
\end{corollary}
\begin{proof}
By definition, the line bundle $\OO_{E}(1)$ is a quotient of $\pi^{*}(E)$. The hermitian metric on $\OO_{E}(1)$ is the quotient metric from $\pi^{*}(\overline{E})$. By Proposition \ref{prop_vb_log}, $\pi^{*}(\overline{E})$ has logarithmic singularities along $\pi^{-1}(D)$. Then, by Corollary \ref{cor_ex_seq}, so does the induced metric on $\OO_{E}(1)$.
\end{proof}
The end of this section is devoted to some counter-examples.
\begin{example}
\textit{i. Counter-example to Corollary \ref{cor_ex_seq} and Corollary \ref{coro_vb_log} for pre-log-log hermitian vector bundles.} Let $X=\Af^{1}_{\C}$ be the complex line, with analytic coordinate $z$. Let $D$ be the divisor with normal crossings $z=0$. As vector bundle we take $E=\OO_{X}^{\oplus 2}$. We consider a hermitian metric $h$ on $E$ such that, in the standard basis $e_1$, $e_2$ and near the origin, its matrix $H$ looks like
\begin{displaymath}
	H=
	\left(\begin{array}{cc}
		\log|z|^{-1} & 0\\
		0 & 1\\
		\end{array}\right).
\end{displaymath}
It is easily seen that $(E,h)$ is pre-log-log along $D$ (actually good). However, the induced metric on the line bundle $\OO_{E}(-1)\subseteq \pi^{*}(E^{\vee})$ on $\PP(E)=\PP^{1}_{\C}\times\Af^{1}_{\C}$ is not pre-log-log. Indeed, identifiy $\Af^{1}_{\C}$ as an open subset of $\PP^{1}_{\C}$ via $t\mapsto (1:t)$. If $e_1^{\vee}$, $e_{2}^{\vee}$ is the dual basis, define the section $s$ of $\OO_{E}(-1)_{\mid\Af^{1}_{\C}\times\Af^{1}_{\C}}$ by
\begin{displaymath}
	s=e_{1}^{\vee}+te_{2}^{\vee}.
\end{displaymath}
Then $h^{\vee}(s,s)=(\log|z|^{-1})^{-1}+|t|^2$ and
\begin{displaymath}
	\frac{\pd}{\pd t}\log h^{\vee}(s,s)=\frac{\overline{t}}{(\log|z|^{-1})^{-1}+|t|^{2}}.
\end{displaymath}
If we restrict $\pd\log h^{\vee}(s,s)/\pd t$ to the set $C:t=(\log|z|^{-1})^{-1/2}$, we find
\begin{displaymath}
	\left(\frac{\pd}{\pd t}\log h^{\vee}(s,s)\right)_{\mid C}=\frac{(\log|z|^{-1})^{1/2}}{2}
\end{displaymath}
which does not have log-log growth near $z=0$.\\
\textit{ii. The notion of hermitian vector bundle with logarithmic singularities depends on the compactification.} Let $Y$ be a smooth complex projective surface. Let $p$ be a closed point in $Y$ and $\pi:X\rightarrow Y$ the blowing-up of $Y$ at $p$. Let $D$ be a divisor with normal crossings in $Y$ with $p\in D$. Then $\pi^{-1}(D)$ is a divisor with normal crossings. Define $U=X\setminus \pi^{-1}(D)$ and $V=Y\setminus D$. Then $\pi$ induces an isomorphism between $U$ and $V$. Let $h$ be a smooth hermitian metric on $\omega_{Y}\mid_{V}$, and endow $\omega_{X}\mid_{U}$ with the induced metric $\pi^{*}(h)$. Assume that $(\omega_X,\pi^{*}(h))$ has logarithmic singularities along $\pi^{-1}(D)$. Then we claim that $h$ does not define a metric on $\omega_{Y}$ with logarithmic singularities along $D$. Indeed, suppose that $\overline{\omega_Y}=(\omega_{Y},h)$ had logarithmic singularities along $D$. Then, by Proposition \ref{prop_vb_log}, $\pi^{*}(\overline{\omega_{Y}})$ would have logarithmic singularities along $\pi^{-1}(D)$. Observe that
\begin{displaymath}
	\pi^{*}(\overline{\omega_Y})_{\mid U}=\overline{\omega_X}_{\mid U}.
\end{displaymath}
By Proposition \ref{prop_extension} we would derive the equality $\pi^{*}(\omega_{Y})=\omega_{X}$. However we know that
\begin{displaymath}
	\pi^{*}(\omega_Y)=\omega_{X}\otimes\OO(-E)
\end{displaymath}
where $E$ is the exceptional divisor $\pi^{-1}(p)$. Since the self-intersection $(E^{2})=-1$, $\OO(-E)$ is not trivial. We thus arrive to a contradiction and the claim is proved. We remark that we can produce such examples just endowing $\omega_X$ with a smooth hermitian metric and then restricting it to $U$.
\end{example}
\section{Global bounds for real log-log growth (1,1)-forms}\label{section_decomp_thm}
\subsection{Statement of the theorem and consequences}\label{statement_decomp_thm}
Let $X$ be a complex analytic manifold and $D\subseteq X$ a divisor with \textit{simple} normal crossings. Decompose $D$ into smooth irreducible components, $D=D_{1}\cup\ldots\cup D_{m}$. For every $D_{k}$ we fix a global section $s_{k}$ of $\OO(D_{k})$ with divisor $\gdiv s_{k}=D_{k}$. We endow $\OO(D_{k})$ with a smooth hermitian metric $\|\cdot\|_{k}$ such that $\|s_{k}\|_{k}^{2}\leq e^{-e}$. Therefore $1\leq\log\log\|s_{k}\|^{-2}\leq +\infty$ on $X$.
\begin{notation}\label{notation_Theta_N}
For every integer $N\geq 0$ we define the real positive smooth function on $X\setminus D$
\begin{displaymath}
	\Theta_{N}=\sum_{k=1}^{m}(\log\log\|s_{k}\|_{k}^{-2})^{N}.
\end{displaymath}
\end{notation}
The purpose of this section is the proof of the following global bounds for real log-log growth (1,1)-forms on a compact complex analytic manifold.
\begin{theorem}\label{decomp_thm}
Suppose that $X$ is compact and let $\omega$ be a smooth positive (1,1)-form on $X$. Let $\eta$ be a real log-log growth (1,1)-form on $X$, with singularities along $D$. Then there exist constants $A,B>0$ and an integer $N\geq 0$ such that on $X\setminus D$
\begin{equation}\label{main_inequality}
	\eta+B\Theta_{N}(\ddc(-\Theta_1)+A\omega)\geq 0.
\end{equation}
If moreover $\eta$ has Poincar\'e growth along $D$, then $N$ can be chosen to be 0.
\end{theorem}
The proof of the theorem is postponed until \textsection \ref{subsection_proof_thm}. Now we may discuss a result appearing as a particular instance of Theorem \ref{decomp_thm}.
\begin{theorem}\label{decomp_thm_2}
Suppose that $X$ is compact and let $\omega$ be a smooth positive (1,1)-form on $X$. Let $f:X\setminus D\rightarrow\R$ be a pre-log-log function, with singularities along $D$. Then there exist positive pre-log-log functions, with singularities along $D$,
\begin{displaymath}
	\varphi,\psi:X\setminus D\longrightarrow\R_{\geq 0}
\end{displaymath}
and constants $A,B\geq 0$, $N\in\Int_{\geq 0}$ with the properties
\begin{itemize}
	\item[i.] $f$ is the difference of $\varphi$ and $\psi$: $f=\varphi-\psi$;
	\item[ii.] the following inequalities hold on $X\setminus D$:
		\begin{displaymath}
			\begin{split}
				\omega_{\varphi}:=&\ddc(-\varphi)+B\Theta_{N}(\ddc(-\Theta_1)+A\omega)\geq 0,\\
				\omega_{\psi}:=&\ddc(-\psi)+B\Theta_{N}(\ddc(-\Theta_1)+A\omega)\geq 0.
			\end{split}
		\end{displaymath}
		If $f$ is P-singular, then $N$ can be chosen to be 0;
	\item[iii.] if $f$ is P-singular, one can take $\varphi$, $\psi$ to be P-singular with
	\begin{displaymath}
		\begin{split}
			&\ddc(-\varphi)+A\omega\geq 0,\\
			&\ddc(-\psi)+A\omega\geq 0
		\end{split}
	\end{displaymath}
on $X\setminus D$.
\end{itemize}
\end{theorem}
\begin{proof}
Since $X$ is compact, from the log-log growth of $f$ it is easily seen that for some constant $C>0$ and integer $M\geq 0$,
\begin{displaymath}
	f+C\Theta_{M}\geq 0\,\text{ on }\, X\setminus D.
\end{displaymath}
If $f$ is P-singular, then by Corollary \ref{cor_P_singular} we can take (as we do) $M\leq 1$. We define $\widetilde{\varphi}=f+C\Theta_{M}$ and $\widetilde{\psi}=C\Theta_{M}$. These are positive pre-log-log functions, with singularities along $D$ (see Lemma \ref{lemma_Theta_N} below). If $f$ is P-singular (hence $M\leq 1$), then $\widetilde{\varphi}$, $\widetilde{\psi}$ are P-singular (again by Lemma \ref{lemma_Theta_N}). By Theorem \ref{decomp_thm} there exist constants $A,B\geq 0$ and $N\in\Int_{\geq 0}$ such that
\begin{displaymath}
	\omega_{\widetilde{\varphi}}:=\ddc(-\widetilde{\varphi})+B\Theta_{N}(\ddc(-\Theta_1)+A\omega)\geq 0
\end{displaymath}
and
\begin{displaymath}
	\omega_{\widetilde{\psi}}:=\ddc(-\widetilde{\psi})+B\Theta_{N}(\ddc(-\Theta_1)+A\omega)\geq 0
\end{displaymath}
hold on $X\setminus D$. Hence $\varphi=\widetilde{\varphi}$ and $\psi=\widetilde{\psi}$ satisfy the requirements of \textit{i} and \textit{ii}. If $f$ is P-singular, then $\ddc(-\widetilde{\varphi})$, $\ddc(-\widetilde{\psi})$ have Poincar\'e growth along $D$ and we may take $N=0$. In this case we have
\begin{displaymath}
	\begin{split}
		\omega_{\widetilde{\varphi}}=&\ddc(-\widetilde{\varphi})+mB\ddc(-\Theta_1)+mAB\omega\\
		=&\ddc(-(\widetilde{\varphi}+mB\Theta_1))+mAB\omega
	\end{split}
\end{displaymath}
and similarly
\begin{displaymath}
		\omega_{\widetilde{\psi}}=\ddc(-(\widetilde{\psi}+mB\Theta_1))+mAB\omega
\end{displaymath}
In view of these equalities, $\varphi=\widetilde{\varphi}+mB\Theta_1$ and $\psi=\widetilde{\psi}+mB\Theta_1$ fulfill the requirement of \textit{iii}.
\end{proof}
We include the next corollary for its own interest, but we will not need it in the sequel.
\begin{corollary}
Suppose that $X$ is compact and K\"ahler. Let $\omega$ be a K\"{a}hler form on $X$. Let $f:X\setminus D\rightarrow\R$ be a P-singular function and $f=\varphi-\psi$ a decomposition as in Theorem \ref{decomp_thm_2} \textit{iii}. Then the functions 
\begin{displaymath}
	-\varphi,-\psi:X\setminus D\longrightarrow\R_{\leq 0}
\end{displaymath}
uniquely extend to quasiplurisubharmonic functions\footnote{A \textit{quasiplurisubharmonic} function on a complex analytic manifold $M$ is an upper semi-continuous function $h:M\rightarrow[-\infty,+\infty[$ which is locally the sum of a smooth function and a plurisubharmonic function.} on $X$.
\end{corollary}
\begin{proof}
First of all, since $-\varphi$ and $-\psi$ are pre-log-log along $D$, by Proposition \ref{prop_prop_log} we have the equality of currents $\ddc[-\varphi]=[\ddc(-\varphi)]$ and $\ddc[-\psi]=[\ddc(-\psi)]$ on $X$. The inequalities
\begin{equation}\label{positive_currents}
	\begin{split}
		&\ddc[-\varphi]+A\omega\geq 0,\\
		&\ddc[-\psi]+A\omega\geq 0
	\end{split}
\end{equation}
then hold on $X$ in the sense of currents. Let $U\subset X$ be an open subset diffeomorphic to a complex euclidian ball. Because $\omega$ is $d$-closed (K\"ahler assumption), by Poincar\'e's lemma $\omega_{\mid U}$ is $d$-exact. Since $U$ itself is K\"ahler, $\omega_{\mid U}$ is in fact $\ddc$-exact. Write $\omega_{\mid U}=\ddc h$ for some smooth function $h$ on $U$. Then the currents $\ddc[-\varphi_{\mid U}+A h]$ and $\ddc[-\psi_{\mid U}+Ah]$ are positive on $U$, by (\ref{positive_currents}). Since $-\varphi$ and $-\psi$ are bounded above and $D$ is polar, $\widetilde{\varphi}:=-\varphi_{\mid U}+Ah$, $\widetilde{\psi}:=-\psi_{\mid U}+Ah$ uniquely extend to plurisubharmonic functions on $U$ (see \cite{Demailly}, Theorem 5.24). Since $h$ is smooth, these extensions determine extensions of $-\varphi_{\mid U}$ and $-\psi_{\mid U}$ to quasiplurisubharmonic functions on $U$, clearly unique. The corollary follows.
\end{proof}
\subsection{Construction of pre-log-log functions}\label{construction_pre-log-log}
\subsubsection{Preliminaries}
\begin{lemma}\label{lemma_ineq_diff_form}
Let $M$ be a complex analytic manifold and $\alpha$, $\beta$, $\CC^{\infty}$ differential forms of type (1,0) on $M$. For every function $K>0$ on $M$, the following inequality holds:
\begin{displaymath}
	2\Real(i\alpha\wedge\overline{\beta})\geq -\frac{i}{K}\alpha\wedge\overline{\alpha}
	-iK\beta\wedge\overline{\beta}.
\end{displaymath}
\end{lemma}
\begin{proof}
On one hand, the (1,1)-form 
\begin{displaymath}
	\mu=i(\alpha/K^{1/2}+K^{1/2}\beta)\wedge\overline{(\alpha/K^{1/2}+K^{1/2}\beta)}
\end{displaymath}
is semi-positive. On the other hand, there is an equality
\begin{displaymath}
	\mu=	\frac{i}{K}\alpha\wedge\overline{\alpha}+2\Real(i\alpha\wedge\overline{\beta})+
	iK\beta\wedge\overline{\beta}.
\end{displaymath}
The lemma follows.
\end{proof}
\begin{lemma}\label{lemma_theta_N}
Let $L$ be a line bundle on $X$. Suppose that $s\in\Gamma(X,L)$ is a global section such that $\gdiv s$ is a divisor with normal crossings. Let $\|\cdot\|$ be a smooth hermitian metric on $L$ with $\|s\|^{2}\leq e^{-e}$. For every integer $N\geq 0$, the smooth function
\begin{displaymath}
	\theta_{N}=(\log\log\|s\|^{-2})^{N}:X\setminus\gdiv s\longrightarrow\R_{\geq 1}
\end{displaymath}
is pre-log-log, with singularities along $\gdiv s$. On $X\setminus\gdiv s$ the following identities hold:
\begin{displaymath}
	\pd\theta_{N}=N\theta_{N-1}\frac{\pd\log\|s\|^{-2}}{\log\|s\|^{-2}}
\end{displaymath}
and
\begin{displaymath}
	\ddc(-\theta_N)=\frac{\theta_1-N+1}{N\theta_N}\frac{i}{2\pi}\pd\theta_N\wedge\cpd\theta_N
	-N\theta_{N-1}\frac{\c1(\overline{L})}{\log\|s\|^{-2}},
\end{displaymath}
provided $N\geq 1$. If $N=1$, $\theta_1$ is P-singular along $\gdiv s$.
\end{lemma}
\begin{proof}
The lemma is trivial for $N=0$. We suppose $N\geq 1$. First of all we remark that $\theta_{N}$ has log-log growth along $\gdiv s$. Next we compute $\pd\theta_{N}$ and $\ddc(-\theta_{N})$ on $X\setminus\gdiv s$. We find
\begin{equation}\label{eq8}
	\pd\theta_{N}=N\theta_{N-1}\frac{\pd\log\|s\|^{-2}}{\log\|s\|^{-2}}
\end{equation}
and
\begin{displaymath}
	\begin{split}
		\ddc(-\theta_{N})=-\frac{i}{2\pi}\pd\cpd\theta_{N}=-&\frac{i}{2\pi}N\pd\theta_{N-1}\wedge\frac{\cpd\log\|s\|^{-2}}{\log\|s\|^{-2}}\\
		+&\frac{i}{2\pi}N\theta_{N-1}\frac{\pd\log\|s\|^{-2}\wedge\cpd\log\|s\|^{-2}}{(\log\|s\|^{-2})^{2}}\\
		-&\frac{i}{2\pi}N\theta_{N-1}\frac{\pd\cpd\log\|s\|^{-2}}{\log\|s\|^{-2}}.
	\end{split}
\end{displaymath}
To simplify the last equality, rearrange terms, use (\ref{eq8}) and the trivial fact $\theta_{N+M}=\theta_{N}\theta_{M}$, and recall that on $X\setminus\gdiv s$ we have $\c1(\overline{L})=\ddc\log\|s\|^{-2}$. We get
\begin{equation}\label{eq9}
	\ddc(-\theta_{N})=\frac{\theta_{1}-N+1}{N\theta_{N}}\frac{i}{2\pi}\pd\theta_{N}\wedge\cpd\theta_{N}-N\theta_{N-1}\frac{\c1(\overline{L})}{\log\|s\|^{-2}}.
\end{equation}
Observe that the quotient $(\theta_{1}-N+1)/N\theta_{N}$ is bounded, so it has log-log growth along $\gdiv s$. Also the function $1/\log\|s\|^{-2}$ is bounded and $\c1(\overline{L})$ is smooth, so that $\c1(\overline{L})/\log\|s\|^{-2}$ has log-log growth along $\gdiv s$. Hence from (\ref{eq8}) and (\ref{eq9}) we see that it is enough to prove that $\pd\theta_{N}$ has log-log growth along $\gdiv s$ or, still, that $\pd\log\|s\|^{-2}/\log\|s\|^{-2}$ has log-log growth along $\gdiv s$.\\
Let $V$ be an open analytic chart adapted to $\gdiv s$ such that $L_{\mid V}$ can be trivialized and $s=z_{1}\ldots z_{m} e$, where $e$ is a holomorphic frame of $L_{\mid V}$. We can write
\begin{displaymath}
	\frac{\pd\log\|s\|^{-2}}{\log\|s\|^{-2}}=\frac{\pd\log\|e\|^{-2}}{\log\|s\|^{-2}}+\sum_{k=1}^{m}\frac{\log|z_{k}|}{\log\|s\|^{-2}}\frac{\dd z_{k}}{z_{k}\log|z_k|^{-1}}.
\end{displaymath}
The differential form $\pd\log\|e\|^{-2}$ is smooth on $V$ and $(\log\|s\|^{-2})^{-1}$ is bounded on $V$, so that the first term is bounded on any small enough open $V^{\prime}\subset\subset V$. As for the sum, we observe that $\log|z_{k}|/\log\|s\|^{-2}$ is bounded on any small enough open $V^{\prime}\subset\subset V$, because $\log\|s\|^{-2}=\log\|e\|^{-2}+\sum_{j=1}^{m}\log|z_{j}|^{-2}$ and $\log\|e\|^{-2}$ is smooth. Hence $\pd\log\|s\|^{-2}/\log\|s\|^{-2}$ has log-log growth along $\gdiv s$. The proof is complete.
\end{proof}
\begin{lemma}\label{lemma_Theta_N}
Under the hypothesis of \textsection \ref{statement_decomp_thm} and with the notations therein, the functions $\Theta_N$ are pre-log-log, with singularities along $D$. If $N=1$, then $\Theta_1$ is P-singular.
\end{lemma}
\begin{proof}
Write $\theta_{N}^{(k)}=(\log\log\|s_k\|_{k}^{-2})^{N}$, $\Theta_N=\sum_{k=1}^{m}\theta_{N}^{(k)}$ and apply Lemma \ref{lemma_theta_N}.
\end{proof}
\subsubsection{Local results}
Let $L$ be a line bundle on $X$ admitting a global section $s\in\Gamma(X,L)$ whose associated divisor $\gdiv s$ is \textit{smooth} and \textit{irreducible}. Let $\|\cdot\|$ be a smooth hermitian metric on $L$ with $\|s\|^{2}\leq e^{-e}$ on $X$. Define, as before, the function
\begin{displaymath}
	\theta_1=\log\log\|s\|^{-2}:X\setminus \gdiv s\longrightarrow\R_{\geq 1}.
\end{displaymath}
By Lemma \ref{lemma_theta_N} we can write
\begin{equation}\label{eq15}
	\ddc(-\theta_1)=\frac{i}{2\pi}\frac{\pd\log\|s\|^{-2}\wedge\cpd\log\|s\|^{-2}}{(\log\|s\|^{-2})^2}
	-\frac{\c1(\overline{L})}{\log\|s\|^{-2}}.
\end{equation}
Let $(V;z_1,\ldots,z_n)$ be an analytic chart adapted to $\gdiv s$ with $V\cap \gdiv s\neq\emptyset$. We suppose that $L_{\mid V}$ can be trivialized and we denote $u$ for a holomorphic frame such that $s=z_1 u$. In local coordinates equality (\ref{eq15}) becomes $\ddc(-\theta_1)=\alpha+\beta+\gamma$, where
\begin{itemize}
	\item[] \begin{displaymath}
			\begin{split}
				\alpha:=&\frac{i}{2\pi}\frac{\pd\log|z_1|^{-2}\wedge\cpd\log|z_1|^{-2}}
				{(\log\|s\|^{-2})^2}\\
				=&a\frac{i}{2\pi}\frac{\dd z_1\wedge\cdz_1}{|z_1|^{2}(\log|z_1|^{-2})^{2}},
			\end{split}
		\end{displaymath}
	\item[] \begin{displaymath}
			\begin{split}
				\beta:=&-2\Real\left(\frac{i}{2\pi}\frac{\pd\log|z_1|^{-2}\wedge\cpd\log\|u\|^{-2}}
				{(\log\|s\|^{-2})^{2}}\right)\\
				=&2a\Real\left(\frac{i}{2\pi}\frac{\dd z_1\wedge\cpd\log\|u\|^{2}}
				{z_1(\log|z_1|^{-2})^{2}}\right),
			\end{split}
		\end{displaymath}
	\item[] $$\gamma:=\frac{i}{2\pi}\frac{\pd\log\|u\|^{2}\wedge\cpd\log\|u\|^{2}}{(\log\|s\|^{-2})^2}
	-\frac{\c1(\overline{L})}{\log\|s\|^{-2}},$$
\end{itemize}
and $a=(\log|z_1|/\log\|s\|)^{2}$. Decompose $\cpd\log\|u\|^{2}=\sum_{j=1}^{n}q_{j}\cdz_j$, where the functions $q_j$ are smooth on $V$. Then $\beta$ can be expanded as a sum $\beta=\sum_{j=1}^{n}\beta_{j}$, with
\begin{displaymath}
	\beta_{j}:=2a\Real\left(q_{j}\frac{i}{2\pi}\frac{\dd z_{1}\wedge\cdz_{j}}{z_{1}(\log|z_1|^{-2})^{2}}\right).
\end{displaymath}
\begin{proposition}\label{prop_local}
Let $L$ be a line bundle on $X$ admitting a global section $s\in\Gamma(X,L)$ such that $\gdiv s$ is smooth and irreducible. Let $\|\cdot\|$ be a smooth hermitian metric on $L$ with $\|s\|^{2}\leq e^{-e}$ and define $\theta_1=\log\log\|s\|^{-2}$. Let $\omega$ be a positive (1,1)-form on $X$. For every analytic chart $(V;z_1,\ldots,z_n)$ adapted to $\gdiv s$ and any open $V^{\prime}\subset\subset V$ intersecting $\gdiv s$, there exists $A>0$ such that on $V^{\prime}\setminus\gdiv s$
\begin{displaymath}
	\ddc(-\theta_1)+A\omega\geq \frac{1}{4}\frac{i}{2\pi}\frac{\dd z_1\wedge\cdz_1}{|z_1|^{2}(\log|z_1|^{-2})^{2}}.
\end{displaymath}
\end{proposition}
\begin{proof}
Without loss of generality, suppose $\sup_{V}|\log\|u\||$ and $\sup_{V}|q_j|$ finite for all $j$ (otherwise, replace $V$ by a relatively compact open subset containing $V^{\prime}$). We divide the proof into three steps.\\
\\
\textit{Step 1.} Observe that because $\log\|s\|=\log|z_1|+\log\|u\|$ and $\log\|u\|$ is bounded on $V$, the function $a$ uniformly tends to 1 as $z_1$ tends to $0$. Therefore there exists an open $V^{\prime\prime}\subset V^{\prime}$ such that $1/2\leq a\leq 2$ on $V^{\prime\prime}\setminus\gdiv s$. For later need, we take $V^{\prime\prime}$ so that $\overline{V^{\prime}\setminus V^{\prime\prime}}$ \textit{does not} intersect $\gdiv s$. On $V^{\prime\prime}\setminus \gdiv s$ the following inequality holds:
\begin{equation}\label{eq18}
	\alpha\geq\frac{1}{2}\frac{i}{2\pi}\frac{\dd z_1\wedge\cdz_1}{|z_1|^{2}(\log|z_1|^{-2})^{2}}.
\end{equation}
\\
\textit{Step 2.} Define $C=\max_{j}\sup_{V}|q_j|$. Shrinking $V^{\prime\prime}$ if necessary, we assume that $|z_1|\leq 1/16n(C+1)$ on $V^{\prime\prime}$. Since $a\leq 2$ on $V^{\prime\prime}$, we have
\begin{equation}\label{eq17}
	\begin{split}
		\beta_1=2a\Real\left(\frac{i}{2\pi}q_{1}z_{1}\frac{\dd z_1\wedge\cdz_1}
		{|z_1|^{2}(\log|z_1|^{-2})^{2}}\right)&\geq
		-4C|z_1|\frac{i}{2\pi}\frac{\dd z_1\wedge\cdz_1}{|z_1|^{2}(\log|z_1|^{-2})^{2}}\\
		&\geq-\frac{1}{4n}\frac{i}{2\pi}\frac{\dd z_1\wedge\cdz_1}{|z_1|^{2}(\log|z_1|^{-2})^{2}}
	\end{split}
\end{equation}
on $V^{\prime\prime}\setminus\gdiv s$.
Fix a constant $K\geq 8nC$. Applying Lemma \ref{lemma_ineq_diff_form}, for every $j>1$ we find
\begin{equation}\label{eq16}
	\beta_j\geq -2\frac{C}{K}\frac{i}{2\pi}\frac{\dd z_1\wedge\cdz_1}{|z_1|^{2}(\log|z_1|^{-2})^{2}}
	-2CK\frac{i}{2\pi}\frac{\dd z_j\wedge\cdz_j}{(\log|z_1|^{-2})^2}
\end{equation}
also on $V^{\prime\prime}\setminus\gdiv s$.
From inequalities (\ref{eq17}) and (\ref{eq16}) we derive that
\begin{equation}\label{eq19}
	\beta\geq -\frac{1}{4}\frac{i}{2\pi}\frac{\dd z_1\wedge\cdz_1}{|z_1|^{2}(\log|z_1|^{-2})^{2}}
	-2CK\sum_{j>1}\frac{i}{2\pi}\frac{\dd z_{j}\wedge\cdz_{j}}{(\log|z_1|^{-2})^{2}}
\end{equation}
holds on $V^{\prime\prime}\setminus \gdiv s$.\\
\\
\textit{Step 3.} To conclude, add up (\ref{eq18}) and (\ref{eq19}) to find
\begin{equation}\label{eq20}
	\ddc(-\theta_{1})\geq \frac{1}{4}\frac{i}{2\pi}\frac{\dd z_{1}\wedge\cdz_{1}}{|z_1|^{2}(\log|z_1|^{-2})^{2}}
	-2CK\sum_{j>1}\frac{\dd z_j\wedge\cdz_j}{(\log|z_1|^{-2})^{2}}+\gamma.
\end{equation}
on $V^{\prime\prime}\setminus \gdiv s$. The last two terms in (\ref{eq20}) define a smooth differential form on $V\setminus \gdiv s$, bounded on $V^{\prime\prime}\setminus \gdiv s$. Hence, there exists a constant $A>0$ such that
\begin{equation}\label{eq20bis}
	\ddc(-\theta_1)+
	A\omega\geq\frac{1}{4}\frac{i}{2\pi}
	\frac{\dd z_{1}\wedge\cdz_1}{|z_1|^{2}(\log|z_1|^{-2})^{2}}
\end{equation}
holds on $V^{\prime\prime}\setminus\gdiv s$. Because $\theta_1$ is smooth away from $\gdiv s$ and $\overline{V^{\prime}\setminus V^{\prime\prime}}$ is compact and disjoint from $\gdiv s$, after possibly increasing $A$ inequality (\ref{eq20bis}) holds on $V^{\prime}\setminus\gdiv s$ as well, as was to be shown.
\end{proof}
\subsubsection{Global results}
We keep the hypothesis and notations of \textsection \ref{statement_decomp_thm}.
\begin{proposition}\label{prop_bound_theta_N}
Suppose that $X$ is compact and let $\omega$ be a smooth positive (1,1)-form on $X$. Let $L$ be a line bundle on $X$ admitting a global section $s\in\Gamma(X,L)$ such that $\gdiv s$ is a divisor with normal crossings. Let $\|\cdot\|$ be a smooth hermitian metric on $L$ such that $\|s\|^{2}\leq e^{-e}$. Define $\theta_{N}=(\log\log\|s\|^{-2})^{N}$ for any $N\in\Int_{\geq 0}$. Then there exists a constant $A=A(N)>0$ such that $\ddc(-\theta_N)+A\omega\geq 0$ on $X\setminus\gdiv s$.
\end{proposition}
\begin{proof}
The case $N=0$ is trivial. We treat the case $N\geq 1$. By Lemma \ref{lemma_theta_N} the following identity holds:
\begin{displaymath}
	\ddc(-\theta_N)=\frac{\theta_1-N+1}{N\theta_N}\frac{i}{2\pi}\pd\theta_N\wedge\cpd\theta_N
	-N\theta_{N-1}\frac{\c1(\overline{L})}{\log\|s\|^{-2}}.
\end{displaymath}
First of all, the function $\theta_{N-1}/\log\|s\|^{-2}$ is bounded and the differential form $\c1(\overline{L})$ is smooth on $X$. Since $X$ is compact, there exists a constant $A>0$ such that
\begin{equation}\label{eq10}
	-N\theta_{N-1}\frac{\c1(\overline{L})}{\log\|s\|^{-2}}+\frac{A}{2}\omega\geq 0\,\text{ on }\, X\setminus\gdiv s
\end{equation}
Still by the compactness hypothesis and by the very definition of $\theta_{1}$, there exists an open neighborhood $V$ of $\gdiv s$ such that $\theta_{1\mid V}\geq N-1$. Moreover $i\pd\theta_{N}\wedge\cpd\theta_{N}\geq 0$, so that
\begin{equation}\label{eq11}
	\frac{\theta_{1}-N+1}{N\theta_{N}}\frac{i}{2\pi}\pd\theta_{N}\wedge\cpd\theta_{N}\geq 0\,\text{ on }\, V\setminus\gdiv s.
\end{equation}
Finally, since $\theta_{M}\geq 1$ is smoooth on $X\setminus \gdiv s$ for every integer $M\geq 0$ and $X\setminus V$ is compact, after possibly increasing $A$ we have
\begin{equation}\label{eq12}
	\frac{\theta_{1}-N+1}{N\theta_{N}}\frac{i}{2\pi}\pd\theta_{N}\wedge\cpd\theta_{N}
	+\frac{A}{2}\omega\geq 0\,\text{ on }\, X\setminus V.
\end{equation}
Equations (\ref{eq10}), (\ref{eq11}) and (\ref{eq12}) together give the desired positivity property.\\
\end{proof}
\begin{corollary}\label{corol_bound_Theta_N}
Suppose that $X$ is compact and let $\omega$ be a smooth positive (1,1)-form on $X$. For every integer $N\geq 0$ there exists $A=A(N)>0$ such that $\ddc(-\Theta_N)+A\omega\geq 0$ holds on $X\setminus D$.
\end{corollary}
\begin{proof}
This follows from Proposition \ref{prop_bound_theta_N} by writing $\Theta_N=\sum_{k=1}^{m}\theta_{N}^{(k)}$, with the notation $\theta_{N}^{(k)}=(\log\log\|s_k\|_k^{-2})^{N}$.
\end{proof}
The last proposition of this subsection provides a first approach to Theorem \ref{decomp_thm}. We may thus place under the hypothesis and notations therein.
\begin{proposition}\label{prop_bound_Theta_1}
Suppose that $X$ is compact and let $\omega$ be a smooth positive (1,1)-form on $X$. For every finite covering $\lbrace (V_{\alpha};z_{1}^{\alpha},\ldots,z_{n}^{\alpha})\rbrace_{\alpha}$ of $X$ by analytic charts adapted to $D$, together with relatively compact open subsets $V^{\prime}_{\alpha}\subset\subset V_{\alpha}$ still forming a covering, there exists a constant $A>0$ such that
\begin{displaymath}
	\ddc(-\Theta_1)+A\omega\geq\frac{1}{4}\sum_{k=1}^{m_{\alpha}}\frac{i}{2\pi}
	\frac{\dd z_{k}^{\alpha}\wedge\cdz_{k}^{\alpha}}{|z_k^{\alpha}|^{2}(\log|z_k^{\alpha}|^{-2})^{2}}
\end{displaymath}
holds on $V^{\prime}_{\alpha}\setminus D$ for every $\alpha$.
\end{proposition}
\begin{proof}
This follows immediately from Proposition \ref{prop_local}.
\end{proof}
\subsection{Proof of Theorem \ref{decomp_thm}}\label{subsection_proof_thm}
We now complete the proof of Theorem \ref{decomp_thm}.

Since $X$ is compact, we can choose a finite open covering $V_{\alpha}$ of $X$ as in Proposition \ref{prop_bound_Theta_1}. It is enough to prove the existence of constants $A,B,N$ fulfilling (\ref{main_inequality}) on a single $V^{\prime}_{\alpha}\subset\subset V_{\alpha}$. We write $\lbrace z_i\rbrace_i$ for the coordinates on $V_{\alpha}$, instead of $\lbrace z_i^{\alpha}\rbrace_i$. Following Notation \ref{not_dzeta}, we develop
\begin{displaymath}
	\eta=\sum_{j=1}^{n}h_{jj}\frac{i}{2\pi}\dd\zeta_j\wedge\cdzeta_{j}
	+\sum_{j<k}2\Real(h_{jk}\frac{i}{2\pi}\dd\zeta_{j}\wedge\cdzeta_{k}),
\end{displaymath}
where the functions $h_{jk}$, $j\leq k$, have log-log growth along $D\cap V_{\alpha}$. There exist a constant $C>0$ and an integer $N\geq 0$ such that
\begin{displaymath}
	|h_{jk\mid V_{\alpha}^{\prime}\setminus D}|\leq C\Theta_{N}.
\end{displaymath}
Therefore, by Lemma \ref{lemma_ineq_diff_form}, on $V^{\prime}_{\alpha}\setminus D$ there is a lower bound
\begin{displaymath}
	\eta\geq-C\Theta_{N}\sum_{j=1}^{n}\frac{i}{2\pi}\dd\zeta_{j}\wedge\cdzeta_{j}
	-C\Theta_{N}\sum_{j<k}\left(\frac{i}{2\pi}\dd\zeta_{j}\wedge\cdzeta_{j}+\frac{i}{2\pi}\dd\zeta_{k}\wedge\cdzeta_{k}\right).
\end{displaymath}
From this inequality and Proposition \ref{prop_bound_Theta_1} we see that there exist $B\geq 1$ and a smooth differential form $\sigma$ on $V_{\alpha}$ such that
\begin{displaymath}
	\eta+B\Theta_{N}(\ddc(-\Theta_1)+A\omega)\geq \Theta_{N}\sigma
\end{displaymath}
holds on $V^{\prime}_{\alpha}\setminus D$. After possibly increasing $A$, we also have $\sigma+A\omega\geq 0$ on $V^{\prime}_{\alpha}$. Since $B\geq 1$, we finally find
\begin{displaymath}
	\eta+B\Theta_{N}(\ddc(-\Theta_1)+2A\omega)\geq 0
\end{displaymath}
on $V^{\prime}_{\alpha}\setminus D$, as was to be shown. Observe that if $\eta$ has Poincar\'e growth, then we can choose $N=0$.
\section{Bounding height integrals}\label{section_bounding_height_integrals}
\subsection{Geometric assumptions and statement of the theorem}\label{geometric_assumptions}
Let $X$ be a complex analytic manifold and $D\subset X$ a divisor with normal crossings. Let $L$ be a line bundle on $X$ and $\|\cdot\|$ a \textit{pre-log-log hermitian metric} on $L$, with singularities along $D$ (see Definition \ref{def_good_vb}). If $\|\cdot\|_{0}$ is any smooth hermitian metric on $L$, then we can write
\begin{displaymath}
	\|\cdot\|=e^{-f/2}\|\cdot\|_{0},
\end{displaymath}
where $f:X\setminus D\rightarrow\R$ is a pre-log-log function. If $\|\cdot\|$ is good along $D$, then $f$ is P-singular along $D$. As usual, abbreviate $\overline{L}=(L,\|\cdot\|)$ and $\overline{L}_{0}=(L,\|\cdot\|_0)$.\\
Suppose now that $Y\subset X$ is a \textit{compact} complex analytic submanifold of pure dimension $d$. We assume the following conditions are fulfilled:
\begin{itemize}
	\item[\textit{i}.] the submanifold $Y$ meets $D$ in a divisor with normal crossings $E$ in $Y$;	\item[\textit{ii}.] the restriction $\omega:=\c1(\overline{L}_0)_{\mid Y}$ is semi-positive and
		\begin{displaymath}
			\deg_{L}Y:=\int_{Y}\omega^{d}>0;
		\end{displaymath}
	\item[\textit{iii.}] there exists a global section $s\in\Gamma(Y,L)$ such that $\|s\|_{0}^{2}\leq e^{-e}$ on $Y$ and $\gdiv s$ is a divisor with normal crossings containing $E$. In particular $\gdiv s$ is reduced, so that we may indistinctly treat $\gdiv s$ as a reduced Weil divisor or a reduced scheme. For every integer $N\geq 0$, we define the pre-log-log function, with singularities along $\gdiv s$,
		\begin{displaymath}
			\ell_{N}=(\log\log\|s\|_{0}^{-2})^{N}:Y\setminus\gdiv s\longrightarrow\R_{\geq 1};
		\end{displaymath}
	\item[\textit{iv.}] there exist pre-log-log functions, with singularities along $E$,
		\begin{displaymath}
			\Theta_{1},\Theta_{N},\varphi,\psi:Y\setminus E\longrightarrow\R_{\geq 0}\quad (N\in\Int_{\geq 0}\,\,\text{depending on}\,\,f)
		\end{displaymath}
		with $f=\varphi-\psi$, and bounds
		\begin{displaymath}
			\varphi\leq C\ell_{M}, \psi\leq C\ell_{M}, \Theta_{1}\leq C\ell_{1},\Theta_{N}\leq C\ell_{N},
		\end{displaymath}
	for some constant $C\geq 0$ and integer $M\geq 0$. Moreover, if $f$ is P-singular, we suppose that $M=1$ and $N=0$;
	\item[\textit{v.}] there exists $A>0$ such that $\tau:=\ddc(-\Theta_1)+A\omega\geq 0$, and for every integer $Q\geq 0$ there exists $A_{Q}>0$ such that $\tau_{Q}:=\ddc(-\ell_{Q})+A_{Q}\omega\geq 0$. For $Q=0$, we can choose $A_{0}=1$, so that $\tau_0=\omega$;
	\item[\textit{vi.}] there exists $B>0$ such that
		\begin{displaymath}
			\begin{split}
				\ddc (-\varphi)+B\Theta_{N}\tau\geq 0,\\
				\ddc(-\psi)+B\Theta_{N}\tau\geq 0
			\end{split}
		\end{displaymath}
		hold on $Y\setminus E$ (and so on $Y\setminus\gdiv s$). Observe that by the bounds in \textit{iv}, we then have
		\begin{displaymath}
			\begin{split}
				\omega_{\varphi}:=\ddc(-\varphi)+BC\ell_{N}\tau\geq 0,\\
				\omega_{\psi}:=\ddc(-\psi)+BC\ell_{N}\tau\geq 0
			\end{split}
		\end{displaymath}
		on $Y\setminus\gdiv s$.
\end{itemize}
The aim of this section is to find bounds for the \textit{height integrals}
\begin{displaymath}
	J_{p}:=\int_{Y} f\c1(\overline{L})^{p}\c1(\overline{L}_{0})^{d-p},\quad 0\leq p\leq d.
\end{displaymath}
We observe that $f\c1(\overline{L})^{p}\c1(\overline{L}_{0})^{d-p}_{\mid Y}$ is a pre-log-log differential form on $Y$ with singularities along $E$, hence locally integrable (see Proposition \ref{prop_prop_log}). In particular, since $E\subseteq\gdiv s$ and $\gdiv s$ is Lebesegue negligible, the integrals $J_{p}$ can be computed on $Y\setminus\gdiv s$.
\begin{theorem}\label{thm_height_integral}
There exist constants $\alpha,\beta>0$, $R\in\Int_{\geq 0}$, depending only on $A$, $\lbrace A_{N}\rbrace_{N}$, $B$, $C$, $M$ and $N$, such that, for any $p\in\lbrace 0,\ldots,d\rbrace$,
\begin{displaymath}
	|J_{p}|\leq\alpha\deg_{L}Y +\beta\cdot(\deg_{L}Y)\cdot\log^{R}\left(\int_{Y}\log\|s\|_{0}^{-2}
	\frac{\c1(\overline{L}_{0})^{d}}{\deg_{L}Y}\right).
\end{displaymath}
If $\|\cdot\|$ is good along $E$, then we can take $R=1$, so that
\begin{displaymath}
	|J_{p}|\leq\alpha\deg_{L}Y +\beta\cdot(\deg_{L}Y)\cdot\log\left(\int_{Y}\log\|s\|_{0}^{-2}
	\frac{\c1(\overline{L}_{0})^{d}}{\deg_{L}Y}\right).
\end{displaymath}
\end{theorem}
The theorem will be reduced to the bounds claimed by the following two propositions.
\begin{proposition}\label{prop_height_int1}
Let $\sigma$ be a closed, real and semi-positive pre-log-log (t,t)-form on $Y$, with singularities along $\gdiv s$. Let $a,b$ be integers such that $a+b+t=d$. If $a,b\geq 0$, define the integral
\begin{displaymath}
	I(M,a,b,\sigma)=\int_{Y}\ell_{M}\omega_{\varphi}^{a}\omega_{\psi}^{b}\sigma.
\end{displaymath}
Otherwise set $I(M,a,b,\sigma)=0$. Then, if $a>0$, the following bound holds
\begin{displaymath}
	\begin{split}
		I(M,a,b,\sigma)&\leq BCI(M+N,a-1,b,\sigma\tau)\\
		&+CI(M,a-1,b,\sigma\tau_{M})\\
		&+bBC^{2}I(2M,a-1,b-1,\sigma\tau\tau_N)\\
		&+(a-1)BC^{2} I(2M,a-2,b,\sigma\tau\tau_N).
	\end{split}
\end{displaymath}
If $f$ is P-singular (in which case $N=0$ and $M=1$), then
\begin{displaymath}
	\begin{split}
		I(1,a,b,\sigma)&\leq ABC I(1,a-1,b,\sigma\tau_0)\\
		&+(B+1)C I(1,a-1,b,\sigma\tau_1).
	\end{split}
\end{displaymath}
Similar bounds are true if $b>0$, exchanging the role of $a$ and $b$.
\end{proposition}
\begin{proposition}\label{prop_height_int2}
Let $\sigma$ be a $(d,d)$-form which is a product of (1,1)-forms of the kind $\tau$ or $\tau_{Q}$, $Q\geq 0$. Let $W$ be the set of integers $Q\geq 0$ such that $\tau_{Q}$ appears in $\sigma$. Fix an integer $K\geq 0$. Then there exist constants $\alpha,\beta>0$ and an integer $R\geq 0$, depending only on $A$, $\lbrace A_{N}\rbrace_{N}$, $C$, $K$ and $W$, such that
\begin{displaymath}
		0\leq\int_{Y}\ell_{K}\sigma
		\leq\alpha\deg_{L}Y
		+\beta(\deg_{L}Y)\log^{R}\left(\int_{Y}\log\|s\|_{0}^{-2}\frac{\c1(\overline{L}_{0})^{d}}{\deg_{L}Y}\right).
\end{displaymath}
If $K=1$ and $\sigma$ is a product of differential forms of the kind $\tau_0$, $\tau_1$ (i.e. $W\subseteq\lbrace 0,1\rbrace$), then we can take $R=1$.
\end{proposition}
Assuming for the moment the propositions, we prove Theorem \ref{thm_height_integral}.
\begin{proof}[Proof of Theorem \ref{thm_height_integral}]
We first observe that $\c1(\overline{L})=\ddc f + \c1(\overline{L}_0)$, so that
\begin{equation}\label{eq21}
	J_p=\sum_{j=0}^{p}\binom{p}{j}\int_{Y}f(\ddc f)^{j}\omega^{d-j}.
\end{equation}
Next, on $Y\setminus \gdiv s$ we write $f=\varphi-\psi$ and $\ddc f=\omega_{\psi}-\omega_{\varphi}$. We get
\begin{displaymath}
	J_{p}=\sum_{j=0}^{p}\binom{p}{j}\sum_{k=0}^{j}(-1)^{j-k}\binom{j}{k}\left(\int_{Y}\varphi\omega_{\psi}^{k}\omega_{\varphi}^{j-k}\omega^{d-j}
	-\int_{Y}\psi\omega_{\psi}^{k}\omega_{\varphi}^{j-k}\omega^{d-j}\right).
\end{displaymath}
The coefficients $\binom{p}{j}\binom{j}{k}$ can be bounded in terms of $\dim X$ (hence independently of $Y$ and the hermitian line bundles). Therefore we are reduced to bound integrals of the kind
\begin{displaymath}
	\int_{Y}\varphi\omega_{\varphi}^{a}\omega_{\psi}^{b}\omega^{c}, \int_{Y}\psi\omega_{\varphi}^{a}\omega_{\psi}^{b}\omega^{c}.
\end{displaymath}
for integers $a,b,c\geq 0$ such that $a+b+c=d$.  Since the differential forms $\omega_{\varphi}^{a}\omega_{\psi}^{b}\omega^{c}$ are semi-positive and $0\leq\varphi,\psi\leq C\ell_{M}$, we have to find upper bounds for the integrals
\begin{displaymath}
	I(M,a,b,\omega^{c})=\int_{Y}\ell_{M}\omega_{\varphi}^{a}\omega_{\psi}^{b}\omega^{c}.
\end{displaymath}
Successively applying Proposition \ref{prop_height_int1}, we reduce our problem to find bounds for integrals $\int_{Y}\ell_{K}\sigma$, where $\sigma$ is a product of (1,1)-forms of type $\tau$ or $\tau_{Q}$, for some integers $K$, $Q\geq 0$. If $f$ is P-singular, then $K=1$ and $\sigma$ is a product of forms of type $\tau_0$ and $\tau_1$. We conclude by Proposition \ref{prop_height_int2}.
\end{proof}
\subsection{Proofs of Proposition \ref{prop_height_int1} and Proposition \ref{prop_height_int2}}
We proceed to prove Proposition \ref{prop_height_int1} and Proposition \ref{prop_height_int2}. The proofs make an extensive use of Stokes' theorem for pre-log-log differential forms. We refer to Proposition \ref{prop_prop_log} for the statement and references. To simplify the exposition, it will be worth having at our disposal the computations summarized in the next lemma.
\begin{lemma}\label{lemma_computation_omega}
Let $a,b\geq 0$ be integers. On $Y\setminus E$ the following equalities hold:
\begin{itemize}
	\item[i.] $$\pd(\omega_{\varphi}^{a}\omega_{\psi}^{b})=aBC(\pd\ell_{N})\omega_{\varphi}^{a-1}\omega_{\psi}^{b}\tau+bBC(\pd\ell_{N})\omega_{\varphi}^{a}
	\omega_{\psi}^{b-1}\tau;$$
	\item[ii.] \begin{displaymath}
		\begin{split}
			\pd\cpd(\omega_{\varphi}^{a}\omega_{\psi}^{b})=& aBC(\pd\cpd\ell_{N})\omega_{\varphi}^{a-1}\omega_{\psi}^{b}\tau\\
			-&aBC(\cpd\ell_{N})\pd(\omega_{\varphi}^{a-1}\omega_{\psi}^{b})\tau\\
			+&bBC(\pd\cpd\ell_{N})\omega_{\varphi}^{a}\omega_{\psi}^{b-1}\tau\\
			-&bBC(\cpd\ell_{N})\pd(\omega_{\varphi}^{a}\omega_{\psi}^{b-1})\tau.
		\end{split}
	\end{displaymath}
\end{itemize}
\end{lemma}
\begin{proof}
It is enough to apply the definition of $\omega_{\varphi}$, $\omega_{\psi}$, Leibniz' rule and the fact that $\ddc(-\varphi)$, $\ddc(-\psi)$ and $\tau$ are $\pd$ and $\cpd$-closed. 
\end{proof}
\begin{proof}[Proof of Proposition \ref{prop_height_int1}]
Write $I=I(M,a,b,\sigma)$. We suppose that $a>0$ and $b\geq 0$. We decompose
\begin{displaymath}
	\omega_{\varphi}^{a}=\ddc(-\varphi)\omega_{\varphi}^{a-1}+BC\ell_{N}\tau\omega_{\varphi}^{a-1}.
\end{displaymath}
Accordingly, the integral $I$ decomposes as $I=I_{1}+BCI_{2}$, with
\begin{displaymath}
	\begin{split}
		I_{1}&=\int_{Y}\ell_{M}\ddc(-\varphi)\omega_{\varphi}^{a-1}
		\omega_{\psi}^{b}\sigma,\\
		I_{2}&=\int_{Y}\ell_{M+N}\omega_{\varphi}^{a-1}\omega_{\psi}^{b}\sigma\tau.
	\end{split}
\end{displaymath}
\textit{Bounding $I_{1}$}. To get bounds on $I_1$ we apply Stokes' theorem for pre-log-log forms. For this, we first recall that $\sigma$ is closed of degree $(t,t)$ and $\omega_{\varphi}$, $\omega_{\psi}$ are of degree $(1,1)$. Then, by Leibniz' rule, we compute
\begin{displaymath}
	\begin{split}
		d\left[\ell_{M}\frac{i}{2\pi}\cpd(-\varphi)\omega_{\varphi}^{a-1}
		\omega_{\psi}^{b}\sigma\right]=&
		\ell_{M}\ddc(-\varphi)\omega_{\varphi}^{a-1}\omega_{\psi}^{b}\sigma\\
		+&(\pd\ell_{M})\frac{i}{2\pi}\cpd(-\varphi)\omega_{\varphi}^{a-1}
		\omega_{\psi}^{b}\sigma\\
		-&\ell_{M}\frac{i}{2\pi}\cpd(-\varphi)\pd(\omega_{\varphi}^{a-1}\omega_{\psi}^{b})
		\sigma,
	\end{split}
\end{displaymath}
where we used that $\ddc=i\pd\cpd/2\pi$. By Stokes' theorem, we find $I_{1}=I_{1,1}+I_{1,2}$, where
\begin{displaymath}
	\begin{split}
		I_{1,1}&=-\frac{i}{2\pi}\int_{Y}(\pd\ell_{M})\cpd(-\varphi)\omega_{\varphi}^{a-1}\omega_{\psi}^{b}
		\sigma,\\
		I_{1,2}&=\frac{i}{2\pi}\int_{Y}\ell_{M}\cpd(-\varphi)\pd(\omega_{\varphi}^{a-1}
		\omega_{\psi}^{b})\sigma.
	\end{split}
\end{displaymath}
\textit{Bounding $I_{1,1}$}. Again we apply Stokes' theorem. By Lebniz' rule we have
\begin{displaymath}
	\begin{split}
		d\left[-\frac{i}{2\pi}(-\varphi)(\pd\ell_{M})\omega_{\varphi}^{a-1}
		\omega_{\psi}^{b}\sigma\right]=&
		\frac{i}{2\pi}(\pd\ell_{M})\cpd(-\varphi)\omega_{\varphi}^{a-1}
		\omega_{\psi}^{b}\sigma\\
		+&\frac{i}{2\pi}(-\varphi)(\pd\cpd\ell_{M})\omega_{\varphi}^{a-1}
		\omega_{\psi}^{b}\sigma\\
		+&\frac{i}{2\pi}(-\varphi)(\pd\ell_{M})\cpd(\omega_{\varphi}^{a-1}\omega_{\psi}^{b})\sigma.
	\end{split}
\end{displaymath}
Therefore, by Stokes' theorem, we get $I_{1,1}=I_{1,1,1}+I_{1,1,2}$, where
\begin{displaymath}
	\begin{split}
		I_{1,1,1}&=\int_{Y}\varphi\ddc(-\ell_{M})\omega_{\varphi}^{a-1}
		\omega_{\psi}^{b}\sigma,\\
		I_{1,1,2}&=\frac{i}{2\pi}\int_{Y}(-\varphi)(\pd\ell_{M})\cpd(\omega_{\varphi}^{a-1}
		\omega_{\psi}^{b})\sigma.
	\end{split}
\end{displaymath}
\textit{Bounding $I_{1,1,1}$}.
Recall that, by assumption, there exists a constant $A_{M}>0$ such that $\tau_{M}=\ddc(-\ell_{M})+A_{M}\omega\geq 0$. Then, since $\omega_{\varphi}^{a-1}\omega_{\psi}^{b}\omega\sigma$ is a semi-positive form and $0\leq\varphi\leq C\ell_{M}$, we have a bound
\begin{equation}\label{eq:bound1}
	\begin{split}
		I_{1,1,1}\leq&\int_{Y}\varphi (\ddc(-\ell_M)+A_{M}\omega)\omega_{\varphi}^{a-1}
		\omega_{\psi}^{b}\sigma\\
		\leq&C\int_{Y}\ell_{M}\omega_{\varphi}^{a-1}\omega_{\psi}^{b}\sigma\tau_{M}\\
		=&CI(M,a-1,b,\sigma\tau_{M}).
	\end{split}
\end{equation}
\textit{Bounding $I_{1,1,2}$}. To bound the integral $I_{1,1,2}$ we first appeal to Lemma \ref{lemma_computation_omega} to develop
\begin{displaymath}
	\cpd(\omega_{\varphi}^{a-1}\omega_{\psi}^{b})=(a-1)BC(\cpd\ell_{N})\omega_{\varphi}^{a-2}\omega_{\psi}^{b}\tau
	+bBC(\cpd\ell_{N})\omega_{\varphi}^{a-1}\omega_{\psi}^{b-1}\tau.
\end{displaymath}
Then we write $I_{1,1,2}=(a-1)BCI_{1,1,2}^{(1)}+bBCI_{1,1,2}^{(2)}$, where
\begin{displaymath}
	\begin{split}
		I_{1,1,2}^{(1)}&=\frac{i}{2\pi}\int_{Y}(-\varphi)(\pd\ell_{M})(\cpd\ell_{N})\omega_{\varphi}^{a-2}\omega_{\psi}^{b}
		\sigma\tau,\\
		I_{1,1,2}^{(2)}&=\frac{i}{2\pi}\int_{Y}(-\varphi)(\pd\ell_{M})(\cpd\ell_{N})\omega_{\varphi}^{a-1}
		\omega_{\psi}^{b-1}\sigma\tau.
	\end{split}
\end{displaymath}
In these integrals we observe that
\begin{equation}\label{eq22}
	i\pd\ell_{M}\wedge\cpd\ell_{N}=NM\ell_{N+M-2}i\pd\ell_{1}\wedge\cpd\ell_{1},
\end{equation}
which is a semi-positive differential form. Since $\varphi\geq 0$ and $\omega_{\varphi}$, $\omega_{\psi}$, $\sigma$ and $\tau$ are also semi-positive, we find $I_{1,1,2}^{(1)}\leq 0$ and $I_{1,1,2}^{(2)}\leq 0$. This shows
\begin{equation}\label{eq:bound2}
	I_{1,1,2}=(a-1)BCI_{1,1,2}^{(1)}+bBCI_{1,1,2}^{(2)}\leq 0.
\end{equation}
From (\ref{eq:bound1}) and (\ref{eq:bound2}) we conclude with the bound
\begin{equation}\label{eq:bound3}
	I_{1,1}=I_{1,1,1}+I_{1,1,2}\leq CI(M,a-1,b,\sigma\tau_{M}).
\end{equation}
\textit{Bounding $I_{1,2}$}. As before we proceed by successive applications of Stokes' theorem. First of all we compute
\begin{displaymath}
	\begin{split}
		\dd\left[\frac{i}{2\pi}\ell_{M}(-\varphi)\pd(\omega_{\varphi}^{a-1}\omega_{\psi}^{b})
		\sigma\right]&=
		\frac{i}{2\pi}\ell_{M}\cpd(-\varphi)\pd(\omega_{\varphi}^{a-1}\omega_{\psi}^{b})\sigma\\
		&+\frac{i}{2\pi}(-\varphi)(\cpd\ell_{M})\pd(\omega_{\varphi}^{a-1}\omega_{\psi}^{b})\sigma\\
		&+\frac{i}{2\pi}(-\varphi)\ell_{M}\cpd\pd(\omega_{\varphi}^{a-1}\omega_{\psi}^{b})\sigma.
	\end{split}
\end{displaymath}
By Stokes' theorem we find $I_{1,2}=I_{1,2,1}+I_{1,2,2}$, with
\begin{displaymath}
	\begin{split}
		I_{1,2,1}&=\frac{i}{2\pi}\int_{Y}\varphi(\cpd\ell_{M})\pd(\omega_{\varphi}^{a-1}\omega_{\psi}^{b})
		\sigma\\
		I_{1,2,2}&=\frac{i}{2\pi}\int_{Y}(-\varphi)\ell_{M}\pd\cpd(\omega_{\varphi}^{a-1}\omega_{\psi}^{b})
		\sigma.
	\end{split}
\end{displaymath}
\textit{Bounding $I_{1,2,1}$}. By Lemma \ref{lemma_computation_omega} we get the expansion
\begin{displaymath}
	\pd(\omega_{\varphi}^{a-1}\omega_{\psi}^{b})=(a-1)BC(\pd\ell_{N})\omega_{\varphi}^{a-2}\omega_{\psi}^{b}\tau
	+bBC(\pd\ell_{N})\omega_{\varphi}^{a-1}\omega_{\psi}^{b-1}\tau.
\end{displaymath}
Accordingly, we decompose $I_{1,2,1}=(a-1)BCI_{1,2,1}^{(1)}+bBCI_{1,2,1}^{(2)}$, where
\begin{displaymath}
	\begin{split}
		I_{1,2,1}^{(1)}&=-\frac{i}{2\pi}\int_{Y}\varphi(\pd\ell_{N})(\cpd\ell_{M})\omega_{\varphi}^{a-2}
		\omega_{\psi}^{b}\sigma\tau\\
		I_{1,2,1}^{(2)}&=-\frac{i}{2\pi}\int_{Y}\varphi(\pd\ell_{N})(\cpd\ell_{M})\omega_{\varphi}^{a-1}
		\omega_{\psi}^{b-1}\sigma\tau.
	\end{split}
\end{displaymath}
By (\ref{eq22}) the differential form $i\pd\ell_{N}\wedge\cpd\ell_{M}$ is semi-positive. Since $\varphi\geq 0$ and $\omega_{\varphi}$, $\omega_{\psi}$, $\sigma$ and $\tau$ are also semi-positive, we have $I_{1,2,1}^{(1)},I_{1,2,1}^{(2)}\leq 0$. This proves
\begin{equation}\label{eq:bound4}
	I_{1,2,1}=(a-1)BCI_{1,2,1}^{(1)}+bBCI_{1,2,1}^{(2)}\leq 0.
\end{equation}
\textit{Bounding $I_{1,2,2}$}. To bound the integral $I_{1,2,2}$ we first recall from Lemma \ref{lemma_computation_omega}
\begin{displaymath}
	\begin{split}
		\pd\cpd(\omega_{\varphi}^{a-1}\omega_{\psi}^{b})&=
		(a-1)BC(\pd\cpd\ell_{N})\omega_{\varphi}^{a-2}\omega_{\psi}^{b}\tau\\
		&-(a-1)BC(\cpd\ell_{N})\pd(\omega_{\varphi}^{a-2}\omega_{\psi}^{b})\tau\\
		&+bBC(\pd\cpd\ell_{N})\omega_{\varphi}^{a-1}\omega_{\psi}^{b-1}\tau\\
		&-bBC(\cpd\ell_{N})\pd(\omega_{\varphi}^{a-1}\omega_{\psi}^{b-1})\tau.
	\end{split}
\end{displaymath}
Corresponding to this expansion, we write $I_{1,2,2}=(a-1)BCI_{1,2,2}^{(1)}+(a-1)BCI_{1,2,2}^{(2)}+bBCI_{1,2,2}^{(3)}+bBCI_{1,2,2}^{(4)}$, with the obvious notations for the integrals $I_{1,2,2}^{(j)}$ (see below).\\
\textit{Bounding $I_{1,2,2}^{(1)}$}. We have
\begin{displaymath}
	I_{1,2,2}^{(1)}=\int_{Y}\varphi\ell_{M}\ddc(-\ell_{N})\omega_{\varphi}^{a-2}\omega_{\psi}^{b}
	\sigma\tau.
\end{displaymath}
Recall that for some constant $A_{N}>0$ the differential form $\tau_{N}=\ddc(-\ell_{N})+A_{N}\omega$ is semi-positive. Moreover $0\leq\varphi\leq C\ell_{M}$. Thus we find the bound
\begin{equation}\label{eq:bound5}
	\begin{split}
		I_{1,2,2}^{(1)}\leq& C\int_{Y}\ell_{2M}\omega_{\varphi}^{a-2}\omega_{\psi}^{b}\sigma\tau\tau_{N}\\
		=&CI(2M,a-2,b,\sigma\tau\tau_{N}).
	\end{split}
\end{equation}
\textit{Bounding $I_{1,2,2}^{(2)}$}. We write
\begin{displaymath}
	I_{1,2,2}^{(2)}=\frac{i}{2\pi}\int_{Y}\varphi\ell_{M}(\cpd\ell_{N})\pd(\omega_{\varphi}^{a-1}\omega_{\psi}^{b})\sigma\tau.
\end{displaymath}
By (\ref{eq22}) and Lemma \ref{lemma_computation_omega}, $i\pd(\omega_{\varphi}^{a-1}\omega_{\psi}^{b})\wedge\cpd\ell_{N}$ is semi-positive, so that
\begin{equation}\label{eq:bound6}
 	I_{1,2,2}^{(2)}\leq 0.
\end{equation}
\textit{Bounding $I_{1,2,2}^{(3)}$}. We have
\begin{displaymath}
	I_{1,2,2}^{(3)}=\int_{Y}\varphi\ell_{M}\ddc(-\ell_{N})\omega_{\varphi}^{a-1}\omega_{\psi}^{b}
	\sigma\tau.
\end{displaymath}
Reasoning as for $I_{1,2,2}^{(1)}$, we arrive to
\begin{equation}\label{eq:bound7}
	\begin{split}
		I_{1,2,2}^{(3)}\leq& C\int_{Y}\ell_{2M}\omega_{\varphi}^{a-1}\omega_{\psi}^{b-1}\sigma\tau\tau_{N}\\
		=&CI(2M,a-1,b-1,\sigma\tau\tau_{N}).
	\end{split}
\end{equation}
\textit{Bounding $I_{1,2,2}^{(4)}$}. We finally bound the integral
\begin{displaymath}
	I_{1,2,2}^{(4)}=\frac{i}{2\pi}\int_{Y}\varphi\ell_{M}(\cpd\ell_{N})\pd(\omega_{\varphi}^{a-1}
	\omega_{\psi}^{b-1})\sigma\tau.
\end{displaymath}
Again by (\ref{eq22}) and Lemma \ref{lemma_computation_omega}, the differential form $i\pd(\omega_{\varphi}^{a-1}\omega_{\psi}^{b-1})\wedge(\cpd\ell_{N})$ is semi-positive, so that
\begin{equation}\label{eq:bound8}
 	I_{1,2,2}^{(4)}\leq 0.
\end{equation}
We now conclude with a bound for $I_{1,2}$, since the inequalities (\ref{eq:bound4})--(\ref{eq:bound8}) yield
\begin{equation}\label{eq:bound9}
	\begin{split}
		I_{1,2}\leq I_{1,2,2}\leq&(a-1)BC^{2}I(2M,a-2,b,\sigma\tau\tau_{N})\\
		+&bBC^{2}I(2M,a-1,b-1,\sigma\tau\tau_{N}).
	\end{split}
\end{equation}
As for $I_{1}$, the bounds (\ref{eq:bound3}) and (\ref{eq:bound9}) lead to
\begin{equation}\label{eq:bound10}
	\begin{split}
		I_{1}=I_{1,1}+I_{1,2}\leq&CI(M,a-1,b,\sigma\tau_{M})\\
		+&bBC^{2}I(2M,a-1,b-1,\sigma\tau\tau_{N})\\
		+&(a-1)BC^{2}I(2M,a-2,b,\sigma\tau\tau_{N}).
	\end{split}
\end{equation}
\textit{Bounding $I_{2}$}. We have
\begin{equation}\label{eq:bound11}
		I_{2}=\int_{Y}\ell_{M+N}\omega_{\varphi}^{a-1}\omega_{\psi}^{b}\sigma\tau=I(M+N,a-1,b,\sigma\tau)
\end{equation}
To conclude we put (\ref{eq:bound10}) and (\ref{eq:bound11}) together and we get
\begin{displaymath}
	\begin{split}
		I=I_{1}+BCI_{2}\leq& BCI(M+N,a-1,b,\sigma\tau)\\
		+&CI(M,a-1,b,\sigma\tau_{M})\\
		+&bBC^{2}I(2M,a-1,b-1,\sigma\tau\tau_{N})\\
		+&(a-1)BC^{2}I(2M,a-2,b,\sigma\tau\tau_{N}),
	\end{split}
\end{displaymath}
as was to be shown.

Suppose now that $f$ is P-singular (so that $N=0$ and $M=1$). We can write
\begin{displaymath}
	\begin{split}
		&\omega_{\varphi}=\ddc(-\widetilde{\varphi})+ABC\omega,\\
		 &\omega_{\psi}=\ddc(-\widetilde{\psi})+ABC\omega,
	\end{split}
\end{displaymath}
where $\widetilde{\varphi}=\varphi+BC\Theta_1$ and $\widetilde{\psi}=\psi+BC\Theta_1$. The same method followed above allows to establish the bound
\begin{displaymath}
	\begin{split}
		I(1,a,b,\sigma)&\leq ABC I(1,a-1,b,\sigma\tau_0)\\
		&+(B+1)C I(1,a-1,b,\sigma\tau_1).
	\end{split}
\end{displaymath}
The details are left to the reader.
\end{proof}
\begin{proof}[Proof of Proposition \ref{prop_height_int2}]
Let $\sigma$ be a $(d,d)$-form which is a product of (1,1)-forms of type $\tau$ or $\tau_{Q}$, $Q\geq 0$. We write $\sigma=\tau^{s}\sigma_{1}$, for some $s\geq 0$ and $\sigma_1$ a product of forms of type $\tau_Q$. Define
\begin{equation}\label{eq:bound12}
	J(K,s,\sigma_1)=\int_{Y}\ell_{K}\tau^{s}\sigma_1.
\end{equation}
First of all we show how to reduce $s$ to $0$. The argument is by induction. If $s>0$, recalling that $\tau=\ddc(-\Theta_1)+A\omega$ we write 
\begin{equation}\label{eq:bound13}
	\tau^{s}=\ddc(-\Theta_1)\tau^{s-1}+A\omega\tau^{s-1}.
\end{equation}
Since $\tau_0=\omega$, we get from the definition of $J$ in (\ref{eq:bound12}) and from (\ref{eq:bound13})
\begin{equation}\label{eq:bound14}
	J(K,s,\sigma_1)=AJ(K,s-1,\sigma_1\tau_0)+\int_{Y}\ell_{K}\ddc(-\Theta_1)\tau^{s-1}\sigma_1.
\end{equation}
We next bound the integral on the right hand side of (\ref{eq:bound14}). Since $\tau$ and $\sigma_1$ are $\pd$ and $\cpd$-closed, applying Stokes' theorem for pre-log-log forms we get
\begin{equation}\label{eq:bound15}
	\int_{Y}\ell_{K}\ddc(-\Theta_1)\tau^{s-1}\sigma_1=\int_{Y}\Theta_1\ddc(-\ell_{K})\tau^{s-1}\sigma_1.
\end{equation}
Now $\ddc(-\ell_{K})=\tau_{K}-A_{K}\tau_0$, the forms $\tau_0$, $\tau_{K}$ are semi-positive and $0<\Theta_1\leq C\ell_{1}$, so that from (\ref{eq:bound15}) and the definition of $J$ we derive
\begin{equation}\label{eq:bound16}
	\int_{Y}\ell_{K}\ddc(-\Theta_1)\tau^{s-1}\sigma_1\leq CJ(1,s-1,\sigma_1\tau_{K}).
\end{equation}
Observe that because $\ell_{K}=\ell_{1}^{K}$ and $\ell_{1}\geq 1$, the inequality $\ell_{1}\leq\ell_{K}$ holds. Therefore
\begin{equation}\label{eq:bound16bis}
	J(1,s-1,\sigma_{1}\tau_{K})\leq J(K,s-1,\sigma_{1}\tau_{K}).
\end{equation}
From (\ref{eq:bound14})--(\ref{eq:bound16bis}) we arrive to
\begin{displaymath}
	J(K,s,\sigma_1)\leq AJ(K,s-1,\sigma_1\tau_0)+CJ(K,s-1,\sigma_1\tau_{K}).
\end{displaymath}
Hence we may suppose that $s=0$, so that $\sigma$ is a product $\tau_{Q_1}\ldots\tau_{Q_d}$. We have to deal with
\begin{displaymath}
	L(K,Q_1,\ldots,Q_d)=\int_{Y}\ell_{K}\tau_{Q_1}\ldots\tau_{Q_d}.
\end{displaymath}
Again by an inductive argument we show how to reduce all the integers $Q_i$ to 0. Suppose that $Q_1>0$. Then we write $\tau_{Q_1}=\ddc(-\ell_{Q_1})+A_{Q_1}\tau_0$, so that
\begin{equation}\label{eq:bound17}
	L(K,Q_1,\ldots,Q_d)=A_{Q_1}L(K,0,Q_2,\ldots,Q_d)+\int_{Y}\ell_{K}\ddc(-\ell_{Q_1})\sigma_1
\end{equation}
where $\sigma_1=\tau_{Q_2}\ldots\tau_{Q_d}$. We study the integral on the right hand side of (\ref{eq:bound17}). Because $\sigma_1$ is $\pd$ and $\cpd$-closed, by Stokes' theorem for pre-log-log forms we find
\begin{equation}\label{eq:bound18}
	\begin{split}
		L_2:=\int_{Y}\ell_{K}\ddc(-\ell_{Q_1})\sigma_1&=\frac{i}{2\pi}\int_{Y}(\pd\ell_K)(\cpd\ell_{Q_1})\sigma_1\\
		&=\frac{i}{2\pi}\int_{Y}KQ_{1}\ell_{K+Q_{1}-2}(\pd\ell_{1})(\cpd\ell_1)\sigma_1.
	\end{split}
\end{equation}
By the very definition of $\ell_1$, we have
\begin{equation}\label{eq:bound19}
	\pd\ell_1\wedge\cpd\ell_1=\frac{\pd\log\|s\|_{0}^{-2}\wedge\cpd\log\|s\|_{0}^{-2}}{(\log\|s\|_{0}^{-2})^{2}}.
\end{equation}
On the other hand, since $\|s\|_{0}^{2}\leq e^{-e}$, there exists a constant $D$ depending only on $K+Q_{1}-2$ such that
\begin{equation}\label{eq25}
	\frac{\ell_{K+Q_{1}-2}}{(\log\|s\|_{0}^{-2})^{2}}=\frac{(\log\log\|s\|_{0}^{-2})^{K+Q_{1}-2}}{(\log\|s\|_{0}^{-2})^{2}}
	\leq D\frac{1}{(\log\|s\|_{0}^{-2})^{3/2}}.
\end{equation}
Moreover observe that
\begin{equation}\label{eq26}
	\frac{1}{(\log\|s\|_{0}^{-2})^{3/2}}\pd\log\|s\|_{0}^{-2}=-2\pd(\log\|s\|_{0}^{-2})^{-1/2}.
\end{equation}
Because $i\pd\log\|s\|_{0}^{-2}\wedge\cpd\log\|s\|_{0}^{-2}$ and $\sigma_1$ are semi-positive, combining (\ref{eq:bound18})--(\ref{eq26}) we get
\begin{equation}\label{eq:bound20}
	L_2\leq-2\frac{i}{2\pi}\int_{Y}KQ_{1}D(\pd(\log\|s\|_{0}^{-2})^{-1/2})(\cpd\log\|s\|_{0}^{-2})\sigma_1.
\end{equation}
By Lemma \ref{lemma_Stokes} below, we can apply Stokes' theorem to the right hand side of (\ref{eq:bound20}) and obtain
\begin{equation}\label{eq:bound21}
	L_2\leq 2\int_{Y}KQ_{1}D(\log\|s\|_{0}^{-2})^{-1/2}\omega\sigma_{1}.
\end{equation}
By the hypothesis on $\|\cdot\|_0$, $(\log\|s\|_{0}^{-2})^{-1/2}\leq 1$. Using the positivity of $\omega\sigma_1$, from (\ref{eq:bound21}) we derive
\begin{equation}\label{eq:bound22}
	L_{2}\leq2KQ_{1}D\int_{Y}\omega\sigma_1.
\end{equation}
Now recall that $\omega$ is $\pd$ and $\cpd$-closed, and $\tau_{Q}=\ddc(-\ell_{Q})+A_{Q}\omega$, so that
\begin{displaymath}
	\omega\sigma_1=A_{Q_2}\ldots A_{Q_{d}}\omega^{d}+\ddc\sigma_2
\end{displaymath}
for some pre-log-log form $\sigma_2$. Applied to (\ref{eq:bound22}) this provides
\begin{equation}\label{eq:bound23}
	L_2\leq 2Q_{1}A_{Q_2}\ldots A_{Q_d}KD\int_{Y}\omega^{d}=2Q_{1}A_{Q_2}\cdot\ldots\cdot A_{Q_d}KD\deg_{L}Y.
\end{equation}
The identity (\ref{eq:bound17}) together with (\ref{eq:bound23}) imply the inequality
\begin{displaymath}
	L(K,Q_{1},\ldots,Q_{d})\leq A_{Q_{1}}L(K,0,Q_{2},\ldots,Q_{d})+2Q_{1}A_{Q_{2}}\cdot\ldots\cdot A_{Q_{d}}KD\gdeg_{L}Y.
\end{displaymath}
Successively repeating this argument, we reduce $Q_{2},\ldots,Q_{d}$ to 0. Hence it remains to treat the integrals
\begin{displaymath}
	M(K)=\int_{Y}\ell_{K}\omega^{d}.
\end{displaymath}
We want to apply Jensen's inequality. First of all, rewrite
\begin{displaymath}
	M(K)=(\deg_{L}Y)\int_{Y}\ell_{K}\frac{\omega^{d}}{\deg_{L}Y},
\end{displaymath}
so that $\omega^{d}/\deg_{L}Y$ defines a probability measure on $Y$. Secondly, the function $x\mapsto(\log x)^{K}$ is concave on $]e^{K-1},+\infty[$, because
\begin{displaymath}
	\frac{d^{2}}{dx^{2}}(\log x)^{K}=\frac{K}{x^{2}}(\log x)^{K-2}(K-1-\log x).
\end{displaymath}
Since $\|s\|_{0}^{-2}\geq e^{e}$, in particular $e^{K+1}\log\|s\|_{0}^{-2}>e^{K-1}$. We use that $\log^{K}$ is an increasing function and apply Jensen's inequality: 
\begin{displaymath}
	\begin{split}
		M(K)\leq& (\deg_{L}Y)\int_{Y}\log^{K}(e^{K+1}\log\|s\|_{0}^{-2})\frac{\omega^{d}}{\deg_{L}Y}\\
		\leq&(\deg_{L}Y)\log^{K}\left(e^{K+1}\int_{Y}\log\|s\|_{0}^{-2}\frac{\omega^{d}}{\deg_{L}Y}\right).
	\end{split}
\end{displaymath}
By the trivial inequality $x+y\leq 2xy$ for real $x,y\geq 1$, we finally arrive to
\begin{displaymath}
	M(K)\leq(\deg_{L}Y)(2K+2)^{K}\log^{K}\left(\int_{Y}\log\|s\|_{0}^{-2}\frac{\omega^{d}}{\deg_{L}Y}\right).
\end{displaymath}
This concludes the proof of the proposition, except for the fact that we can take $R=1$ when $K=1$ and $\sigma$ is a product of forms $\tau_0$ and $\tau_1$. This last case is similarly treated and left to the reader.
\end{proof}
\begin{lemma}\label{lemma_Stokes}
Let $\mu$ be a closed pre-log-log (d-1,d-1)-form on $Y$, with singularities along $\gdiv s$. Then we have
\begin{displaymath}
	-\frac{i}{2\pi}\int_{Y}\pd\left(\frac{1}{(\log\|s\|_{0}^{-2})^{1/2}}\right)\wedge\cpd\log\|s\|_{0}^{-2}\mu=
	\int_{Y}\frac{1}{(\log\|s\|_{0}^{-2})^{1/2}}\wedge\omega\mu.
\end{displaymath}
\end{lemma}
\begin{proof}
The proof of the lemma follows the ideas of Lemma 7.36 in \cite{BKK}.

First of all, as for log-log growth differential forms, the form
\begin{displaymath}
	\pd\left(\frac{1}{(\log\|s\|_{0}^{-2})^{1/2}}\right)\wedge\cpd\log\|s\|^{-2}_{0}\mu=-\frac{1}{2}(\log\|s\|_{0}^{-2})^{1/2}\pd\ell_1\wedge\cpd\ell_1\mu
\end{displaymath}
is locally integrable on $Y$. Indeed, after localizing to an analytic chart adapted to $\gdiv s$ and changing to polar coordinates, we are reduced to point out that, for every $0<\delta<1/2$, we have an estimate
\begin{equation}\label{eq:Stokes1}
	\int_{0}^{\varepsilon/e}(\log\log t^{-1})^{N}(\log t^{-1})^{1/2}\frac{\dd t}{t(\log t)^2}\prec\int_{0}^{\varepsilon/e}\frac{\dd t}{t(\log t^{-1})^{3/2-\delta}}<+\infty.
\end{equation}
Define
\begin{displaymath}
	I=-\frac{i}{2\pi}\int_{Y}\pd\left(\frac{1}{(\log\|s\|_{0}^{-2})^{1/2}}\right)\wedge\cpd\log\|s\|_{0}^{-2}\mu.
\end{displaymath}
We construct a finite open covering $\lbrace (V_{\alpha}^{\prime};\lbrace z_{i}^{\alpha}\rbrace_{i})\rbrace_{\alpha}$ of $\gdiv s$, by adapted analytic open charts. Suppose that via the coordinates $\lbrace z_{i}^{\alpha}\rbrace_{\alpha}$, $V_{\alpha}^{\prime}$ is identified with $\Delta_{1/e}^{r}\times\Delta_{1/e}^{s}$ ($r=r(\alpha)$), so that $V_{\alpha}^{\prime}\setminus D$ corresponds to $\Delta_{1/e}^{r\ast}\times\Delta_{1/e}^{s}$. We further assume that in these coordinates we can write $s=z_{1}^{\alpha}\cdot\ldots\cdot z_{r}^{\alpha}u_{\alpha}$, where $u_{\alpha}$ is a holomorphic unit. After possibly adding a finite number of adapted analytic charts to $\lbrace (V_{\alpha}^{\prime};\lbrace z_{i}^{\alpha}\rbrace_{i})\rbrace_{\alpha}$, the open subsets $V_{\alpha}\subset\subset V_{\alpha}^{\prime}$ identified with $\Delta_{1/2e}^{r}\times\Delta_{1/2e}^{s}$ via the coordinates $\lbrace z_{i}^{\alpha}\rbrace_{i}$ still cover $\gdiv s$. Write $\Omega=\cup_{\alpha} V_{\alpha}$. Take a finite open covering $\lbrace V_{\beta}\rbrace_{\beta}$ of $X\setminus\Omega$, so that $\overline{V_{\beta}}\cap(\gdiv s)=\emptyset$ for all $\beta$. Let $\lbrace\chi_{\alpha}\rbrace_{\alpha}\cup\lbrace\chi_{\beta}\rbrace_{\beta}$ be a partition of unity subordinate to $\lbrace V_{\alpha}\rbrace_{\alpha}\cup\lbrace V_{\beta}\rbrace_{\beta}$, with $\chi_{\gamma}$ vanishing outside $V_{\gamma}$ for all $\gamma=\alpha,\beta$. We can expand
\begin{displaymath}
	I=\sum_{\gamma}I_{\gamma},
\end{displaymath}
where
\begin{displaymath}
	I_{\gamma}:=-\frac{i}{2\pi}\int_{Y}\pd\left(\frac{1}{(\log\|s\|_{0}^{-2})^{1/2}}\right)\wedge\cpd\log\|s\|_{0}^{-2}\chi_{\gamma}\mu.
\end{displaymath}
We first treat the integrals $I_{\beta}$. Observe that for any $\CC^{\infty}$ differential form $\nu$ on $Y\setminus\gdiv s$, the differential form $\chi_{\beta}\nu$ is $\CC^{\infty}$ on $Y$, because $\chi_{\beta}$ vanishes on $Y\setminus V_{\beta}$ and $\overline{V_{\beta}}\cap(\gdiv s)=\emptyset$. Moreover the equality $\dd(\chi_{\beta}\nu)=\dd\chi_{\beta}\wedge\nu+\chi_{\beta}\wedge\dd\nu$ holds on $Y$. For $\nu=i\pd\cpd\log\|s\|_{0}^{-2}/2\pi$ we find $\chi_{\beta}\nu=\chi_{\beta}\omega$ on $Y$. These observations, together with $\dd\mu=0$, yield, by Stokes' theorem,
\begin{equation}\label{eq:Stokes2}
	\begin{split}
		I_{\beta}=&
		\int_{Y}\frac{1}{(\log\|s\|_{0}^{-2})^{1/2}}\omega\chi_{\beta}\mu\\
		-&\int_{Y}\frac{i}{2\pi}\frac{1}{(\log\|s\|_{0}^{-2})^{1/2}}\cpd\log\|s\|_{0}^{-2}(\dd\chi_{\beta})\mu,
	\end{split}
\end{equation}
On the other hand, for every $\alpha$ define $B_{\varepsilon}^{\alpha}(\gdiv s)$ to be the $\varepsilon$-neighborhood of $\gdiv s$ given by
\begin{displaymath}
	B_{\varepsilon}^{\alpha}(\gdiv s)=\bigcup_{k=1}^{r}B_{\varepsilon}^{\alpha}(T_{k}),
\end{displaymath}
where $T_{k}$ is the divisor $z_{k}^{\alpha}=0$ in $V_{\alpha}$ and
\begin{displaymath}
	B_{\varepsilon}^{\alpha}(T_{k})=\Delta_{1/2e}^{k-1}\times\Delta_{\varepsilon}\times\Delta_{1/2e}^{r-k}\times\Delta_{1/2e}^{s}\subset V_{\alpha}.
\end{displaymath}
Then we write $I_{\alpha}=\lim_{\varepsilon\to 0}I_{\alpha,\varepsilon}$, where $I_{\alpha,\varepsilon}=I_{\alpha,\varepsilon}^{(1)}+I_{\alpha,\varepsilon}^{(2)}+I_{\alpha,\varepsilon}^{(3)}$ and
\begin{displaymath}
	\begin{split}
		I_{\alpha,\varepsilon}^{(1)}:=&-\frac{i}{2\pi}\int_{Y\setminus \overline{B_{\varepsilon}^{\alpha}(\gdiv s)}}
		\dd\left(\frac{1}{(\log\|s\|_{0}^{-2})^{1/2}}\cpd\log\|s\|_{0}^{-2}\chi_{\alpha}\mu\right)\\
		I_{\alpha,\varepsilon}^{(2)}:=&\int_{Y\setminus \overline{B_{\varepsilon}^{\alpha}(\gdiv s)}}\frac{1}{(\log\|s\|_{0}^{-2})^{1/2}}\omega\chi_{\alpha}\mu\\
		I_{\alpha,\varepsilon}^{(3)}:=&-\frac{i}{2\pi}\int_{Y\setminus \overline{B_{\varepsilon}^{\alpha}(\gdiv s)}}\frac{1}{(\log\|s\|_{0}^{-2})^{1/2}}\cpd\log\|s\|_{0}^{-2}(\dd\chi_{\alpha})\mu.
	\end{split}
\end{displaymath}
The differential form in $I_{\alpha,\varepsilon}^{(2)}$ is integrable on $Y$, since it has log-log growth along $\gdiv s$ (Proposition \ref{prop_prop_log}). The differential form in $I_{\alpha,\varepsilon}^{(3)}$ is integrable on $Y$, too. Indeed, after localizing to an analytic chart adapted to $\gdiv s$ and changing to polar coordinates, we are reduced to prove the convergence of the integrals
\begin{equation}\label{eq:Stokes3}
	\int_{0}^{1/e}\frac{(\log\log t^{-1})^{N}}{(\log t^{-1})^{1/2}}\dd t
\end{equation}
and
\begin{equation}\label{eq:Stokes4}
	\int_{0}^{1/e}(\log\log t^{-1})^{N}\frac{1}{(\log t^{-1})^{1/2}}\frac{\dd t}{t\log t^{-1}}=\int_{0}^{1/e}(\log\log t^{-1})^{N}\frac{\dd t}{t(\log t^{-1})^{3/2}}.
\end{equation}
From the boundedness of $(\log\log t^{-1})^{N}/(\log t^{-1})^{1/2}$ the convergence of (\ref{eq:Stokes3}) follows. The second one (\ref{eq:Stokes4}) has already been treated (\ref{eq:Stokes1}). Therefore $\lim_{\varepsilon\to 0} I_{\alpha,\varepsilon}^{(2)}=I_{\alpha}^{(2)}$ and $\lim_{\varepsilon\to 0} I_{\alpha,\varepsilon}^{(3)}=I_{\alpha}^{(3)}$, where
\begin{displaymath}
	\begin{split}
		I_{\alpha}^{(2)}:=&\int_{Y}\frac{1}{(\log\|s\|_{0}^{-2})^{1/2}}\omega\chi_{\alpha}\mu,\\
		I_{\alpha}^{(3)}:=&-\frac{i}{2\pi}\int_{Y}\frac{1}{(\log\|s\|_{0}^{-2})^{1/2}}\cpd\log\|s\|_{0}^{-2}(\dd\chi_{\alpha})\mu.
	\end{split}
\end{displaymath}
As for the integral $I_{\alpha,\varepsilon}^{(1)}$, after applying Stokes' theorem we find
\begin{displaymath}
	I_{\alpha,\varepsilon}^{(1)}=\frac{i}{2\pi}\int_{\pd B_{\varepsilon}^{\alpha}(\gdiv s)}\frac{1}{(\log\|s\|_{0}^{-2})^{1/2}}\cpd\log\|s\|_{0}^{-2}\chi_{\alpha}\mu.
\end{displaymath}
Taking into account that $\chi_{\alpha}$ vanishes on $\pd V_{\alpha}$, we easily see that
\begin{equation}\label{eq:Stokes5}
	|I_{\alpha,\varepsilon}^{(1)}|\leq \sum_{k=1}^{r}\frac{1}{2\pi}\int_{\pd^{*} B_{\varepsilon}^{\alpha}(T_{k})}\frac{1}{(\log\|s\|_{0}^{-2})^{1/2}}\left|\cpd\log\|s\|_{0}^{-2}\chi_{\alpha}\mu\right|,
\end{equation}
with the notation
\begin{displaymath}
	\pd^{*} B_{\varepsilon}^{\alpha}(T_{k})=\overline{\Delta_{1/2e}^{k-1}}\times\pd\Delta_{\varepsilon}\times\overline{\Delta_{1/2e}^{r-k}}\times\overline{\Delta_{1/2e}^{s}}.
\end{displaymath}
Observe that $\pd^{*} B_{\varepsilon}^{\alpha}(T_k)$ is fibered in circles over $T_k$, via the projection
\begin{displaymath}
	\begin{split}
		p_{k,\varepsilon}^{\alpha}:\pd^{*}B_{\varepsilon}(T_k)&\longrightarrow T_k\\
		(z_{1}^{\alpha},\ldots,z_{d}^{\alpha})&\longmapsto (z_{1}^{\alpha},\ldots,z_{k-1}^{\alpha},0,z_{k+1}^{\alpha}\ldots,z_{d}^{\alpha}).
	\end{split}		
\end{displaymath}
By $\int_{p_{k,\varepsilon}^{\alpha}}$ we will mean integral along the fibers of $p_{k,\varepsilon}^{\alpha}$. We claim that, for every $k=1,\ldots,r$,
\begin{equation}\label{eq:Stokes6}
	\lim_{\varepsilon\to 0}\int_{T_k}\int_{p_{k,\varepsilon}^{\alpha}}\frac{1}{(\log\|s\|_{0}^{-2})^{1/2}}\left|\cpd\log\|s\|_{0}^{-2}\chi_{\alpha}\mu\right|=0.
\end{equation}
Indeed, write $s=z_1^{\alpha}\ldots z_r^{\alpha} u_{\alpha}$ on $V^{\prime}_{\alpha}$, where $u_{\alpha}$ is a holomorphic unit. Since $\chi_{\alpha}\mu$ is a $(d-1,d-1)$ pre-log-log form, the differential form to integrate under $\int_{p_{k,\varepsilon}^{\alpha}}$ has the shape
\begin{displaymath}
	f(z_1^{\alpha},\ldots,z_{d}^{\alpha})\frac{(\log\log|z_{k}^{\alpha}|^{-1})^{M}}{(\log\|s\|^{-2}_{0})^{1/2}}\left|\frac{\cdz_k^{\alpha}}{\bar{z}_k^{\alpha}}\wedge\prod_{j\neq k}\dd z_j^{\alpha}\wedge\cdz_j^{\alpha}\right|,
\end{displaymath}
where $f$ behaves as follows:
\begin{displaymath}
	0\leq f(z_1^{\alpha},\ldots,z_{d}^{\alpha})\prec\prod_{j\neq k}\frac{(\log\log|z_j^{\alpha}|^{-1})^{M}}{|z_j^{\alpha}|^{2}(\log|z_j^{\alpha}|^{-1})^{2}}.
\end{displaymath}
We point out that $|\cdz_{k}^{\alpha}/\bar{z}_{k}^{\alpha}|$ is bounded along the fibers of $p_{k,\varepsilon}^{\alpha}$, that the form $f\left|\prod_{j\neq k}\dd z_{j}^{\alpha}\wedge\cdz_{j}^{\alpha}\right|$ is integrable and $(\log\log|z_{k}^{\alpha}|^{-1})^{M}/(\log\|s\|^{-2}_{0})^{1/2}$ vanishes along $T_k$. This is enough to prove the claim (\ref{eq:Stokes6}).
Therefore, from (\ref{eq:Stokes5}) $\lim_{\varepsilon\to 0}I_{\alpha,\varepsilon}^{(1)}=0$ and consequently
\begin{equation}\label{eq:Stokes7}
	\begin{split}
		I_{\alpha}=&I_{\alpha}^{(2)}+I_{\alpha}^{(3)}\\
		=&\int_{Y}\frac{1}{(\log\|s\|_{0}^{-2})^{1/2}}\omega\chi_{\alpha}\mu\\
		-&\frac{i}{2\pi}\int_{Y}\frac{1}{(\log\|s\|_{0}^{-2})^{1/2}}\cpd\log\|s\|_{0}^{-2}(\dd\chi_{\alpha})\mu.
	\end{split}
\end{equation}
Finally, since the open covering $\lbrace V_{\alpha}\rbrace_{\alpha}\cup\lbrace V_{\beta}\rbrace_{\beta}$ is finite, from (\ref{eq:Stokes2}) and (\ref{eq:Stokes7}) we derive
\begin{equation}\label{eq:Stokes8}
	\begin{split}
		I=&\int_{Y}\frac{1}{(\log\|s\|_{0}^{-2})^{1/2}}(\sum_{\gamma}\chi_{\gamma})\omega\mu\\
		-&\frac{i}{2\pi}\int_{Y}\frac{1}{(\log\|s\|_{0}^{-2})^{1/2}}\cpd\log\|s\|_{0}^{-2}\dd(\sum_{\gamma}\chi_{\gamma})\mu.
	\end{split}
\end{equation}
The assertion of the lemma follows from (\ref{eq:Stokes8}) once we note that $\sum_{\gamma}\chi_{\gamma}=1$ and $\dd(\sum_{\gamma}\chi_{\gamma})$$=0$.
\end{proof}
\subsection{Application to the projective case}\label{projective_case}
We suppose that $X$ is a nonsingular complex projective variety (non necessarily connected) and $D\subseteq X$ a divisor with \textit{simple} normal crossings. Consider the following elements:
\begin{itemize}
	\item[\textit{i}.] as in section \ref{section_decomp_thm}, for every integer $N\geq 0$ introduce a function $\Theta_N$ (see Notation \ref{notation_Theta_N}). By Lemma \ref{lemma_Theta_N}, the functions $\Theta_N$ are pre-log-log, with singularities along $D$;
	\item[\textit{ii.}] an ample line bundle $L$ on $X$, admitting a global section $s\in\Gamma(X,L)$ such that $\gdiv s$ is a divisor with \textit{simple} normal crossings containing $D$. In particular $\gdiv s$ is reduced and may be seen as a reduced Weil divisor or a reduced scheme;
	\item[\textit{iii}.] a pre-log-log metric $\|\cdot\|$ on $L$, with singularities along $D$;
	\item[\textit{iv}.] a smooth hermitian metric $\|\cdot\|_0$ on $L$ with $\omega:=\c1(\overline{L}_{0})>0$ and $\|s\|_{0}^{2}\leq e^{-e}$; 
	\item[\textit{v}.] as in \textsection \ref{geometric_assumptions} we introduce $\ell_{Q}=(\log\log\|s\|_{0}^{-2})^{Q}$, $Q\in\Int_{\geq 0}$. By Lemma \ref{lemma_theta_N}, the function $\ell_Q$ is pre-log-log, with singularities along $\gdiv s$. Moreover, Proposition \ref{prop_bound_theta_N} asserts the existence of a constant $A_{Q}>0$ such that $\tau_{Q}:=\ddc(-\ell_{Q})+A_{Q}\omega\geq 0$. We can take $A_0=1$.
\end{itemize}
We write $\|\cdot\|=e^{-f/2}\|\cdot\|_{0}$, where $f:X\setminus D\rightarrow\R$ is a pre-log-log function, with singularities along $D$. If $\|\cdot\|$ is good, then $f$ is P-singular. Since $X$ is compact, associated to $\omega=\c1(\overline{L}_0)$ there is a decomposition $f=\varphi-\psi$ as in Theorem \ref{decomp_thm_2}. Moreover, because $D\subseteq\gdiv s$, there exist a constant $C\geq 0$ and an integer $M\geq 0$ such that
\begin{displaymath}
	\varphi\leq C\ell_{M}, \psi\leq C\ell_{M}, \Theta_1\leq C\ell_1, \Theta_{N}\leq C\ell_N
\end{displaymath}
hold on $X\setminus\gdiv s$. If $\|\cdot\|$ is good, then we can take $M=1$. Finally, by Theorem \ref{decomp_thm_2} and Proposition \ref{prop_bound_theta_N}, there exist constants $A,B>0$ and an integer $N\geq 0$ such that $\tau:=\ddc(-\Theta_1)+A\omega\geq 0$ and 
\begin{displaymath}
	\begin{split}
		\omega_{\varphi}:=&\ddc(-\varphi) +B\Theta_{N}(\ddc(-\Theta_1)+A\omega)\geq 0,\\
		\omega_{\psi}:=&\ddc(-\psi)+B\Theta_{N}(\ddc(-\Theta_1)+A\omega)\geq 0.
	\end{split}
\end{displaymath}
If $\|\cdot\|$ is good along $D$, then we can take $N=0$.
Therefore the assumptions of \textsection \ref{geometric_assumptions} are fulfilled for $X$.\\
Let now $\pi:X^{\prime}\rightarrow X$ be a morphism of complex analytic manifolds such that the inverse image schemes $\pi^{-1}(D)\subseteq \pi^{-1}(\gdiv s)$ are divisors with normal crossings. Let $Y\subseteq X^{\prime}$ be a compact complex analytic submanifold of pure dimension $d$. Suppose that $Y$ meets $\pi^{-1}(\gdiv s)$ in a divisor with normal crossings in $Y$. Then $Y\cap D$ is a divisor with normal crossings in $Y$, too. We pull-back by $\pi$ all the objects introduced above ($\pi^{-1}(D)$, $\pi^{*}\Theta_N$, $\pi^{*}L$, $\pi^{*}s$, $\pi^{*}\|\cdot\|$, etc.) and then we restrict them to $Y$. We obtain corresponding objects on $Y$ ($\pi^{-1}(D)\cap Y$, $(\pi^{*}\Theta_{N})_{\mid Y}$, $(\pi^{*}L)_{\mid Y}$, $(\pi^{*}s)_{\mid Y}$, $(\pi^{*}\|\cdot\|)_{\mid Y}$, etc.). Provided that $\deg_{\pi^{*}L}Y>0$, the requirements of \textsection \ref{geometric_assumptions} are fulfilled on $Y$. It is important to point out that the involved constants $A$, $\lbrace A_{Q}\rbrace_{Q}$, $B$, $C$, $M$, $N$ don't depend on the data $X^{\prime}$, $\pi$, $Y$. For every integer $0\leq p\leq d$ define
\begin{displaymath}
	J_{p}^{\ast}=\int_{Y}(\pi^{*}f)\c1(\pi^{*}\overline{L})^{p}\c1(\pi^{*}\overline{L}_{0})^{d-p}.
\end{displaymath}
As a consequence of Theorem \ref{thm_height_integral} we get the following corollary.
\begin{corollary}\label{corol_height_integral}
Let $\pi:X^{\prime}\rightarrow X$ be a morphism of complex analytic manifolds such that $\pi^{-1}(D)$ and $\pi^{-1}(\gdiv s)$ are divisors with normal crossings. Let $Y\subseteq X^{\prime}$ be a compact complex analytic submanifold of pure dimension $d$, intersecting $\gdiv s$ in a divisor with normal crossings in $Y$. Suppose that $\deg_{\pi^{*}L}Y>0$. There exist constants $\alpha$, $\beta>0$ and an integer $R\geq 0$, depending only on $A$, $\lbrace A_{Q}\rbrace_{Q}$, $B$, $C$, $M$ and $N$ such that, for all $p\in\lbrace 0,\ldots,d\rbrace$,
\begin{displaymath}
	|J_{p}^{\ast}|\leq\alpha\deg_{\pi^{*}L}Y+\beta\cdot(\deg_{\pi^{*}L}Y)\cdot\log^{R}
	\left(\int_{Y}\log\pi^{*}\|s\|_{0}^{-2}\frac{\c1(\pi^{*}\overline{L}_{0})^{d}}{\deg_{\pi^{*}L}Y}\right).
\end{displaymath}
If $\|\cdot\|$ is good along $D$, then we can take $R=1$:
\begin{displaymath}
	|J_{p}^{\ast}|\leq\alpha\deg_{\pi^{*}L}Y+\beta\cdot(\deg_{\pi^{*}L}Y)\cdot\log\left(\int_{Y}\log\pi^{*}\|s\|_{0}^{-2}\frac{\c1(\pi^{*}\overline{L}_{0})^{d}}{\deg_{\pi^{*}L}Y}\right).
\end{displaymath}
\end{corollary}

Let $Z$ be a reduced closed subscheme of $X$ of pure dimension $d$, intersecting $\gdiv s$ properly. Denote by $Z_{1},\ldots, Z_{r}$ the irreducible components of $Z$. For every $i=1,\ldots,r$, let $\pi_{i}:X_{i}\rightarrow X$ be an imbedded resolution of singularities of $Z_{i}$, such that $\pi_{i}^{-1}(\gdiv s)$ is a divisor with normal crossings intersecting the strict transform $\widetilde{Z}_{i}$ of $Z_{i}$ in a divisor with normal crossings in $\widetilde{Z}_{i}$ (see \cite{BKK}, Theorem 7.27). Then Corollary \ref{corol_height_integral} applies to $\pi_{i}$, $Z_{i}$, for every $i=1,\ldots,r$. If $\theta$ is a smooth differential form of degree $2d$ on $X\setminus \gdiv s$, locally integrable on $X$, then we adopt the convention
\begin{displaymath}
	\int_{Z}\theta=\sum_{i=1}^{r}\int_{\widetilde{Z}_{i}}\pi_{i}^{*}\theta.
\end{displaymath}
This definition intrinsically depends on $Z$, and not on the choice of the resolutions $\pi_{i}$. With this convention, define
\begin{displaymath}
	J_{p}=\int_{Z}f\c1(\overline{L})^{p}\c1(\overline{L}_{0})^{d-p}.
\end{displaymath}
In this situation Corollary \ref{corol_height_integral} reads as follows.
\begin{corollary}\label{corollary_example}
Let $Z$ be a reduced closed subscheme of $X$ of pure dimension $d$, intersecting $\gdiv s$ properly. There exist positive constants $\alpha$, $\beta$, and an integer $R\geq 0$, depending only on $A$, $\lbrace A_{Q}\rbrace_{Q}$, $B$, $C$, $M$ and $N$ such that, for all $p\in\lbrace 0,\ldots,d\rbrace$,
\begin{displaymath}
	|J_{p}|\leq\alpha\deg_{L}Z+\beta\cdot(\deg_{L} Z)\cdot\log^{R}\left(\int_{Z}\log\|s\|_{0}^{-2}\frac{\c1(\overline{L}_{0})^{d}}{\deg_{L}Z}\right),
\end{displaymath}
with $R=1$ whenever $\|\cdot\|$ is good along $D$.
\end{corollary}
\begin{proof}
Decompose $Z$ into irreducible components: $Z=Z_{1}\cup\ldots\cup Z_{r}$. Following the convention above, define
\begin{displaymath}
	J_{p}^{(i)}=\int_{Z_{i}}f\c1(\overline{L})^{p}\c1(\overline{L}_{0})^{d-p},
\end{displaymath}
so that
\begin{equation}\label{eq:example1}
	J_{p}=\sum_{i=1}^{r}J_{p}^{(i)}.
\end{equation}
By Corollary \ref{corol_height_integral}, there exist constants $\alpha,\beta>0$ and an integer $R\geq 0$, depending only on $A$, $\lbrace A_{Q}\rbrace_{Q}$, $B$, $C$, $M$ and $N$ such that, for all $p\in\lbrace 0,\ldots,d\rbrace$,
\begin{equation}\label{eq:example2}
	|J_{p}^{(i)}|\leq\alpha\deg_{L}Z_{i}+\beta\cdot(\deg_{L} Z_{i})\cdot\log^{R}\left(\int_{Z_{i}}\log\|s\|_{0}^{-2}\frac{\c1(\overline{L}_{0})^{d}}{\deg_{L}Z_{i}}\right).
\end{equation}
If $\|\cdot\|$ is good along $D$, then we can take $R=1$. Now recall that the function $\log^{R}$ is increasing and concave on $]e^{R-1},+\infty[$, so that
\begin{equation}\label{eq:example3}
	\begin{split}
	\sum_{i=1}^{r}\frac{\gdeg_{L}Z_{i}}{\gdeg_{L}Z}&\log^{R}
	\left(\int_{Z_{i}}\log\|s\|_{0}^{-2}\frac{\c1(\overline{L}_{0})^{d}}{\gdeg_{L}Z_{i}}\right)\\ \leq
	&\log^{R}\left(e^{R+1}\sum_{i=1}^{r}\frac{\gdeg_{L}Z_{i}}{\gdeg_{L}Z}
	\int_{Z_{i}}\log\|s\|_{0}^{-2}\frac{\c1(\overline{L}_{0})^{d}}{\gdeg_{L}Z_{i}}\right)\\
	=&\log^{R}\left(e^{R+1}\int_{Z}\log\|s\|_{0}^{-2}\frac{\c1(\overline{L}_{0})^{d}}{\gdeg_{L}Z}\right)\\
	\leq &(2R+2)^{R}\log^{R}\left(\int_{Z}\log\|s\|_{0}^{-2}\frac{\c1(\overline{L}_{0})^{d}}{\gdeg_{L}Z}\right).
	\end{split}
\end{equation}
For the last inequality we used that $x+y\leq 2xy$ for real $x,y\geq 1$. The lemma follows combining (\ref{eq:example1})--(\ref{eq:example3}).

\end{proof}
\section{Arakelovian heights}\label{section_heights}
In this section we turn to an arithmetic situation and deal with arakelovian heights on arithmetic varieties. We prove Theorem \ref{main_theorem}, which can be seen as an arithmetic counterpart of the bounds in \textsection \ref{section_bounding_height_integrals}. A remarkable and straightforward outcome is the finiteness property of arakelovian heights with respect to pre-log-log hermitian ample line bundles, as well as the existence of a universal lower bound (Corollary \ref{Corollary_Main}). 
\subsection{Heights attached to pre-log-log hermitian line bundles}
Let $K$ be a number field and $\OO_K$ its ring of integers. Write $\BS=\Spec\OO_K$. Throughout this section we work with a fixed arithmetic variety $\pi:\X\rightarrow\BS$ of relative dimension $n$. We recall this means that $\X$ is a flat and projective scheme over $\BS$, with regular generic fiber $\X_K=\X\times_{\BS}\Spec K$ of pure dimension $n$. The set of complex points $\X(\C)$ of $\X$ has a natural structure of complex analytic manifold, and it can be decomposed as
\begin{displaymath}
	\X(\C)=\coprod_{\sigma:K\hookrightarrow\C}\X_{\sigma}(\C).
\end{displaymath}
Complex conjugation induces an antiholomorphic involution
\begin{displaymath}
	F_{\infty}:\X(\C)\longrightarrow\X(\C).
\end{displaymath}
We fix $D\subset\X_K$ a divisor, such that $D(\C)\subset\X(\C)$ has \textit{simple normal crossings}. Write $U=\X(\C)\setminus D(\C)$.
\begin{notation}[\cite{BKK}]
We define $\Z^{p}_{U}(\X)$ to be the free abelian group generated by the irreducible reduced subschemes $Z\subseteq\X$ of codimension $p$, such that $Z(\C)$ intersects $D(\C)$ properly. We call $\Z^{p}_{U}(\X)$ the group of cycles of codimension $p$, intersecting $D(\C)$ properly. A cycle $z$ is said to be \textit{vertical} if its components are supported on closed fibers $\X_{\wp}$, $\wp\in\BS\setminus\{(0)\}$. A cycle $z$ is said to be \textit{horizontal} if its irreducible components are flat over $\BS$. We denote by $\Z^{p}_{U}(\X_K)$ the subgroup of $\Z^{p}_{U}(\X)$ of horizontal cycles.
\end{notation}
\begin{definition}
A pre-log-log hermitian line bundle on $\X$, with singularities along $D$, is a couple $\HAL=(\AL,\|\cdot\|)$ formed by
\begin{itemize}
	\item[\textit{i}.] a line bundle (invertible sheaf) $\AL$ on $\X$;
	\item[\textit{ii}.] a pre-log-log hermitian metric $\|\cdot\|$ on the line $\AL_{\C}$ on $\X(\C)$, with singularities along $D(\C)$, and invariant under the action of complex conjugation $F_{\infty}$: $F^{*}_{\infty}\|\cdot\|=\|\cdot\|$.
\end{itemize}
\end{definition}
In \cite{BKK}, \cite{BKK2}, Burgos, Kramer and K\"uhn attach a height morphism to a pre-log-log hermitian line bundle $\HAL$,
\begin{displaymath}
	h_{\HAL}:\bigoplus_{p}\Z^{p}_{U}(\X)\longrightarrow\R,
\end{displaymath}
generalizing the height morphism for smooth hermitian line bundles introduced by Bost, Gillet and Soul\'e in \cite{BGS}. We refer the reader to the cited bibliography for the precise definition and basic properties of $h_{\HAL}$, both in the smooth and pre-log-log case. For our purposes, it will be enough to state the following propositions summarizing the main features of $h_{\HAL}$.
\begin{proposition}\label{prop_char_height}
Let $\HAL$ be a pre-log-log hermitian line bundle on $\X$, with singularities along $D$. The height $h_{\HAL}$ satisfies the following three properties:
\begin{itemize}
	\item[H1.] if $P_{K}:\Spec K\rightarrow \X$ is a $K$-valued point whose image does not belong to $D$, and $P:\BS\rightarrow\X$ denotes its extension to $\BS$, then
		\begin{displaymath}
			h_{\HAL}(P_{*}\BS)=\adeg(P^{*}\HAL)\footnote{The \textit{arithmetic degree} $\adeg$ of a hermitian line bundle  $\overline{\mathcal{M}}=(\mathcal{M},\|\cdot\|)$ over $\Spec\OO_K$ is defined as follows: if $s$ is a non zero global section of $\mathcal{M}$, then $\adeg(\overline{\mathcal{M}})=\log\sharp\left(\frac{\mathcal{M}}{s\OO_{K}}\right)-\sum_{\sigma:K\hookrightarrow\C}\log\|s\|_{\sigma}$. }
		\end{displaymath}
	\item[H2.] if $z$ is a vertical cycle supported on a closed fiber $\X_{\wp}$, then
		\begin{displaymath}
			h_{\HAL}(z)=\log(\Norm\wp)\gdeg_{\AL_{\X_{\wp}}}z,
		\end{displaymath}
		where $\Norm\wp$ denote the norm of the ideal $\wp$;
	\item[H3.] let $z\in\Z_{U}^{n+1-p}(\X_K)$ be irreducible and reduced. Let $s$ be a rational section of $\AL^{\otimes N}$ which does not identically vanish on $z$. If $(\gdiv(s).z)(\C)$ intersects $D(\C)$ properly, then
		\begin{displaymath}
					Nh_{\HAL}(z)=h_{\HAL}(\gdiv(s).z)-\int_{z(\C)}\log(\|s\|_{\AL^{\otimes N}})\c1(\HAL)^{p-1}.
		\end{displaymath}	
\end{itemize}
\end{proposition}
\begin{proof}
This follows from the definition of $h_{\HAL}$ and the extended arithmetic intersection theory in \cite{BKK}.
\end{proof}
\begin{remark}
\textit{i}. The convergence of the integral in \textit{H3} is implicit in the statement.\\
\textit{ii.} If $\AL_{K}$ is an ample line bundle, then an easy inductive argument shows that the properties \textit{H1}, \textit{H2} and \textit{H3} actually characterize $h_{\HAL}$.
\end{remark}
\begin{proposition}\label{prop_char_height2}
Let $\AL$ be a line bundle on $\X$, $\|\cdot\|$ a pre-log-log hermitian metric on $\AL$, with singularities along $D$, and $\|\cdot\|_{0}$ a smooth hermitian metric on $\AL$. Write $\HAL=(\AL,\|\cdot\|)$, $\HAL_{0}=(\AL,\|\cdot\|_{0})$ and $\|\cdot\|=e^{-f/2}\|\cdot\|_{0}$, where $f:\X(\C)\setminus D(\C)\rightarrow\R$ is a pre-log-log function, with singularities along $D(\C)$. For any cycle $z\in Z^{n+1-p}_{U}(\X)$ we have
\begin{displaymath}
	h_{\HAL}(z)=h_{\HAL_{0}}(z)+\frac{1}{2}\sum_{k=0}^{p-1}\int_{z(\C)}f\c1(\HAL)^{k}\c1(\HAL_{0})^{p-1-k}.
\end{displaymath}
\end{proposition}
\begin{proof}
This is contained in \cite{BKK2}, Theorem 4.1.
\end{proof}
\begin{remark}
Proposition \ref{prop_char_height2} allows to recover Proposition \ref{prop_char_height} once it is known for smooth hermitian line bundles. In this case the properties \textit{H1}, \textit{H2} and \textit{H3} are already established in \cite{BGS}.
\end{remark}
Let $F|K$ be a finite extension of fields and write $\BT=\Spec\OO_{F}$. Base changing by $\BT\rightarrow\BS$, we get an arithmetic variety $\X_{\BT}\rightarrow\BT$, together with a finite flat morphism $g:\X_{\BT}\rightarrow\X$ of degree $[F:K]$. If $D\subseteq\X_K$ is an effective divisor such that $D(\C)\subseteq\X(\C)$ has simple normal crossings, then $D_{F}\subseteq\X_{F}$ is an effective divisor and $D_{F}(\C)\subseteq\X_{\BT}(\C)$ has simple normal crossings as well. Let $\HAL$ be a pre-log-log hermitian line bundle on $\X$, with singularities along $D$. The pull-back $g^{*}\HAL$ of $\HAL$ to $\X_{\BT}$ is a pre-log-log hermitian line bundle, with singularities along $D_{F}$. Since $g$ is flat, for every cycle $z$ on $\X$ there is a well defined pull-back cycle $g^{*}(z)$. Observe that if $z(\C)$ intersects $D(\C)$ properly, then $g^{*}(z)(\C)$ intersects $D_F(\C)$ properly. Namely, the correspondence $z\mapsto g^{*}(z)$ induces a morphism
\begin{equation}\label{eq_bc}
	\begin{split}
		\Z^{p}_{U}(\X)&\longrightarrow \Z^{p}_{V}(\X_{\BT})\\
		z &\longmapsto g^{*}(z),
	\end{split}
\end{equation}
where $V=\X_{\BT}(\C)\setminus D_{F}(\C)$. This morphism maps $\Z^{p}_{U}(\X_K)$ into $\Z^{p}_{V}(\X_F)$.
\begin{lemma}\label{lemma_base_change}
Let $\HAL$ be a pre-log-log hermitian line bundle on $\X$, with singularities along $D$. Let $F|K$ be a finite extension of fields and $\BT=\Spec\OO_F$. Write $g:\X_{\BT}\rightarrow\X$ for the finite flat projection induced by $\BT\rightarrow\BS$. Let $w\in\Z^{p}_{V}(\X_F)$ be an irreducible and reduced cycle and set $z=g(w)_{\red}$. Let $\delta$ be the degree of $g\mid_{w}$. Then we have the equality
\begin{displaymath}
	h_{g^{*}\HAL}(w)=\delta h_{\HAL}(z).
\end{displaymath}
Consequently, for the morphism $g^{*}$ of (\ref{eq_bc}), we have
\begin{displaymath}
	h_{g^{*}\HAL}(g^{*}(z))=[F:K]h_{\HAL}(z)
\end{displaymath}
for every $z\in\Z^{p}_{U}(\X)$ and every $p$.
\end{lemma}
\begin{proof}
This follows for instance from the case of smooth metrics (see \cite{BGS}, \textsection 3.1.4 and Proposition 3.2.1) and Proposition \ref{prop_char_height2}, since $g_{\mid w(\C)}:w(\C)\rightarrow z(\C)$ is generically smooth and finite of degree $\delta$, so that
\begin{displaymath}
	\int_{w(\C)}g^{*}\left(f\c1(\HAL)^{k}\c1(\HAL_{0})^{p-1-k}\right)=\delta\int_{z(\C)}f\c1(\HAL)^{k}\c1(\HAL_{0})^{p-1-k}.
\end{displaymath}
\end{proof}
\begin{notation}
Let $\HAL$ be a pre-log-log hermitian line bundle, with singularities along $D$. Let $z\in\Z^{p}_{U}(\X_K)$ with $\gdeg_{\AL_K}z\neq 0$. We define its normalized height to be
\begin{displaymath}
	\widetilde{h}_{\HAL}(z)=\frac{h_{\HAL}(z)}{[K:\Q]\gdeg_{\AL_K}z}.
\end{displaymath}
\end{notation}
\begin{lemma}\label{lemma_norm_height}
Let $g:\X_{\BT}\rightarrow\X$ be as before. Let $w\in\Z^{p}_{V}(\X_F)$ be an irreducible and reduced cycle and set $z=g(w)_{\red}$. For the normalized height we have $\widetilde{h}_{g^{*}\HAL}(w)=\widetilde{h}_{\HAL}(z)$ and $\widetilde{h}_{g^{*}\HAL}(g^{*}z)=\widetilde{h}_{\HAL}(z)$.
\end{lemma}
\begin{proof}
For every $w\in\Z^{n+1-p}_{V}(\X_F)$ irreducible and reduced and $z=g(w)_{\red}$, we have the equalities
\begin{displaymath}
	\begin{split}
		\gdeg_{(g^{*}\AL)_F}w &=\frac{1}{[F:\Q]}\int_{w(\C)}\c1(g^{*}\HAL)^{p-1}\\
		&=\frac{\delta}{[F:\Q]}\int_{z(\C)}\c1(\HAL)^{p-1}\\
		&=\frac{\delta}{[F:K]}\gdeg_{\AL_K}z,
	\end{split}
\end{displaymath}
where $\delta$ is the degree of $g_{\mid w}$. It follows that $\gdeg_{(g^{*}\AL)_F}(g^{*}z)=\gdeg_{\AL_K}z$. Combined with Lemma \ref{lemma_base_change}, we get $\widetilde{h}_{g^{*}\HAL}(w)=\widetilde{h}_{\HAL}(z)$ and $\widetilde{h}_{g^{*}\HAL}(g^{*}z)=\widetilde{h}_{\HAL}(z)$.
\end{proof}
\begin{remark}
The normalization $\widetilde{h}_{\HAL}$ just introduced is not the standard one. However it appears naturally in the statement and proof of Theorem \ref{main_theorem}. 
\end{remark}
\subsection{Proof of the main theorem}
We now proceed to prove Theorem \ref{main_theorem}. The argument mainly relies on the bounds established in \textsection \ref{section_bounding_height_integrals} , and more concretely the situation studied in \textsection \ref{projective_case}. However, in the reduction steps we will need the following Bertini's type theorem. The proof is essentially well known, but we include it in the Appendix for lack of reference.
\begin{proposition}\label{prop_Bertini}
Let $X$ be a nonsingular projective scheme over an algebraically closed field $k$. Let $D\subseteq X$ be a divisor with simple normal crossings. Let $L$ be an ample line bundle on $X$. Then there exist an integer $N>0$ and global sections $s_1,\ldots,s_r\in\HH^{0}(X,L^{\otimes N})$ such that $\supp(\gdiv s_1)$,\ldots, $\supp(\gdiv s_r)$ are divisors with simple normal crossings and the following equality of schemes holds:
\begin{displaymath}
	D=(\supp(\gdiv s_1)\cap\ldots\cap\supp(\gdiv s_r))_{\red}.
\end{displaymath}
\end{proposition}
Under the hypothesis of Theorem \ref{main_theorem}, since $D(\C)$ has simple normal crossings, $D_{\overline{K}}$ has also simple normal crossings in $\X_{\overline{K}}$. We will apply Proposition \ref{prop_Bertini} through the following straightforward corollary.
\begin{corollary}\label{corollary_Bertini}
There exist a finite extension $K'|K$, a positive integer $N$ and global sections $s_1,\ldots, s_r\in\HH^{0}(\X_{K'},\AL_{K'}^{\otimes N})$ such that
\begin{itemize}
	\item[B1.] $\supp(\gdiv s_j)(\C)$ is a divisor with simple normal crossings in $\X(\C)$, for every $j=1,\ldots,r$;
	\item[B2.] $D_{K'}=(\supp(\gdiv s_1)\cap\ldots\cap\supp(\gdiv s_r))_{\red}$.
\end{itemize}
\end{corollary}
The next two lemmas provide the final reductions before the proof of Theorem \ref{main_theorem}.
\begin{lemma}\label{first_reduction}
It is enough to proof Theorem \ref{main_theorem} in the following situations:
\begin{itemize}
	\item[i.] after some finite extension $K'\mid K$;
	\item[ii.] $\AL$ is very ample and $z\in\Z^{p}_{U}(\X_K)$ is irreducible and reduced. 
\end{itemize}
\end{lemma}
\begin{proof}
The first claim \textit{i} is clear. For the proof of \textit{ii}, we first note that the statement of Theorem \ref{main_theorem} for $\AL^{\otimes N}$ already implies the statement of the theorem for $\AL$, since $\widetilde{h}_{\HAL^{\otimes N}}(z)=N\widetilde{h}_{\HAL}(z)$ and $\widetilde{h}_{\HAL_{0}^{\otimes N}}(z)=N\widetilde{h}_{\HAL_{0}}(z)$. Hence we assume $\AL_{K}$ is very ample. We then proceed in two steps.\\
\textit{Step 1.} We can suppose that $\AL$ is very ample. Indeed, there exists some model $(\Y,\mathscr{A})$ of $(\X_{K},\AL_{K})$ with $\mathscr{A}$ very ample. The metrics $\|\cdot\|$, $\|\cdot\|_{0}$ on $\AL$ induce metrics on $\mathscr{A}$, and we write $\overline{\mathscr{A}}$ and $\overline{\mathscr{A}}_0$ for the corresponding hermitian line bundles. For every effective cycle $z\in\Z^{p}_{U}(\X_K)$, we write $\tilde{z}$ for the corresponding effective and horizontal cycle on $\Y$. By Proposition 3.2.2 in \cite{BGS}, there exists a positive constant $C$, independent of $z$, such that
\begin{equation}\label{eq_mt1}
	\left |h_{\overline{\mathscr{A}}_{0}}(\tilde{z})-h_{\HAL_{0}}(z)\right |\leq C\deg_{\AL_K} z.
\end{equation}
Since $\overline{\mathscr{A}}_{\C}$ and $\overline{\mathscr{L}}_{\C}$ are isometric, Proposition \ref{prop_char_height2} and (\ref{eq_mt1}) together give
\begin{displaymath}
	\left |h_{\overline{\mathscr{A}}}(\tilde{z})-h_{\HAL}(z)\right|=\left |h_{\overline{\mathscr{A}}_{0}}(\tilde{z})-h_{\HAL_{0}}(z)\right |\leq C\deg_{\AL_K} z.
\end{displaymath}
\textit{Step 2.} We can suppose that $z$ is irreducible and reduced. Indeed, suppose that we have shown the existence of constants $\alpha$, $\beta$, $\gamma>0$ and $R\in\Int_{\geq 0}$ such that, for every $w\in\Z^{p}_{U}(\X_K)$ irreducible and reduced, we have $\widetilde{h}_{\HAL_{0}}(w)+\gamma\geq 1$ and
	\begin{equation}\label{eq_mt3}
		\left|\widetilde{h}_{\HAL}(w)-\widetilde{h}_{\HAL_0}(w)\right|\leq\alpha+\beta\log^{R}\left(\widetilde{h}_{\HAL_0}(w)+\gamma\right).
	\end{equation}
After possibly increasing $\gamma$ we can suppose that $\widetilde{h}_{\HAL_0}(w)+\gamma> e^{R-1}$, for every irreducible and reduced $w\in\Z^{p}_{U}(\X_K)$. Let us now consider $w=\sum_{i\in I}w_{i}\in\Z^{p}_{U}(\X_K)\setminus\{0\}$ where the $w_i\in\Z^{p}_{U}(\X_K)$ are irreducible and reduced. Then (\ref{eq_mt3}) yields
\begin{displaymath}
	\begin{split}
		\left |\widetilde{h}_{\HAL}(w)-\widetilde{h}_{\HAL_{0}}(w)\right|&
		\leq\frac{\deg_{\AL_K}w_i}{\deg_{\AL_K}w}\left |\widetilde{h}_{\HAL}(w_i)-\widetilde{h}_{\HAL_0}(w_i)\right|\\
		&\leq \alpha+\beta\sum_{i\in I}\frac{\deg_{\AL_K}w_i}{\deg_{\AL_K}w}\log^{R}\left (\widetilde{h}_{\HAL_0}(w_i)+\gamma\right).
	\end{split}
\end{displaymath}
Since $\widetilde{h}_{\HAL_0}(w_i)+\gamma> e^{R-1}$ for all $i$ and $\log^{R}$ is concave on $]e^{R-1},+\infty[$, we conclude
\begin{displaymath}
	\begin{split}
		\sum_{i\in I}\frac{\deg_{\AL_K}w_i}{\deg_{\AL_K}w}\log^{R}\left(\widetilde{h}_{\HAL_0}(w_i)+\gamma\right)&
		\leq\log^{R}\left(\sum_{i\in I}\frac{\deg_{\AL_K}w_i}{\deg_{\AL_K}w}(\widetilde{h}_{\HAL_0}(w_i)+\gamma)\right)\\
		&=\log^{R}\left(\widetilde{h}_{\HAL_{0}}(w)+\gamma\right).
	\end{split}
\end{displaymath}
This completes the proof.
\end{proof}
By Corollary \ref{corollary_Bertini} and Lemma \ref{first_reduction}, after possibly extending $K$ and choosing a suitable model of $(\X_K,\AL_K)$, we can suppose that $\AL$ is very ample and there exist sections $s_1,\ldots,s_r\in\HH^{0}(\X_K,\AL_K)$ with the properties \textit{B1}, \textit{B2} above (with $K'=K$). After possibly multiplying the sections $s_j$ by a sufficiently divisible integer, we can even suppose that $s_1,\ldots,s_r\in\HH^{0}(\X,\AL)$. We denote $E_{j}=\supp(\gdiv s_j)_{K}$. We fix these data until the end of the proof. 
\begin{lemma}\label{second_reduction}
Let $z\in\Z^{p}_{U}(\X_K)$ be irreducible and reduced. Let $F$ be a finite extension of $K$ over which all the irreducible components of $z_{\KK}$ are defined. Let $\BT=\Spec\OO_F$ and $g:\X_{\BT}\rightarrow\X$ be the finite flat projection induced by $\BT\rightarrow\BS$. Write $g^{*}(z)=\sum_{i\in I} z_{i}$, $z_{i}$ irreducible, reduced and flat over $\BT$. Then $z(\C)$ intersects $D(\C)$ properly if, and only if, every $z_i(\C)$ intersects one of the $E_{j,F}(\C)$ properly.
\end{lemma}
\begin{proof}
Straightforward.
\end{proof}
Now we can complete the proof of Theorem \ref{main_theorem}.
\begin{proof}[Proof of Theorem \ref{main_theorem}] First of all, for every integer $M\geq 0$ we construct a function $\Theta_{M}$, as in \textsection \ref{section_decomp_thm}, for the complex analytic variety $\X(\C)$ and the divisor with simple normal crossings $D(\C)$.\\
Observe that we can suppose that the metric $\|\cdot\|_{0}$ satisfies $\c1(\HAL_0)>0$ and $\| s_j\|_{0}^{2}\leq e^{-e}$ for every $j=1,\ldots,r$. Indeed, by \cite{BGS}, Proposition 3.2.2 (or also Proposition \ref{prop_char_height2} for \textit{smooth} metrics), a change of smooth metric causes only bounded variations of the normalized height.

We introduce the pre-log-log functions
\begin{displaymath}
	\ell_{Q}^{(j)}=(\log\log\|s_{j}\|_{0}^{-2})^{Q}:\X(\C)\setminus E_{j}(\C)\longrightarrow\R.
\end{displaymath}
For every $Q$ we fix a positive constant $A_{Q}>0$ such that $\ddc(-\ell_{Q}^{(j)})+A_{Q}\c1(\HAL_0)\geq 0$, $A_{0}=1$, for all $j=1,\ldots,r$ (see Proposition \ref{prop_bound_theta_N}).

Write $\|\cdot\|=e^{-f/2}\|\cdot\|_{0}$, where $f:\X(\C)\setminus D(\C)\rightarrow\R$ is a pre-log-log function (resp. P-singular if $\|\cdot\|$ is good). Attached to $\c1(\HAL_0)$ we perform a decomposition $f=\varphi-\psi$ as in Theorem \ref{decomp_thm_2}. Recall that $\varphi$, $\psi$ are positive pre-log-log (resp. P-singular) functions along $D(\C)$, with
\begin{displaymath}
	\begin{split}
		\omega_{\varphi}:=&\ddc(-\varphi)+B\Theta_N (\ddc(-\Theta_1)+A\c1(\HAL_0))\geq 0,\\
		\omega_{\psi}:=&\ddc(-\psi)+B\Theta_N(\ddc(-\Theta_1)+A\c1(\HAL_0))\geq 0,
	\end{split}
\end{displaymath}
for some $A,B>0$ and $N\in\Int_{\geq 0}$. If $\|\cdot\|$ is good, then we can take $N=1$. For every $j=1,\ldots,r$, $D\subseteq E_{j}$. By compactness of $\X(\C)$ there exist constants $C>0$, $M\in\Int_{\geq 0}$ such that
\begin{displaymath}
	\varphi\leq C\ell_{M}^{(j)}, \psi\leq C\ell_{M}^{(j)}, \Theta_{1}\leq C\ell_{1}^{(j)}, \Theta_{N}\leq C\ell_{N}^{(j)},
\end{displaymath}
for all $j\in\lbrace 1,\ldots,r\rbrace$.

Let $z\in\Z^{n+1-p}_{U}(\X_K)$ be irreducible and reduced. Denote by $F$ an extension of $K$ such that all the irreducible components of $z_{\KK}$ are defined over $F$. Let $\BT=\Spec\OO_F$, $g:\X_{\BT}\rightarrow\X$ be as before. Decompose $g^{*}(z)=\sum_{i\in I}z_i$, with the $z_i$ irreducible and flat over $\BT$. By Lemma \ref{second_reduction}, for every $z_{i}$ there exists $j=j(i)\in\lbrace 0,\ldots,r\rbrace$ such that $z_{i}(\C)$ intersects $E_{j,F}(\C)$ properly. By Corollary \ref{corollary_example} and Proposition \ref{prop_char_height2} we have
\begin{equation}\label{eq_mt2}
	\left|\widetilde{h}_{g^{*}\HAL}(z_{i})-\widetilde{h}_{g^{*}\HAL_0}(z_{i})\right|\leq \alpha
	+\beta\log^{R}\left(\int_{z_{i}(\C)}\log g^{*}\|s_{j}\|_{0}^{-2}\frac{\c1(g^{*}\HAL_{0})^{p-1}}{[K:\Q]\gdeg_{(g^{*}\AL)_{F}}z_{i}}\right)
\end{equation}
for $\alpha,\beta>0$, $R\in\Int_{\geq 0}$ depending only on $A$, $\lbrace A_{Q}\rbrace_{Q}$, $B$, $C$, $M$ and $N$. Moreover, if $\|\cdot\|$ is good, then we can take $R=1$. Applying the property \textit{H3} of heights (see Proposition \ref{prop_char_height}), we rewrite (\ref{eq_mt2}) as
\begin{equation}\label{eq_mt4}
	\left|\widetilde{h}_{g^{*}\HAL}(z_{i})-\widetilde{h}_{g^{*}\HAL_0}(z_{i})\right|\leq \alpha
	+\beta\log^{R}\left(2\widetilde{h}_{g^{*}\HAL_0}(z_i)-2\widetilde{h}_{g^{*}\HAL_0}(\gdiv(g^{*}s_j).z_i)\right).
\end{equation}
To derive this inequality we point out that
\begin{displaymath}
	\gdeg_{(g^{*}\AL)_F}(\gdiv (g^{*}s_j).z_i)=\gdeg_{(g^{*}\AL)_{F}}z_i,
\end{displaymath}
so that
\begin{displaymath}
	\frac{h_{g^{*}\HAL_0}(\gdiv(g^{*}s_j).z_i)}{[F:\Q]\gdeg_{(g^{*}\AL)_{F}}z_{i}}=\widetilde{h}_{g^{*}\HAL_0}(\gdiv(g^{*}s_j).z_i).
\end{displaymath}
By Theorem \ref{thm_BGS} there exists a positive constant $\kappa>0$ such that, for every effective cycle $w\neq 0$ on $\X$, we have $\widetilde{h}_{\HAL_0}(w)> -\kappa$. Combined with Lemma \ref{lemma_norm_height} this yields $\widetilde{h}_{g^{*}\HAL_0}(\gdiv(g^{*}s_j).z_i)> -\kappa$, because $\gdiv(g^{*}s_j).z_i$ is effective. Therefore, from (\ref{eq_mt4}) we deduce
\begin{equation}\label{eq_mt5}
	\begin{split}
		\left|\widetilde{h}_{g^{*}\HAL}(z_{i})-\widetilde{h}_{g^{*}\HAL_0}(z_{i})\right|\leq &\alpha
		+\beta\log^{R}\left(2\widetilde{h}_{g^{*}\HAL_0}(z_i)+2\kappa+2e^{R+1}\right)\\
	\leq&\alpha+2^{R}\beta\log^{R}\left(\widetilde{h}_{g^{*}\HAL_0}(z_i)+\kappa+e^{R+1}\right),
	\end{split}
\end{equation}
where we applied the trivial inequalities $\log 2\leq 1$ and $x+y\leq 2xy$ for real $x,y\geq 1$. Now $\widetilde{h}_{g^{*}\HAL_0}(z_i)+\kappa+e^{R+1}>e^{R-1}$ and $\log^{R}$ is concave on $]e^{R-1},+\infty[$. From (\ref{eq_mt5}) we derive
\begin{equation}\label{eq_mt6}
	\begin{split}
		\Big|\widetilde{h}_{g^{*}\HAL}(g^{*}z)-&\widetilde{h}_{g^{*}\HAL_0}(g^{*}z)\Big|\\
		\leq&\sum_{i\in I}\frac{\gdeg_{(g^{*}\AL)_F}(z_i)}{\gdeg_{(g^{*}\AL)_{F}}(g^{*}z)}\left|\widetilde{h}_{g^{*}\HAL}(z_i)-\widetilde{h}_{g^{*}\HAL_0}(z_i)\right|\\
		\leq&\alpha+2^{R}\beta\log^{R}
		\left(\sum_{i\in I}\frac{\gdeg_{(g^{*}\AL)_F}(z_i)}{\gdeg_{(g^{*}\AL)_{F}}(g^{*}z)}\widetilde{h}_{g^{*}\HAL_0}(z_i)+\kappa+e^{R+1}\right)\\
		=&\alpha+2^{R}\beta\log^{R}\left(\widetilde{h}_{g^{*}\HAL_{0}}(g^{*}z)+\kappa+e^{R+1}\right).
	\end{split}
\end{equation}
By Lemma \ref{lemma_norm_height}, $\widetilde{h}_{g^{*}\HAL}(g^{*}z)=\widetilde{h}_{\HAL}(z)$ and $\widetilde{h}_{g^{*}\HAL_0}(g^{*}z)=\widetilde{h}_{\HAL_0}(z)$. Hence (\ref{eq_mt6}) is equivalent to
\begin{equation}\label{eq_mt7}
	\left|\widetilde{h}_{\HAL}(z)-\widetilde{h}_{\HAL_0}(z)\right|\leq\alpha+2^{R}\beta\log^{R}\left(\widetilde{h}_{\HAL_0}(z)+\kappa+e^{R+1}\right).
\end{equation}
The constants $\alpha$, $2^{R}\beta$, $\gamma:=\kappa+e^{R+1}>0$, $R\in\Int_{\geq 0}$ ($R=1$ if $\|\cdot\|$ is good) in (\ref{eq_mt7}) depend only on $\HAL$ and $\HAL_0$, and not on $z$. This concludes the proof of the theorem.
\end{proof}
\section{Examples}\label{section_examples}
\subsection{Automorphic vector bundles on toroidal compactifications}
The first natural examples of good hermitian vector bundles are provided by the theory of (fully decomposed) automorphic vector bundles on locally symmetric varieties, and their extensions to smooth toroidal compactifications. These have been firstly worked by Mumford in his proof of Hirzebruch's proportionality principle in the non-compact case \cite{Mumford}. In this section we quote from \textit{loc. cit.} the main construction and Mumford's theorem. As an application, we briefly discuss the example of modular forms.

Let $B$ be a bounded symmetric domain. We can write $B=G/K$, where $G$ is a semi-simple adjoint group and $K$ is a maximal compact subgroup. Denote $K_{\C}, G_{\C}$ the complexifications of $K$ and $G$. Inside $G_{\C}$ there is a parabolic subgroup of the form $P_{+}\cdot K_{\C}$ ($P_{+}$ being its unipotent radical), such that $K=G\cap (P_{+}\cdot K_{\C})$ and $G\cdot(P_{+}\cdot K_{\C})$ is open in $G_{\C}$. Then $\check{B}:=G_{\C}/G\cdot(P_{+}\cdot K_{\C})$ is a rational projective variety and there is a $G$-equivariant immersion $B\hookrightarrow\check{B}$ compatible with the complex structure of $B$. Let $E_{0}$ be a $G$-equivariant vector bundle on $B$ attached to a representation $\sigma:K\rightarrow GL_{r}(\C)$. We complexify $\sigma$ and extend it to $P_{+}\cdot K_{\C}$, by letting it act trivially on $P_{+}$. This extension induces a $G_{\C}$ equivariant analytic vector bundle $\check{E}_{0}$ on $\check{B}$, with $\iota^{*}(\check{E}_{0})=E_{0}$. This way we get a holomorphic structure on $E_{0}$ (which depends on the chosen extension of $\sigma$ to $P_{+}\cdot K_{\C}$).

Let $\Gamma$ be a neat arithmetic subgroup of $G$ acting on $B$. Then $X=\Gamma\backslash B$ is a smooth quasi-projective complex variety. The vector bundle $E_{0}$ descends to a holomorphic vector bundle $E$ on $X$. Such a vector bundle is called \textit{fully decomposed automorphic vector bundle}. Since $K$ is compact, there is a $G$-invariant hermitian metric $h_{0}$ on $E_{0}$, thus inducing a hermitian metric $h$ on $E$.
\begin{theorem}\label{thm_mumford}
Let $\overline{X}$ be a smooth toroidal compactification of $X$ with $D=\overline{X}\setminus X$ a divisor with normal crossings. Then the automorphic vector bundle $E$ extends to a vector bundle $E_{1}$ over $\overline{X}$, such that $h$ is good along $D$.
\end{theorem}
The following proposition may be interesting for some arithmetic purposes.
\begin{proposition}
Suppose that $E_{0}=\omega_{B}$ is the canonical bundle of $B$. Then $E_{1}=\omega_{\overline{X}}(D)$ and coincides with the pull-back of an ample line bundle $\OO(1)$ on the Baily-Borel compactification $X^{*}$ of $X$. The global sections of $\OO(n)$ naturally correspond to the modular forms whose automorphy factor is the $n$th power of the jacobian.
\end{proposition}
\begin{remark}
Under the hypothesis of the proposition, the line bundle $\omega_{\overline{X}}(D)$ is not ample in general.
\end{remark}
Equip the line bundle $E_{0}=\omega_{B}$ with the hermitian metric $h_{0}$ induced by the K\"ahler-Einstein metric on $B$, say with Einstein constant $-1$. The existence and uniqueness is guaranteed by a result of Mok and Yau \cite{MY}. Since the K\"ahler-Einstein metric is invariant under automorphisms, $h_{0}$ is $G$-equivariant. By Theorem \ref{thm_mumford} $h_{0}$ induces a good hermitian metric $h$ on $E_{1}=\omega_{\overline{X}}(D)$, with singularities along $D$. Observe that this metric is induced by the K\"ahler-Einstein metric on $X$ with Einstein constant $-1$, by uniqueness. Since the K\"ahler-Einstein metric has negative Ricci curvature, $\c1(\omega_{X}(D),h)\geq 0$ on $X$. Together with the fact that $\omega_{\overline{X}}(D)$ is the pull-back of an ample line bundle $\OO(1)$ on $X^{*}$, this can be shown to be enough for the main theorem to hold, as soon as $X$ and $\overline{X}$ are defined over a number field $K$ and we have chosen suitable models $\overline{\X}$ of $\overline{X}$, $\X^{*}$ of $X^{*}$, etc. over $\Spec\OO_K$. Suppose that $\OO(1)$ extends to an ample line bundle $\mathcal{A}$ on $\X^{*}$, that there is a morphism $\pi:\overline{\X}\rightarrow\X^{*}$ extending the natural projection $\overline{X}\rightarrow X^{*}$ and put $\AL=\pi^{*}(\mathcal{A})$. The line bundle $\AL$ is a model of $\omega_{\overline{X}}(D)$ and it can be endowed with the good hermitian metric induced by $h$. Then Corollary \ref{Corollary_Main} hold for $\HAL$, provided we restrict to effective \textit{horizontal} cycles. The proof follows the same lines as for pre-log-log hermitian ample line bundles, and it will be detailed elsewhere.

\subsection{Some natural hermitian vector bundles on the moduli space of curves}
Let $g\geq 2$ be an integer and $\CM_{g}$ the moduli space of complex stable curves of genus $g$. We denote by $\pi:\UCC_{g}\rightarrow\CM_{g}$ the universal curve. By $\M_g$ we mean the open subset of $\CM_g$ parametrizing smooth curves, and we write $\UC_g=\pi^{-1}(\M_g )$. The boundary $\pd\M_{g}=\CM_{g}\setminus\M_{g}$, which classifies singular stable curves, is a divisor with normal crossings. We write $\pd\UC_{g}=\pi^{-1}(\pd\M_{g})$, which is a divisor with normal crossings, too. For the sake of simplicity we neglect that $\CM_{g}$ and $\UCC_{g}$ are actually $V$-manifolds, and we work as if they were complex analytic manifolds (see \cite{Wolpert} for the definition of $V$-manifold and the description of the moduli space of curves as a $V$-manifold).

The first example concerns $\omega_{\UCC_{g}/\CM_{g}}$, the relative dualizing sheaf of $\pi$. Every fiber of $\pi_{\mid\UC_{g}}$ admits a unique complete hyperbolic metric of constant negative curvature $-1$. By Teichm\"uller's theory these metrics glue together and define a smooth hermitian metric on $\omega_{\UC_{g}/\M_{g}}^{\vee}$. We get a smooth hermitian metric on $\omega_{\UC_{g}/\M_{g}}$. By a theorem of Wolpert \cite{Wolpert2} this metric extends to a good hermitian metric on $\omega_{\UCC_{g}/\CM_{g}}$, with singularities along $\pd\UC_{g}$. It is well known that $\omega_{\UCC_{g}/\CM_{g}}$ is relatively ample (\cite{DM}, Corollary to Theorem 1.2).

Let us now consider $(\Omega_{\M_g},h_{WP})$ the cotangent bundle of $\M_g$ with the Weil-Petersson metric. Recall that if $p$ is a point of $\M_g$ representing a Riemann surface $R$, then $\Omega_{\M_g,p}$ is isomorphic to the space of holomorphic quadratic differentials on $R$. If $R$ is written as $\mathbb{H}/\Gamma$ ($\mathbb{H}$ Poincar\'e's upper half plane and $\Gamma\subseteq\text{PSL}_{2}(\R)$ a discrete subgroup), then the metric $h_{WP}$ on the stalk $\Omega_{\M_g,p}$ is the usual scalar product
\begin{displaymath}
	\langle\varphi,\psi\rangle=\int_{R}\varphi(z)\overline{\psi(z)}y^{2}\dd x \dd y
\end{displaymath}
on automorphic forms of weight (2,0) for the group $\Gamma$. By a result of Trapani \cite{Trapani}, $(\det\Omega_{\M_{g}},$ $\det h_{WP})$ extends to a good hermitian line bundle $\overline{\omega_{\CM_{g}}(\log\pd\M_{g})}$, with singularities along $\pd\M_{g}$. Moreover Trapani shows that $\omega_{\CM_{g}}(\log\pd\M_{g})$ admits a smooth hermitian metric with positive curvature form. Therefore, its pull-back to the moduli space of curves of genus $g$ with level $n$ structure ($n\geq 3$) is ample.

In an ambitious program pioneered by \cite{LSY}, \cite{LSY2}, Liu, Sun and Yau study the goodness and bounded geometry of several natural K\"ahler metrics on the moduli space of curves. The interested reader is referred to \textit{loc. cit.} for precise statements.

\subsection{K\"{a}hler-Einstein metrics on quasi-projective varieties}
\subsubsection{Complex theory}
The main references we follow are \cite{Kobayashi}, \cite{TianYau}, and \cite{Wu}.

Let $M$ be a complex analytic manifold of dimension $n$ and $\Omega$ a smooth positive $(n,n)$-form on $M$, namely a volume form. For every analytic chart $(V;z_1,\ldots,z_n)$ of $M$ write
\begin{displaymath}
	\Omega_{\mid V}=\xi_{V}\prod_{k=1}^{n}\left(\frac{i}{2\pi}\dd z_k\wedge\cdz_k\right).
\end{displaymath}
The functions $\lbrace\xi_{V}\rbrace_{V}$ define a smooth hermitian metric $\|\cdot\|_{\Omega}$ on the canonical line bundle $\omega_{M}$. By definition, the Ricci form of $\Omega$ is the real (1,1)-form
\begin{displaymath}
	\Ric\Omega=\c1(\omega_{M},\|\cdot\|_{\Omega})
\end{displaymath}
which is locally given by
\begin{displaymath}
	\Ric\Omega\mid_{V}=\ddc\log\xi_{V}.
\end{displaymath}
The generalized Fefferman operator $J$ acting on volume forms is defined as
\begin{displaymath}
	J:\Omega\longmapsto\frac{(\Ric\Omega)^{n}}{\Omega}.
\end{displaymath}
\begin{theorem}[Kobayashi \cite{Kobayashi}]\label{thm_kobayashi}
Let $X$ be a compact complex analytic manifold and $D\subseteq X$ a divisor with simple normal crossings. Suppose that the line bundle $\omega_X(D)$ is ample on $X$. Then there exists a unique complete K\"{a}hler-Einstein metric $g_{KE}$ on $X\setminus D$ with constant negative Ricci curvature -1. If $\Omega_{KE}$ is the volume form of $g_{KE}$, then $g_{KE}$ is characterized by being complete on $X\setminus D$ and by the equation
\begin{displaymath}
	J(\Omega_{KE})=1.
\end{displaymath}
\end{theorem}
\begin{proposition}\label{prop_good_KE}
Let $X$ be a compact complex analytic manifold and $D\subseteq X$ a divisor with simple normal crossings. Suppose that $\omega_X(D)$ is ample on $X$. Let $U=X\setminus D$ and endow $\omega_U$ with the smooth hermitian metric $\|\cdot\|_{KE}$ induced by $\Omega_{KE}$. Then $(\omega_U,\|\cdot\|_{KE})$ extends to a good hermitian line bundle $(\omega_X( D),\|\cdot\|_{KE})$, with singularities along $D$.
\end{proposition}
The proof of Proposition \ref{prop_good_KE} follows easily from the more precise growth properties established in the proof of Theorem \ref{thm_kobayashi}. For the sake of completeness we now deepen some of the details involved. In the sequel we fix $X$ and $D$ as in the proposition.
\begin{definition}
Let $M$ be a complex analytic manifold of dimension $n$. Let $V\subseteq\C^{n}$ be an open subset. A holomorphic map $\phi:V\rightarrow M$ is called a quasicoordinate if $\text{rank}(\dd_{v}\phi)=n$ for every $v\in V$. Then $(V,\phi)$ is called a local quasicoordinate of $M$.
\end{definition}
Let $x\in D$ and $(V=V(x);z_1,\ldots,z_n)$ be an analytic chart of $X$ centered at $x$,  by means of which $V$ gets identified with $\Delta_{1}^{r}\times\Delta_{1}^{s}$ and $V\setminus D$ with $\Delta_{1}^{\ast r}\times\Delta_{1}^{s}$ ($r=r(x)$). For every $\eta\in (0,1)^{r}$, define the quasicoordinate
\begin{displaymath}
	\begin{split}
		\phi_{\eta}:V_{\eta}=(\Delta_{3/4})^{r}\times\Delta_{1}^{s}&\longrightarrow V\setminus D\\
		v=(v_1,\ldots,v_n)&\longmapsto (\phi_{\eta,1}(v),\ldots,\phi_{\eta,n}(v))
	\end{split}
\end{displaymath}
where
\begin{displaymath}
	\phi_{\eta,k}(v)=
		\begin{cases}
			\exp\left(\frac{1+\eta_k}{1-\eta_k}\frac{v_k+1}{v_k-1}\right)&\text{ if } 1\leq k\leq r,\\
			v_k &\text{ if } k>r.
		\end{cases}
\end{displaymath}
Observe that 
\begin{displaymath}
	V=\bigcup_{\eta\in (0,1)^{r}}\phi_{\eta}(V_{\eta}).
\end{displaymath}
We now construct a quasicoordinate covering $\mathcal{V}=\lbrace{(V_{\beta},\phi_{\beta})\rbrace}_{\beta}$, containing exactly:
\begin{itemize}
	\item all the quasicoordinates $\lbrace(V_{\eta},\phi_{\eta})\rbrace_{\eta\in(0,1)^{r}}$, for $(V=V(x);z_1,\ldots,z_n)$, $x\in D$, as above. Denote by $W$ the union of the images of all these quasicoordinates. This is an open neighborhood of $D$;
	\item a finite coordinate covering of the compact subset $X\setminus W$.
\end{itemize}
We introduce the H\"{o}lder space of $\CC^{k,\alpha}$-functions on $U=X\setminus D$, with respect to the quasicoordinate covering $\mathcal{V}$.
\begin{definition}[H\"{o}lder spaces]
Let $k\geq 0$ be an integer and $\alpha\in(0,1)$. The $\CC^{k,\alpha}$-norm with respect to $\mathcal{V}$ of a function $u\in\CC^{k}(U)$ is
\begin{displaymath}
		\begin{split}
			\\|u\|_{k,\alpha,\mathcal{V}}=&\sup_{(V,\phi)\in\mathcal{V}}\|\phi^{*}(u)\|_{k,\alpha}\\
			=&\sup_{(V,\phi)\in\mathcal{V}}\Bigg\lbrace
			 \sup_{v\in V}\left(\sum_{|p|+|q|\leq k}\left|\frac{\pd^{|p|+|q|}}{\pd v^{p}\pd\overline{v}^{q}}
			\phi^{*}(u)(v)\right|\right)\\
			&\hspace{1.3cm}+\sup_{v,v'\in V}\Bigg(\sum_{|p|+|q|=k}|v-v'|^{-\alpha}\Bigg|\frac{\pd^{|p|+|q|}}{\pd v^{p}\pd\overline{v}^{q}}\phi^{*}(u)(v)\\
			&\hspace{5.5cm}-\frac{\pd^{|p|+|q|}}{\pd v^{p}\pd\overline{v}^{q}}\phi^{*}(u)(v')\Bigg|\Bigg)\Bigg\rbrace.
		\end{split}
\end{displaymath}
We define the space of $\CC^{k,\alpha}$-functions on $U$ as
\begin{displaymath}
	\CC^{k,\alpha}(U)=\lbrace u\in\CC^{k}(U)\mid \|u\|_{k,\alpha,\mathcal{V}}<+\infty\rbrace,
\end{displaymath}
which is seen to be a Banach space with respect to the norm $\|\cdot\|_{k,\alpha,\mathcal{V}}$.
\end{definition}
\begin{definition}
We define $R^{r,s}(U)$ as the space of $(r,s)$-differential forms $\omega$ on $U$ such that, for all quasicoordinate $(V,\phi)\in\mathcal{V}$,
\begin{displaymath}
	\phi^{*}(\omega)=\sum_{\substack{|p|=r\\|q|=s}}(a_{p}\dd v^{p}+b_{q}\dd\overline{v}^{q}),
\end{displaymath}
with
\begin{displaymath}
	\|a_{p}\|_{k,\alpha},\|b_{q}\|_{k,\alpha}<+\infty
\end{displaymath}
for all multi-indices $p$, $q$ with $|p|=r$, $|q|=s$ and all $k\geq 0$, $\alpha\in (0,1)$. If $(v_1,\ldots,v_n)$ are the standard coordinates on $V\subset\C^{n}$ and $p=(i_1,\ldots,i_r)$, $q=(j_1,\ldots, j_s)$ are ordered multi-indices, we wrote
\begin{displaymath}
	\dd v^{p}=\dd v_{i_1}\wedge\ldots\wedge\dd v_{i_r},\, \dd\overline{v}^{q}=\dd\overline{v}_{j_1}
	\wedge\ldots\wedge\dd\overline{v}_{j_s}.
\end{displaymath}
\begin{lemma}\label{lem_prop_Rspace}
i. $R^{r,s}(U)$ is a $\C$-vector space;\\
ii. $R^{r,s}(U)\wedge R^{r',s'}(U)\subseteq R^{r+r',s+s'}(U)$;\\
iii. $\pd R^{r,s}(U)\subseteq R^{r+1,s}(U)$ and $\cpd R^{r,s}(U)\subseteq R^{r,s+1}(U)$.
\end{lemma}
\begin{proof}
Immediate to check from the definition of $R^{r,s}(U)$.
\end{proof}
\end{definition}
Recall that $\mathcal{P}_{D}$ denotes the sheaf of Poincar\'e forms on $X$ with singularities along $D$ (Definition \ref{def_Poincare_form}).
\begin{lemma}\label{lem_Rspace_Poincare}
We have an inclusion $R^{r,s}(U)\subseteq\Gamma(X,\mathcal{P}_{D})^{(r,s)}$, where the superscript stands for the $(r,s)$ part with respect to the usual bigrading of complex differential forms.
\end{lemma}
\begin{proof}
We localize near the divisor $D$ and consider a quasicoordinate $\phi_{\eta}:V_{\eta}\rightarrow V\setminus D$ in $\mathcal{V}$. Hence $V$ has coordinates $z_1,\ldots,z_n$ and $V\setminus D$ is the divisor $z_1\ldots z_r=0$. For simplicity we consider the differential form
\begin{displaymath}
	\omega=h\frac{\dd z_1}{z_1\log(|z_1|^{-1})}\wedge\ldots\wedge\frac{\dd z_r}{z_r\log(|z_r|^{-1})}\wedge\dd z_{r+1}\wedge\ldots \dd z_{n}
\end{displaymath}
and suppose that $\phi_{\eta}^{*}(\omega)$ has finite $\CC^{k,\alpha}$-norm for all $k\geq 0$ and $\alpha\in (0,1)$. We have to prove that $h$ is bounded on the image of $\phi_{\eta}$. From the definition of $\phi_{\eta}$, a straightforward computation shows that
\begin{displaymath}
	\phi^{*}_{\eta}(\omega)=\phi^{*}(h)\prod_{i=1}^{r}\frac{-2}{1-|v_i|^{2}}\frac{|v_i-1|^{2}}{(v_i-1)^{2}}\dd v_1\wedge\ldots\wedge\dd v_n.
\end{displaymath}
The hypothesis implies the coefficient of $\dd v_1\wedge\ldots\wedge\dd v_n$ has bounded $\sup$-norm. Since $|v_i|\leq 3/4$ for $i=1,\ldots,r$, this immediately yields the boundedness of $\phi^{*}(h)$.
\end{proof}
\begin{corollary}\label{cor_Rfunc_Pgrowth}
Let $u:U=X\setminus D\rightarrow\C$ be a smooth function. If $u\in R^{0,0}(U)$, then $u$ is a P-singular function with singularities along $D$.
\end{corollary}
\begin{proof}
By Lemma \ref{lem_prop_Rspace}, $\dd u\in R^{1,0}(U)\oplus R^{0,1}(U)$ and $\ddc u\in R^{1,1}(U)$. By Lemma \ref{lem_Rspace_Poincare}, $du$ and $\ddc u$ have Poincar\'e growth with singularities along $D$. Hence $u$ is a P-singular function with singularities along $D$.
\end{proof}
\begin{theorem}[Kobayashi \cite{Kobayashi}]\label{thm_kobayashi_2}
Let $X$ be a compact complex analytic manifold and $D\subseteq X$ a divisor with simple normal crossings. Suppose that $\omega_{X}(D)$ is ample on $X$. Let $D_1,\ldots, D_m$ be the irreducible components of $D$ and $s_{i}\in\Gamma(X,\OO(D_i))$ sections with $\gdiv s_i=D_i$, for all $i=1,\ldots,m$. Let $\Omega$ be a volume form on $X$. There exist suitable smooth hermitian metrics on the line bundles $\OO(D_i)$, that we write $\|\cdot\|$ for simplicity, with $\|s_i\|$ small enough, and a function $u\in R^{0,0}(U)$ such that the volume form $\Omega_{KE}$ of the K\"{a}hler-Einstein metric on $U$ is
\begin{displaymath}
	\Omega_{KE}=e^{u}\frac{\Omega}{\prod_{i=1}^{m}\|s_i\|^{2}\log(\|s_i\|^{2})^{2}}.
\end{displaymath}
\end{theorem}
With Theorem \ref{thm_kobayashi_2} at hand, we can prove Proposition \ref{prop_good_KE}.
\begin{proof}[Proof of Proposition \ref{prop_good_KE}]
We may localize at an analytic chart $(V;z_1,\ldots,z_n)$ of $X$ by means of which $V$ gets identified with $\Delta_{1}^{r}\times\Delta_{1}^{s}$ and $V\setminus D=\Delta_{1}^{\ast r}\times\Delta_{1}^{s}$. For simplicity suppose that $D_i\cap V$ gets identified with $z_i=0$, for $i=1,\ldots, r$ and $D_i\cap V=\emptyset$ for $i>r$. A local analytic frame of $\omega_{X}(D)_{\mid V}$ is
\begin{displaymath}
	\frac{\dd z_1}{z_1},\ldots,\frac{\dd z_r}{z_r},\dd z_{r+1},\ldots,\dd z_{n}.
\end{displaymath}
Write $\|s_i\|^{2}=|z_i|^{2} h_{i}$ for $i=1,\ldots,r$. Then by Theorem \ref{thm_kobayashi_2} we can write
\begin{displaymath}
	\begin{split}
	\Omega_{KE\mid V}=
	e^{u}&\prod_{k=1}^{r}\frac{1}{h_k\log(\|s_k\|^{2})^2}\prod_{k>r}\frac{1}{\|s_k\|^2\log(\|s_k\|^{2})^{2}}\cdot\\
	\gamma&
	\prod_{k=1}^{r}\left(\frac{i}{2\pi}\frac{\dd z_k}{z_k}\wedge\frac{\cdz_k}{\cz_k}\right)\wedge
	\prod_{k>r}\left(\frac{i}{2\pi}\dd z_k\wedge\cdz_k\right)
	\end{split}
\end{displaymath}
where $\gamma$ is a smooth positive function. Observe also that the functions $h_k$ as well as the second product are smooth positive functions. We are thus reduced to prove that $u$ and $\log(\log(\|s_i\|^{2})^2)$ are P-singular functions with singularities along $D$. On one hand, Lemma \ref{cor_Rfunc_Pgrowth} proves that $u$ is P-singular. On the other hand, $\log(\log(\|s_i\|^{2})^{2})$ is P-singular by Lemma \ref{lemma_theta_N}.
\end{proof}
\subsubsection{Arithmetic theory}
Let $K$ be a number field and $X$ a nonsingular projective variety over $K$. Let $D\subseteq X$ be a reduced effective divisor such that $D(\C)\subseteq X(\C)$ has simple normal crossings. Suppose that $\omega_{X}(D)$ is an ample divisor on $X$. Then, for every $\sigma:K\hookrightarrow\C$, $\omega_{X_{\sigma,\C}}(D_{\sigma,\C})$ is ample and there exists a unique K\"ahler-Einstein metric on $X_{\sigma}(\C)\setminus D_{\sigma}(\C)$ with constant negative Ricci curvature -1 (see Theorem \ref{thm_kobayashi}). By Proposition \ref{prop_good_KE}, these metrics induce good hermitian metrics on the lines $\omega_{X_{\sigma,\C}}(D_{\sigma,\C})$, with singularities along $D_{\sigma}(\C)$, respectively. The collection of these metrics, for varying $\sigma:K\hookrightarrow\C$, is invariant under the action of complex conjugation $F_{\infty}$. Indeed, let $g_{KE,\sigma}$ be the K\"ahler-Einstein metric on $X_{\sigma}(\C)\setminus D_{\sigma}(\C)$. Then $F_{\infty}^{*}(g_{KE,\sigma})$ defines a complete K\"ahler metric on $X_{\overline{\sigma}}(\C)\setminus D_{\overline{\sigma}}(\C)$. Let $\Omega_{KE,\sigma}$ be the volume form of $g_{KE,\sigma}$, so that $F_{\infty}^{*}\Omega_{KE,\sigma}$ is the volume form of $F_{\infty}^{*}(g_{KE,\sigma})$. Since $\Ric F_{\infty}^{*}\Omega_{KE,\sigma}=F_{\infty}^{*}\Ric\Omega_{KE,\sigma}$, we find 
\begin{displaymath}
	J(F_{\infty}^{*}\Omega_{KE_\sigma})=F_{\infty}^{*}J(\Omega_{KE,\sigma})=1.
\end{displaymath}
By uniqueness we derive $F_{\infty}^{*}(g_{KE,\sigma})=g_{KE,\overline{\sigma}}$. We write $\overline{\omega_{X}(D)}_{KE}$ for the resulting good hermitian line bundle, with singularities along $D$.\\
Let now $(\X,\HAL)$ be any model of $(X,\overline{\omega_{X}(D)}_{KE})$ over $\BS=\Spec\OO_K$. Then Theorem \ref{main_theorem} can be applied for any choice of smooth metric $\|\cdot\|_0$ on $\AL$. If $\AL$ is ample over $\X$, then Corollary \ref{Corollary_Main} applies to $(\X,\HAL)$. In this situation $\HAL$ verifies the finiteness and the universal lower bound properties.
\section{Appendix}
The appendix is aimed to prove Proposition \ref{prop_Bertini}.
\begin{proof}[Proof of Proposition \ref{prop_Bertini}]
Decompose $D$ into smooth irreducible components, $D=D_{1}\cup\ldots\cup D_{m}$. Let $D^{*}$ be the Weil divisor $D_{1}+\ldots+D_{m}$. Denote by $\mathcal{F}$ the family of nonsingular subschemes of $X$ consisting of $X$ itself and all the irreducible components of the intersections
\begin{displaymath}
	\bigcap_{i\in I}D_{i},
\end{displaymath}
where $I$ runs over all the subsets of $\lbrace 1,\ldots,m\rbrace$. Since $L$ is ample, there exists some positive integer $N$ such that $L^{\otimes N}$ and $L^{\otimes N}(-D^{*})$ are very ample. Consider the exact sequence
\begin{displaymath}
	0\rightarrow L^{\otimes N}(-D^{*})\rightarrow L^{\otimes N}\rightarrow L_{D}^{\otimes N}\rightarrow 0.
\end{displaymath}
Taking global sections, we find the exact sequence
\begin{equation}\label{eq:ExSeq}
	0\rightarrow\Gamma(X,L^{\otimes N}(-D^{*}))
	\overset{j}{\rightarrow}\Gamma(X, L^{\otimes N})
	\overset{p}{\rightarrow}\Gamma(X,L_{D}^{\otimes N}).
\end{equation}
For every $Y\in\mathcal{F}$, Bertini's theorem (\cite{Hartshorne}, Chapter II, Theorem 8.18) provides us with a dense open subset $U_{Y}$ of the projective space $\PP=\PP(\Gamma(X,L^{\otimes N}(-D^{*})))$ such that, for any $t\in U_{Y}$, $\supp(\gdiv t)$ intersects $Y$ transversally. Since $\mathcal{F}$ is finite and the open subsets $U_{Y}$ are dense, the intersection
\begin{displaymath}
	U=\bigcap_{Y\in\mathcal{F}}U_{Y}
\end{displaymath}
is a non-empty open subset of $\PP$. Therefore we can take $\overline{t_{1}},\ldots,\overline{t_{r}}\in U$ such that $\supp(\gdiv \overline{t_{1}})$ $\cap$ $\ldots$ $\cap$ $\supp(\gdiv\overline{t_{r}})$=$\emptyset$. Let $t_{1},\ldots, t_{r}\in\Gamma(X,L^{\otimes N}(-D^{*}))$ be representatives of $\overline{t_{1}},\ldots,\overline{t_{r}}$, respectively. Let $s_{1}=j(t_{1}),\ldots,s_{r}=j(t_{r})$ be their images by the morphism $j$ of (\ref{eq:ExSeq}). Since $p(s_{i})=0$, $D\subseteq\supp(\gdiv s_{i})$ for all $i=1,\ldots,r$. Actually, for every $i=1,\ldots,r$, we have
\begin{displaymath}
	\gdiv s_{i}=\gdiv t_{i} +D^{*}.
\end{displaymath}
Therefore, for the support of $\gdiv s_{i}$ we find
\begin{displaymath}
	\supp(\gdiv s_{i})=\supp(\gdiv t_{i})\cup D.
\end{displaymath}
By the choice of the sections $t_{i}$ ($\overline{t_{i}}\in U$), $\supp(\gdiv s_{i})$ is a divisor with simple normal crossings. Finally, since $\supp(\gdiv t_1)\cap\ldots\cap\supp(\gdiv t_r)=\emptyset$, we have an equality of reduced closed subschemes of $X$
\begin{displaymath}
		D=(\supp(\gdiv s_{1})\cap\ldots\cap\supp(\gdiv s_{r}))_{\red}.
\end{displaymath}
The proof of the proposition is complete.
\end{proof}
	
\bibliographystyle{amsplain}
\providecommand{\bysame}{\leavevmode\hbox to3em{\hrulefill}\thinspace}

\textsc{G. Freixas i Montplet, D\'epartement de Math\'ematiques, Universit\'e Paris-Sud, B\^atiment 425, 91405 Orsay cedex, France}\\
\\
\textit{E-mail address}: \textbf{gerard.freixas@math.u-psud.fr}

\end{document}